\documentclass[12p]{article}
\RequirePackage{amsthm,amsmath,amsfonts,amssymb}
\usepackage{xcolor}
\usepackage[left=2.5cm,right=2.5cm,top=2.5cm,bottom=2cm]{geometry}
\usepackage{graphicx}
\usepackage[authoryear, round]{natbib}
\usepackage{tikz}
\usetikzlibrary{angles,quotes}

\newtheorem{theorem}{Theorem}
\newtheorem{lemma}[theorem]{Lemma}
\newtheorem{remark}[theorem]{Remark}
\newtheorem{proposition}[theorem]{Proposition}
\newtheorem{corollary}[theorem]{Corollary}

\newcommand{\mb}{\mathbf}

\newcommand{\cF}{\mathcal{F}}
\newcommand{\rmi}{\mathrm{i}}
\DeclareMathOperator{\Sin}{\mathtt{Sin}}

\begin{document}

\title{Precise asymptotics for the norm of large random rectangular Toeplitz matrices}
\author{Alexei Onatski}
\date{September 4, 2025}
\maketitle
\begin{abstract}
   We study the spectral norm of large rectangular random Toeplitz and circulant matrices with independent entries. For Toeplitz matrices, we show that the scaled norm converges to the norm of a bilinear operator defined via the pointwise product of two scaled sine kernel operators. In the square case, this limit reduces to the squared $2\!\to\!4$ norm of the sine kernel operator, in agreement with the result of \cite{sen13}. For $p\times n$ circulant matrices, we show that their norm divided by $\sqrt{p\log n}$ converges in probability to 1. We further investigate the finite-sample performance of these asymptotic results via Monte Carlo experiments, which reveal both non-negligible bias and dispersion. For circulant matrices, a higher-order asymptotic analysis in the Gaussian case explains these effects, connects the fluctuations to shifted Gumbel distributions, and suggests a natural conjecture on the limiting law.
\end{abstract}

\section{Introduction and main results} \label{sec: intro}
In a seminal study, \cite{sen13} established precise asymptotics for the spectral norm of large symmetric random Toeplitz matrices. They showed that, after appropriate normalization, the spectral norm converges to a constant that can be expressed as the square of an operator norm associated with the classical sine kernel.

In this paper, we extend their analysis to rectangular, non-symmetric \( p \times n \) random Toeplitz matrices. We show that, under analogous scaling, the spectral norm converges to the norm of a bilinear operator defined via the pointwise product of two scaled sine kernel operators. In the special case when the matrix is square (\( p = n \)), our result recovers the limiting constant identified by \cite{sen13}.

Large random \( p \times n \) Toeplitz matrices \( \mathbf{T}_{p \times n} \) arise in a variety of applied contexts. In compressed sensing, for instance, they are used to reconstruct a sparse \( n \)-dimensional signal \( x \) from a lower-dimensional measurement \( \mathbf{T}_{p \times n}x \); see \cite{rauhut09} and references therein. They also play a central role in subspace-based methods of signal processing \cite{nekrutkin10}, and naturally appear in time series analysis, where \( \mathbf{T}_{p \times n} \) serves as the design matrix in autoregressive models of order \( p \) constructed from \( n \) observations of a stochastic process.

Following \cite{sen13} (abbreviated SV13 hereafter), we assume the entries along the diagonals of \( \mathbf{T}_{p \times n} \) are drawn from a sequence \( \{a_i\}_{i\in\mathbb{Z}} \) of independent real-valued random variables satisfying
\begin{equation}\label{Lyapunov}
\mathbb{E}a_i = 0, \qquad \operatorname{Var}(a_i) = 1, \qquad \mathbb{E}|a_i|^\gamma < A
\end{equation}
for some constants \( \gamma > 2 \) and \( A > 0 \). The matrix itself is defined by
\begin{equation}\label{definition of T}
\left( \mathbf{T}_{p \times n} \right)_{ij} = a_{i-j}, \quad \text{for } i = 0, \dots, p - 1;\ j = 0, \dots, n - 1,
\end{equation}
so that, unlike in SV13, it is non-symmetric. We allow the number of rows \( p = p(n) \) to vary with \( n \), subject to \( p \leq n \), and assume that for some \( \alpha \in (0,1) \),
\begin{equation}\label{new asymptotic regime}
\limsup_{n \rightarrow \infty} \left( \frac{n}{p \log^\alpha n} \right) = 0.
\end{equation}
This condition permits the ratio \( n/p \) to fluctuate and even diverge to infinity along subsequences, provided it does not do so too rapidly.

To describe the asymptotic behavior of $\|\mb{T}_{p\times n}\|$, it is convenient to introduce certain integral operators built from the sine kernel. For \( u > 0 \), let \( \Sin_u \) denote the scaled sine kernel operator (the unscaled case corresponds to \( u=1 \)):
\[
\Sin_u(f)(x) := \int_{\mathbb{R}} \frac{\sin(u \pi (x - y))}{\pi (x - y)} f(y) \, dy, \quad f \in L^2(\mathbb{R}).
\]
Given two such operators, their pointwise product leads to the bilinear operator \( \Sin_{u,w} \), defined by
\[
\Sin_{u,w}(f,g)(x) := \Sin_u(f)(x) \cdot \Sin_w(g)(x), \quad f,g \in L^2(\mathbb{R}),
\]
with associated operator norm
\[
\|\Sin_{u,w}\|_{2 \rightarrow 2} := \sup_{\|f\|_2 \leq 1, \|g\|_2 \leq 1} \| \Sin_{u,w}(f,g) \|_2.
\]
A proof of the following theorem is given in Section~\ref{sec: proof of theorem 1}.

\begin{theorem} \label{T theorem} 
Let $\mathbf{T}_{p\times n}$ be a sequence of non-symmetric rectangular random Toeplitz matrices defined in \eqref{definition of T} and satisfying assumptions \eqref{Lyapunov} and \eqref{new asymptotic regime}. Then, as $n \to \infty$, the spectral norm satisfies
\begin{equation}
    \label{main convergence}
    \frac{\|\mathbf{T}_{p\times n}\|}{\sqrt{p\log n}} - \|\Sin_{1,n/p}\|_{2\rightarrow 2} \overset{P}{\longrightarrow} 0.
\end{equation}
If, in addition, 
\begin{equation}\label{stronger asymptotic regime}
    \limsup_{n\to\infty} n/p <\infty,
\end{equation}
then the convergence in probability in \eqref{main convergence} strengthens to convergence in $L^\gamma$.
\end{theorem}
 Note that the quantity $\|\Sin_{1,n/p}\|_{2 \to 2}$ depends on the possibly fluctuating ratio $n/p$, so Theorem \ref{T theorem} expresses the closeness of the scaled spectral norm to this bilinear norm rather than convergence to a fixed constant.

 To provide immediate intuition, Table~\ref{tab: table_intro} reports numerical estimates of $K_{1,n/p} = \|\Sin_{1,n/p}\|_{2 \to 2}$ for selected ratios $p/n \in (0,1]$, obtained from a variational characterization discussed in detail in Section~\ref{sec: variational characterization}. The computational algorithm and its accuracy are described in Section~\ref{sec: K values}. The table shows that for “wide” matrices (small $p/n$), the limiting scaled norm is close to 1, and it decreases moderately as $p/n$ increases.
 
 \begin{table}[!b]
\centering
\begin{tabular}{c c c c c c c c c c c c c c}
\hline
$p/n$ & $K_{1,n/p}$ & & $p/n$ & $K_{1,n/p}$ & & $p/n$ & $K_{1,n/p}$ & & $p/n$ & $K_{1,n/p}$ & & $p/n$ & $K_{1,n/p}$ \\
\hline
0.01 & 1.000 & & 0.21 & 0.985 & & 0.41 & 0.952 & & 0.61 & 0.912 & & 0.81 & 0.869 \\
0.02 & 1.000 & & 0.22 & 0.984 & & 0.42 & 0.951 & & 0.62 & 0.910 & & 0.82 & 0.867 \\
0.03 & 1.000 & & 0.23 & 0.982 & & 0.43 & 0.949 & & 0.63 & 0.908 & & 0.83 & 0.865 \\
0.04 & 0.999 & & 0.24 & 0.981 & & 0.44 & 0.947 & & 0.64 & 0.905 & & 0.84 & 0.863 \\
0.05 & 0.999 & & 0.25 & 0.980 & & 0.45 & 0.945 & & 0.65 & 0.903 & & 0.85 & 0.860 \\
0.06 & 0.999 & & 0.26 & 0.978 & & 0.46 & 0.943 & & 0.66 & 0.901 & & 0.86 & 0.858 \\
0.07 & 0.998 & & 0.27 & 0.977 & & 0.47 & 0.941 & & 0.67 & 0.899 & & 0.87 & 0.856 \\
0.08 & 0.998 & & 0.28 & 0.975 & & 0.48 & 0.939 & & 0.68 & 0.897 & & 0.88 & 0.854 \\
0.09 & 0.997 & & 0.29 & 0.973 & & 0.49 & 0.937 & & 0.69 & 0.895 & & 0.89 & 0.852 \\
0.10 & 0.996 & & 0.30 & 0.972 & & 0.50 & 0.935 & & 0.70 & 0.893 & & 0.90 & 0.850 \\
0.11 & 0.996 & & 0.31 & 0.970 & & 0.51 & 0.933 & & 0.71 & 0.890 & & 0.91 & 0.848 \\
0.12 & 0.995 & & 0.32 & 0.969 & & 0.52 & 0.931 & & 0.72 & 0.888 & & 0.92 & 0.846 \\
0.13 & 0.994 & & 0.33 & 0.967 & & 0.53 & 0.929 & & 0.73 & 0.886 & & 0.93 & 0.843 \\
0.14 & 0.993 & & 0.34 & 0.965 & & 0.54 & 0.927 & & 0.74 & 0.884 & & 0.94 & 0.841 \\
0.15 & 0.992 & & 0.35 & 0.963 & & 0.55 & 0.925 & & 0.75 & 0.882 & & 0.95 & 0.839 \\
0.16 & 0.991 & & 0.36 & 0.962 & & 0.56 & 0.922 & & 0.76 & 0.880 & & 0.96 & 0.837 \\
0.17 & 0.990 & & 0.37 & 0.960 & & 0.57 & 0.920 & & 0.77 & 0.878 & & 0.97 & 0.835 \\
0.18 & 0.989 & & 0.38 & 0.958 & & 0.58 & 0.918 & & 0.78 & 0.875 & & 0.98 & 0.833 \\
0.19 & 0.988 & & 0.39 & 0.956 & & 0.59 & 0.916 & & 0.79 & 0.873 & & 0.99 & 0.831 \\
0.20 & 0.986 & & 0.40 & 0.954 & & 0.60 & 0.914 & & 0.80 & 0.871 & & 1.00 & 0.829 \\
\hline
\end{tabular}
\caption{Estimated values of $K_{1,n/p}=\|\Sin_{1,n/p}\|_{2\rightarrow 2}$ for $p/n$ on the grid $0.01:0.01:1$.}
\label{tab: table_intro}
\end{table}

The appearance of the bilinear operator norm in the convergence statement \eqref{main convergence} can be viewed as a natural generalization of the constant
\[
K_1 := \|\Sin\|_{2 \to 4}^2,
\]
which appears as the limit in SV13’s result for the square symmetric case, where \(\Sin := \Sin_1\) is the standard sine operator. Indeed, when \(p=n\), the Cauchy–Schwarz inequality gives
\[
\|\Sin_{1,1}(f,g)\|_2 = \|\Sin(f) \cdot \Sin(g)\|_2 \leq \|\Sin(f)\|_4 \|\Sin(g)\|_4,
\]
which immediately yields
\[
\|\Sin_{1,1}\|_{2 \to 2} = \|\Sin\|_{2 \to 4}^2 = K_1.
\]

Most of the previous research on the asymptotics of the norm of random Toeplitz matrices, including SV13, \cite{meckes07}, \cite{bose07}, and \cite{adamczak10}, focuses on the square case ($p=n$). The work \cite{adamczak13} addresses the rectangular case and provides a finite-sample bound on the expected spectral norm of $\mathbf{T}_{p \times n}$ (see their eq.~(1)):
\begin{equation}\label{eq: adamchak}
\mathbb{E}\|\mathbf{T}_{p\times n}\|\leq C_\gamma A^{1/\gamma}(\sqrt{n}+\sqrt{p\log p}).
\end{equation}
Here $C_\gamma$ is a positive constant depending on $\gamma>2$, and $A$ is the upper bound on $\mathbb{E}|a_i|^\gamma$ introduced in \eqref{Lyapunov}. Bound \eqref{eq: adamchak} implies, via the Markov inequality, the divergence rate of $\|\mathbf{T}_{p\times n}\|$. Our divergence rate, $\sqrt{p\log n}$, is consistent with this because, under the asymptotic regime \eqref{new asymptotic regime}, 
\[
\sqrt{p\log n}=O(\sqrt{p\log p}).
\]
Theorem \ref{T theorem} complements the rate bound by providing the precise asymptotic limit of the scaled $\|\mathbf{T}_{p\times n}\|$.

Another result of \cite{adamczak13} can be stated as follows (see their eq.~(2)). For any $\delta,\epsilon\in (0,1)$, there exists $C>0$ such that if $n>Cp\log p$, then with probability at least $1-\delta$, for all $x\in\mathbb{R}^n$,
\begin{equation}\label{eq: adamchak 1}
(1-\epsilon)\|x\|\leq\left\|\frac{1}{\sqrt{n}}\mathbf{T}_{p\times n}x\right\|\leq (1+\epsilon)\|x\|.
\end{equation}
In other words, $\mathbf{T}_{p\times n}/\sqrt{n}$ acts as an almost isometry with high probability. In particular, when $n>Cp\log p$, the logarithmic factor in the divergence rate disappears. 

Theorem \ref{T theorem} complements this observation by showing that the logarithmic factor is still present as long as $n\leq Cp\log^\beta n$ for some $\beta\in(0,1)$. Indeed, the asymptotic regime \eqref{new asymptotic regime} is then satisfied with $\alpha=(1+\beta)/2$, say. Since $\log^\beta n\geq \log^\beta p$, this indicates that the transition from the divergence  rate with a logarithmic factor to the one without such a factor occurs within the window
\[
\log^\beta p\ll\frac{n}{p}\lesssim \log p,
\]
where $\beta\in(0,1)$ can be taken arbitrarily close to $1$.

Having established the Toeplitz case, it is natural to ask whether analogous behavior holds for other patterned random matrices (see \cite{bose18}). In Section~\ref{sec: proof of the other theorems}, the proof of Theorem \ref{T theorem} is adapted to establish an analogue for rectangular circulant (a.k.a.~partial circulant) matrices \( \mathbf{C}_{p \times n} \), defined by
\begin{equation}\label{definition of C}
\left( \mathbf{C}_{p \times n} \right)_{ij} = a_{(j - i)\bmod n}, \quad \text{for } i = 0, \dots, p - 1;\ j = 0, \dots, n - 1.
\end{equation}

\begin{theorem}
    \label{C theorem}
    Let $\mathbf{C}_{p\times n}$ be a sequence of non-symmetric rectangular random circulant matrices defined in \eqref{definition of C} and satisfying assumptions \eqref{Lyapunov} and \eqref{new asymptotic regime}. Then, as $n \to \infty$, the spectral norm satisfies
\begin{equation}
    \label{main convergence C}
    \frac{\|\mathbf{C}_{p\times n}\|}{\sqrt{p\log n}} \overset{P}{\longrightarrow} 1.
\end{equation}
If, in addition, \eqref{stronger asymptotic regime} holds, then the convergence in probability in \eqref{main convergence C} strengthens to convergence in $L^\gamma$.
\end{theorem}

There has been substantial research on the extreme singular values of random \emph{square} circulant matrices (see, e.g., \cite{barrera22} and the references therein). These results, however, are specific to the case $p=n$ and do not directly extend to rectangular dimensions. To the best of our knowledge, Theorem~\ref{C theorem} provides the first asymptotic characterization for the rectangular setting $p<n$. Given the explicit and remarkably simple form of the limit in~\eqref{main convergence C}, it is instructive to comment on the source of the discrepancy with the more involved result of Theorem~\ref{T theorem}.

The proofs of both Theorems~\ref{T theorem} and~\ref{C theorem}, as well as that of SV13, share a common starting point: embedding the \( p \times n \) matrix of interest into a larger \( N \times N \) circulant matrix \( \mb{A} \), such that the original matrix occupies its upper-left corner. Letting \( \mb{I}_{q,N} \) denote the \( N \times N \) diagonal matrix with the first \( q \) diagonal entries equal to one and the rest zero, the spectral norm of the original matrix coincides with that of \( \mb{I}_{p,N}\mb{A}\mb{I}_{n,N} \).

Recall that any circulant matrix \( \mb{A} \) admits the spectral decomposition (see, e.g.,~\cite{golub96}, p.~202):
\begin{equation}\label{spectral circulant}
\mb{A} = \sqrt{N} \, \cF_N \operatorname{diag}(\cF_N \mb{a}) \overline{\cF_N}^\top =: \sqrt{N} \, \cF_N \mb{D}(\mb{a}) \overline{\cF_N}^\top,
\end{equation}
where \( \mb{a}^\top \) is the first row of \( \mb{A} \), and \( \cF_N \) is the unitary Discrete Fourier Transform matrix with entries
\[
(\cF_N)_{st} = \frac{1}{\sqrt{N}} \exp\left\{ \frac{2\pi\rmi}{N}st \right\}, \qquad s,t = 0, \dots, N - 1.
\]
Using this decomposition, the norm of \( \mb{I}_{p,N} \mb{A} \mb{I}_{n,N} \) reduces to the norm of the matrix
\begin{equation}
\label{PDP matrix}
\sqrt{N}\mb{P}_{p,N} \mb{D}(\mb{a}) \mb{P}_{n,N},
\end{equation}
where \( \mb{P}_{q,N} := \cF_N \mb{I}_{q,N} \overline{\cF_N}^\top \) is the orthogonal projection onto the subspace spanned by the first \( q \) columns of \( \cF_N \).

The key difference between the Toeplitz and circulant settings arises in how these matrices are embedded. Namely, in the Toeplitz case, it is not possible to embed the matrix \( \mathbf{T}_{p \times n} \) into an \( N \times N \) circulant unless \( N > \max\{p, n\} \), which necessarily makes both \( \mathbf{P}_{p,N} \) and \( \mathbf{P}_{n,N} \) non-trivial projections. 
In contrast, a rectangular circulant \( \mb{C}_{p\times n} \) naturally embeds into an \( n \times n \) circulant, so we may take \( N = n \), in which case \( \mb{P}_{n,N} \) becomes the identity matrix. This structural simplification ultimately leads to the much cleaner limiting behavior described in Theorem~\ref{C theorem}.

In Section~\ref{sec: proof of the other theorems}, we also show that the arguments used to prove Theorems~\ref{T theorem} and~\ref{C theorem} extend naturally to the setting of \emph{symmetric} rectangular matrices. Here, symmetry is understood to mean that matrix entries with indices $(i,j)$ and $(j,i)$ coincide whenever both are defined. The elements of the symmetric Toeplitz and circulant matrices we consider are given by
\begin{align}
    \label{elements of Ts}
    \left(\mathbf{T}_{p \times n}^{(s)}\right)_{ij} &= a_{|i - j|}, \\
    \label{elements of Cs}
    \left(\mathbf{C}_{p \times n}^{(s)}\right)_{ij} &= a_{n/2 - |n/2 - |j - i||},
\end{align}
for \( i = 0, \dots, p - 1 \) and \( j = 0, \dots, n - 1 \). Note that the definition of \( \mathbf{C}_{p \times n}^{(s)} \) is valid for both odd and even values of \( n \).

\begin{theorem} \label{S theorem} Let $\mb{T}_{p\times n}^{(s)}$ and $\mb{C}_{p\times n}^{(s)}$ be sequences of rectangular symmetric random Toeplitz and circulant matrices defined in \eqref{elements of Ts} and \eqref{elements of Cs} and satisfying assumptions \eqref{Lyapunov} and  \eqref{new asymptotic regime}. Then, as $n$ diverges to infinity, the spectral norms $\|\mathbf{T}_{p\times n}^{(s)}\|$ and $\|\mathbf{C}_{p\times n}^{(s)}\|$ satisfy
\begin{equation}
    \label{main convergence S}
    \frac{\|\mathbf{T}_{p\times n}^{(s)}\|}{\sqrt{2p\log n}}-\|\Sin_{1,n/p}\|_{2\rightarrow 2}\overset{P}\rightarrow 0,\qquad \frac{\|\mathbf{C}_{p\times n}^{(s)}\|}{\sqrt{2p\log n}}\overset{P}\rightarrow 1.
\end{equation}
If in addition \eqref{stronger asymptotic regime} holds, then the latter two convergences in probability strengthen to convergence in $L^\gamma$.
\end{theorem}

The only difference with the corresponding results for non-symmetric Toeplitz and circulant matrices is that the norms are scaled by $\sqrt{2p\log n}$ as opposed to $\sqrt{p\log n}$. The example of a square $n\times n$ circulant matrix with Gaussian entries provides useful intuition for understanding the source of this difference. The spectral decomposition \eqref{spectral circulant} implies that all but a few eigenvalues of such matrices, scaled by $1/\sqrt{n}$, are real standard Gaussian random variables in the symmetric case, and complex standard Gaussian random variables in the non-symmetric case. The key distinction then comes from extreme value behavior: the largest absolute value among \( n \) real standard Gaussians grows like \( \sqrt{2\log n} \), while in the complex case, it grows more slowly, like \( \sqrt{\log n} \). For a more detailed discussion, see Section~\ref{sec: proof of the other theorems}.

It is worth noting that the spectral norm of a matrix is invariant under arbitrary permutations of its rows or columns. As a consequence, the conclusions of the above theorems remain unchanged if the order of the columns of the corresponding matrices is reversed; that is, if the \( j \)-th column, \( j = 0, \dots, n - 1 \), is mapped to the \( r(j) := n - 1 - j \)-th position.

In particular, Theorem~\ref{T theorem} continues to hold if the rectangular Toeplitz matrix \( \mb{T}_{p \times n} \) is replaced by the rectangular Hankel matrix \( \mb{H}_{p \times n} \) with entries \( \mb{H}_{ij} = a_{r(j) - i} \). Similarly, Theorem~\ref{C theorem} remains valid if the rectangular circulant matrix \( \mb{C}_{p \times n} \) is replaced by the rectangular reverse circulant matrix \( \mb{R}_{p \times n} \) with entries \( \mb{R}_{ij} = a_{(r(j) - i) \bmod n} \). Finally, Theorem~\ref{S theorem} holds with the symmetric Toeplitz and circulant matrices replaced by, respectively, the symmetric Hankel matrix with entries \( (\mb{H}^{(s)})_{ij} = a_{|r(j) - i|} \) and the symmetric reverse circulant matrix with entries \( (\mb{R}^{(s)})_{ij} = a_{n/2 - |n/2 - |r(j) - i||} \).

Before turning to the details of our setting and proofs, we caution that the convergence in the above theorems is rather slow. Monte Carlo experiments in Section~\ref{sec: MC} show that for $p=500$ and corresponding values of $n$, the theoretical limits often underestimate the observed spectral norms. Moreover, the empirical dispersion remains substantial.  

To gain theoretical insight into these discrepancies, Section~\ref{sec: MC} analyzes the second-order behavior of $\|\mathbf{C}_{p\times n}\|$ in the Gaussian case $a_i \sim \mathcal{N}(0,1)$. The analysis reveals that the asymptotic distribution of the scaled and centered norm stochastically dominates a shifted Gumbel distribution, with the non-negative shift depending on the aspect ratio $p/n$. Because this theoretical lower bound aligns closely with the Monte Carlo distributions across a range of aspect ratios, we are led to conjecture that it may in fact describe the true limiting law. Establishing this rigorously, however, remains an open problem. 

The remainder of the paper is organized as follows. Section~\ref{sec: variational characterization} develops a variational characterization of the bilinear operator norm $\|\Sin_{u,w}\|_{2\rightarrow 2}$, which is central to our analysis. Sections~\ref{sec: proof of theorem 1} and~\ref{sec: proof of the other theorems} present the proofs of Theorems~\ref{T theorem}--\ref{S theorem}. Section~\ref{sec: K values} introduces the numerical algorithm used to compute the values of $K_{1,n/p}$ reported in Table~\ref{tab: table_intro} and assesses its accuracy. Section~\ref{sec: MC} provides the Monte Carlo analysis and derives a second-order asymptotic lower bound in the Gaussian circulant case. Additional technical details are collected in the Technical Appendix, while more routine derivations and arguments closely following SV13 are deferred to the Supplementary Material.

\section{Variational characterization of $\|\Sin_{u,w}\|_{2\rightarrow 2}$}\label{sec: variational characterization}

The proof of Theorem \ref{T theorem} relies on the following variational characterization of the norm of the bilinear sine operator. This result can be seen as a natural extension of SV13’s characterization (1) of $\|\Sin\|_{2 \to 4}$.
\begin{proposition}\label{prop:variational characterization}
For all $u,w>0$,
\begin{equation}\label{new positive definite problem}
\|\Sin_{u,w}\|_{2\rightarrow 2}=\sup_{\Lambda_u \in \mathcal{F}_u,\; \Lambda_w \in \mathcal{F}_w} \left(\int_{-\infty}^\infty \overline{\Lambda_u(x)}\Lambda_w(x)\,\mathrm{d}x\right)^{1/2}=:K_{u,w},
\end{equation}
where \( \mathcal{F}_c \) denotes the class of continuous, positive definite functions supported on \( [-c, c] \) and normalized so that \( \Lambda_c(0) = 1 \).\footnote{By Bochner's theorem (see p.~19 in \cite{rudin62}), this class coincides with the characteristic functions (i.e., Fourier transforms of probability measures) supported on \( [-c, c] \).}
\end{proposition}

\begin{proof}
    We begin by recalling some standard definitions for convolution and Fourier analysis (see also Appendix A.3 in SV13). For functions \( f, g : \mathbb{R} \to \mathbb{C} \), define their convolution by
\[
(f \star g)(x) := \int f(x - y)g(y)\, \mathrm{d}y,
\]
and define the involution \( f^\ast \) as \( f^\ast(x) := \overline{f(-x)} \). Let \( \widehat{f}(t) := \int e^{-2\pi i x t}f(x)\, \mathrm{d}x \) be the Fourier transform of \( f \in L^2(\mathbb{R}) \). Recall the identities
\[
\widehat{f \star g} = \widehat{f} \cdot \widehat{g} \quad \text{and} \quad \overline{\widehat{f}} = \widehat{f^\ast}.
\]

Let us denote the set of all complex-valued functions $\phi\in L^2(\mathbb{R})$ supported on $[-a/2,a/2]$ and normalized so that $\int|\phi(t)|^2\mathrm{d}t=1$ as $\mathcal{G}_{a/2}$. The following characterization of functions from $\mathcal{F}_a$ in terms of functions from $\phi_a\in\mathcal{G}_{a/2}$ is established in  Theorem 2.1 of \cite{garcia69} (see also Lemma 5.1 in \cite{boas45}).
\begin{lemma}[Theorem 2.1 of \cite{garcia69}]\label{lem: Boas and Kac}
$\Lambda_a\in\mathcal{F}_a$ if and only if there exists $\phi_a\in\mathcal{G}_{a/2}$, such that 
\[
\Lambda_a(x)=\int_{-\infty}^{\infty}\overline{\phi_a(t)}\phi_a(x+t)\mathrm{d}t=(\phi_a^\ast\star\phi_a)(x).
\]
\end{lemma}

Using this characterization, we reparametrize the supremum over \( \Lambda_u, \Lambda_w \) in the definition of $K_{u,w}$ in terms of square-integrable functions \( \phi_u, \phi_w \) supported in \( [-u/2, u/2] \) and \( [-w/2, w/2] \), respectively. This leads to
\begin{align*}
K_{u,w}^2
&= \sup_{\phi_u \in \mathcal{G}_{u/2},\; \phi_w \in \mathcal{G}_{w/2}}
\int \overline{(\phi_u^\ast \star \phi_u)(x)} \cdot (\phi_w^\ast \star \phi_w)(x)\, \mathrm{d}x \\
&= \sup_{f_1, f_2 \in L^2(\mathbb{R}) : \|f_1\|_2 \leq 1,\, \|f_2\|_2 \leq 1}
\int \overline{([f_1 \cdot \mathbf{1}_{u/2}]^\ast \star [f_1 \cdot \mathbf{1}_{u/2}])(x)} \cdot ([f_2 \cdot \mathbf{1}_{w/2}]^\ast \star [f_2 \cdot \mathbf{1}_{w/2}])(x)\, \mathrm{d}x,
\end{align*}
where  $\mathbf{1}_a$ is the indicator function of $[-a,a]$.

Note that the Fourier transform of $[f\cdot \mathbf{1}_{a}]^\ast\star[f\cdot \mathbf{1}_{a}]$ equals $|\widehat{f\cdot \mathbf{1}_{a}}|^2$. On the other hand, 
\[
\widehat{f\cdot \mathbf{1}_{a}}=\widehat{f}\star\widehat{\mathbf{1}_{a}}=\Sin_{a}(\widehat{f}).
\]
Applying Plancherel's theorem, we get
\begin{eqnarray*}
K_{u,w}^2&=&\sup_{\widehat{f_1},\widehat{f_2}\in L^2(\mathbb{R}):\|\widehat{f_1}\|_2\leq 1,\|\widehat{f_2}\|_2\leq 1 }\int |\Sin_{u}(\widehat{f_1})(x)|^2|\Sin_{w}(\widehat{f_2})(x)|^2\mathrm{d}x\\
&=&\sup_{\widehat{f_1},\widehat{f_2}\in L^2(\mathbb{R}):\|\widehat{f_1}\|_2\leq 1,\|\widehat{f_2}\|_2\leq 1 }\int |\Sin_{u,w}(\widehat{f_1},\widehat{f_2})(x)|^2\mathrm{d}x\\
&=&\|\Sin_{u,w}\|^2_{2\rightarrow 2},
\end{eqnarray*}
and taking square roots yields the claim. 
\end{proof}

A change of variables in the integral defining \( K_{u,w} \), together with a reparameterization of the support constraints, yields the identity
\begin{equation}\label{K reparameterization}
K_{Cu,Cw} =\sqrt{C}K_{u,w},\qquad\text{for any }C>0.
\end{equation}
In particular, $K_{1,n/p}=K_{p,n}/\sqrt{p}$, and therefore,
\[
\|\Sin_{1,n/p}\|_{2\rightarrow 2} =\frac{1}{\sqrt{p}}\, \|\Sin_{p,n}\|_{2\rightarrow 2}.
\]

Formulation~\eqref{new positive definite problem} is particularly well-suited for numerical approximation. It also connects naturally to earlier work of \cite{garcia69}, who studied the special case of~\eqref{new positive definite problem} under the constraint \( u = w \), and computed the value of \( K_{1,1}^2 = K_1^2 \) to 24 decimal places:
\[
K_{1,1}^2 = 0.686981293033114600949413\cdots.
\]
A related computational algorithm, which we use to compile Table \ref{tab: table_intro}, is discussed in Section~\ref{sec: K values}.

From an analytical perspective, the representation~\eqref{new positive definite problem} immediately yields a simple and explicit lower bound:
\begin{equation}\label{the lower bound}
K_{1,n/p} \geq \sqrt{1 - \frac{p}{3n}}.
\end{equation}
This estimate is obtained by selecting admissible test functions \( \Lambda_1(x) = \max\{1 - |x|, 0\} \) and \( \Lambda_{n/p}(x) = \max\left\{1 - \frac{p}{n}|x|, 0\right\} \), and directly evaluating the integral in~\eqref{new positive definite problem}. For \( n = p \), this gives \( K_{1,1}^2 \geq 2/3 \), which is quite close to the actual value \( 0.6869\cdots \).

Another analytical implication of~\eqref{new positive definite problem} concerns the asymptotic behavior of \( K_{1,n/p} \) when \( n/p \to \infty \) sufficiently slowly so that condition~\eqref{new asymptotic regime} holds. In this regime, \( K_{1,n/p} \to 1 \), making the lower bound~\eqref{the lower bound} asymptotically sharp. Consequently, the spectral norm of the corresponding “wide” random Toeplitz matrices (i.e., those with many more columns than rows) is asymptotically equivalent to \( \sqrt{p \log n} \).

The convergence \( K_{1,n/p} \to 1 \) follows from~\eqref{new positive definite problem} via the following argument. Any positive definite function \( \Lambda(x) \) with \( \Lambda(0) = 1 \) satisfies the pointwise bound \( |\Lambda(x)| \leq 1 \) and the symmetry \( \Lambda(-x) = \overline{\Lambda(x)} \) (see, e.g., Sec.~1.4 of \cite{rudin62}). Therefore, when \( n/p \to \infty \), the optimization problem~\eqref{new positive definite problem} with \( u = 1 \) and \( w = n/p \) reduces to
\[
\sup_{\Lambda \in \mathcal{F}_1} \left( \int_{-\infty}^\infty \Lambda(x)\, \mathrm{d}x \right)^{1/2},
\]
whose solution, as shown in \cite{boas45}, equals 1.

\section{Proof of Theorem \ref{T theorem}}\label{sec: proof of theorem 1}

The argument begins by following several steps from SV13, supplemented with additional ideas tailored to our setting. The main novelty --- introduced later in the proof --- is the connection to the maximization problem~\eqref{new positive definite problem}. We focus on the new elements of the argument and provide only a brief summary of the steps shared with SV13, with full details deferred to the Supplementary Material (SM). 

\subsection{Circulant embedding}\label{sec:circulant embedding}
First, we embed \( \mathbf{T}_{p \times n} \) into an \( N \times N \) circulant matrix with the first row \( (a_0, \dots, a_{n-1}, a_{-p}, \dots, a_{-1}) \), where \( N = n + p \). In our non-symmetric setting, the random variables \( a_{-j} \) and \( a_j \) (for \( j > 0 \)) are independent and identically distributed. For notational convenience, we re-index the negative indices by setting \( a_{-j} := a_{N-j} \). With this convention, we define \( \mb{a}^\top = (a_0, a_1, \dots, a_{N-1}) \), and by equation~\eqref{PDP matrix}, we have
\begin{equation}\label{common form}
\|\mathbf{T}_{p \times n}\| = \sqrt{N}\left\|  \, \mb{P}_{p,N} \, \mb{D}(\mb{a}) \, \mb{P}_{n,N} \right\|.
\end{equation}

We next collect several elementary but useful facts about the projection matrix \( \mb{P}_{r,N}\) and the diagonal matrix \( \mb{D}(\mb{a}) \). Straightforward calculations yield the following explicit formula:
\begin{equation}\label{explicit formula for entries of P}
 (\mb{P}_{r,N})_{kl}=\frac{1}{N}\sum_{j=0}^{r-1}e^{\frac{2\pi\rmi}{N}(l-k)j}=\begin{cases}
     \frac{r}{N}&\text{if }k=l\\
     \frac{1}{N}\frac{1-e^{\frac{2\pi\rmi}{N}(l-k)r}}{1-e^{\frac{2\pi\rmi}{N}(l-k)}}&\text{if }k\neq l
 \end{cases}.
 \end{equation}
For $k\neq l$, this implies 
 \begin{equation}\label{bound on the entries of P}
 |(\mb{P}_{r,N})_{kl}|=\frac{1}{N}\left|\frac{\sin\frac{\pi (l-k)r}{N}}{\sin\frac{\pi (l-k)}{N}}\right|\leq\frac{\left|\sin\frac{\pi (l-k)r}{N}\right|}{2\min\left\{ |l-k|,N-|l-k|\right\}},
 \end{equation}
 where the inequality follows from the fact that $\sin x\geq \min\{2x/\pi,2-2x/\pi\}$ for $x\in[0,\pi]$.
The numerator in \eqref{bound on the entries of P} is small when $r/N$ is near $0$ or $1$. 
 
 In our setting, $r/N$ is near $0$ when $r=p$ and $n/p$ is large. In contrast, in the square matrix setting of SV13, attention can be restricted to $r/N=1/2$, so that the numerator on the right hand side of \eqref{explicit formula for entries of P} equals $0$ whenever $l-k$ is zero and $1$ whenever $l-k$ is odd.

 Turning to matrix $\mb{D}(\mb{a})$, let $d_j=d_j(\mb{a})$ denote its $j+1$-th diagonal entry. We have
\begin{equation}
\label{def of d again}
d_j=\frac{1}{\sqrt{N}}\sum_{k=0}^{N-1}e^{\frac{2\pi\rmi}{N}kj}a_k=\overline{d_{N-j}}.
\end{equation}
Moreover, for any distinct $s,t\in\{0,\dots,\lfloor N/2\rfloor\}$,
\begin{eqnarray*}
\mathbb{E}(d_s\overline{d_t})=\frac{1}{N}\sum_{k=0}^{N-1}e^{\frac{2\pi\rmi}{N}k(s-t)}=0,\qquad \mathbb{E}(d_sd_t)=\frac{1}{N}\sum_{k=0}^{N-1}e^{\frac{2\pi\rmi}{N}k(s+t)}=0.
\end{eqnarray*}
Additionally, 
\[
\mathbb{E}|d_s|^2=1,\qquad \mathbb{E}d_s^2=\begin{cases}
    0&\text{if }s\notin \{0,N/2\}\\
    1&\text{if }s\in \{0,N/2\}
\end{cases}.
\]
Hence, the random variables in the set $\{\Re d_j, \Im d_j\mid j=0,\dots,\lfloor N/2\rfloor\}$ are mutually uncorrelated. Note that $\Im d_0=\Im d_{N/2}=0$ (with $d_{N/2}$ defined only when $N$ is even). The variance of $\Re d_0$ and $\Re d_{N/2}$ are 1, while all other components have variance $1/2$.

In the special case where $a_j$, $j\in\mathbb{Z}$, are standard normal random variables, both $d_0$ and $d_{N/2}$ are standard normal as well. For $j\notin\{0,N/2\}$, the variables $d_j$ are standard complex normal random variables --- that is, $d_j=z_1+\mathrm{i}z_2$, where $z_1$ and $z_2$ are independent $\mathcal{N}(0,1/2)$. Consequently, $|d_j|^2$ follows a standard exponential distribution $\operatorname{Exp}(1)$.

By contrast, in SV13's symmetric setting and under the same Gaussian assumption (after a scaling of 
$a_0$ and $a_n$, which does not affect the results), all $d_j$ are real standard normal random variables, so that $|d_j|^2$ follows a $\chi^2(1)$ distribution. 

\subsection{Truncation, Normalization, and Auxiliary Assumptions}
To streamline the analysis and follow SV13, we make three standard reductions: (i) we replace the variables $a_j$ with appropriately truncated and standardized versions $\bar{a}_j$; (ii) we strengthen convergence in probability to convergence in $L^\gamma$; and (iii) we assume that $N$ is even. These steps are justified by the following lemma, with full details deferred to the SM.

\begin{lemma}[Reductions and standardizations]
\label{lem:reductions}
Under assumptions \eqref{Lyapunov} and \eqref{new asymptotic regime}, the following simplifications hold without affecting the asymptotic result in Theorem~\ref{T theorem}:
\begin{description}
    \item[(i)] The random variables \( a_j \) may be replaced by truncated and standardized versions \( \bar{a}_j \), uniformly bounded by \( n^{1/\gamma} \).
    \item[(ii)] Convergence in probability can be strengthened to convergence in \( L^\gamma \), provided \( \limsup_{n \to \infty} n/p <\infty \).
    \item[(iii)] It is without loss of generality to assume that \( N = n + p \) is even.
\end{description}
\end{lemma}
Part~(i) of Lemma~\ref{lem:reductions} allows us to assume throughout the remainder of the proof that $|a_j|<n^{1/\gamma}$ for all $j$, so that tail bounds such as Bernstein’s inequality (e.g., Theorem 2.8.4 in \cite{vershynin2018}) hold. 
Part~(ii) ensures that convergence in probability can be strengthened to convergence in \( L^\gamma \), as required in Theorem~\ref{T theorem}.
Finally, part~(iii) allows us to avoid repeatedly distinguishing between even and odd values of $N$, thereby making several proofs more concise.

\subsubsection*{Reduction to dominant diagonal entries}

Next, we show that the asymptotic behavior of the norm \(\|\mb{P}_{p,N} \mb{D}(\mb{a}) \mb{P}_{n,N}\|\) is essentially determined by those diagonal entries \(d_j\) of \(\mb{D}=\mb{D}(\mb{a})\) with relatively large absolute values. This argument parallels Section 3.2 of SV13 but involves a modification: the threshold for significance is lowered to order \((\log n)^{\frac{1}{2} - \frac{\alpha}{2}}\), compared to the original \((\log n)^{1/2}\). This adjustment is necessary to accommodate regimes where \(n/p \to \infty\).

To formalize this, define
\begin{equation}\label{epsilon_n definition}
\epsilon_n := (\log n)^{-\frac{\alpha}{2}},
\end{equation}
and the random set of \emph{active} indices
\begin{equation}\label{active set}
S := \{0 \leq j \leq N-1 : |d_j| \geq \epsilon_n (\log n)^{1/2}\}.
\end{equation}
Let \(\mb{R} := \operatorname{diag}(\mathbf{1}_{\{j \in S\}})\) be the diagonal projection onto these indices, and define the truncated diagonal matrix \(\mb{D}^\epsilon := \mb{D} \mb{R}\).

Then, similarly to (16) in SV13, we obtain the bound
\[
\left| \frac{\|\mb{P}_{p,N} \mb{D} \mb{P}_{n,N}\|}{\sqrt{(p/N) \log n}} - \frac{\|\mb{P}_{p,N} \mb{D}^\epsilon \mb{P}_{n,N}\|}{\sqrt{(p/N) \log n}} \right| \leq \frac{\|\mb{P}_{p,N}\| \, \|\mb{D} - \mb{D}^\epsilon\| \, \|\mb{P}_{n,N}\|}{\sqrt{(p/N) \log n}} \leq \frac{\epsilon_n}{\sqrt{p/(2n)}} = o(1),
\]
where we used the fact that $N=n+p\leq 2n$ and assumption \eqref{new asymptotic regime}.

This reduction shows that the asymptotic behavior is effectively governed by the entries indexed by 
$S$, which also serve as the foundation for the block-diagonal reduction developed next.

\subsection{Reduction to block-diagonal form}\label{sec: block-diagonal redux}
The structure of this reduction is guided by a partition of the set \( \{0,1,\dots,N-1\} \), very similar to that introduced in Section 3.3 of SV13. 

For the reader’s convenience and to keep the presentation self-contained, we briefly recall the terminology and construction from SV13 --- adopting some modifications. In places, we follow the exposition nearly \emph{verbatim} to ensure clarity and consistency, while also highlighting key differences.

 \subsubsection*{Partition of $\{0,1,\dots,N-1\}$}\label{subsection step 6}  

Let \( r_n := \lceil \log N \rceil^4 \) and \( m_n := \lfloor N/(4r_n) \rfloor \). Divide the interval \(\{0,1,\dots,N-1\}\) into \(2m_n + 1\) consecutive disjoint subintervals (called \emph{bricks}) \( L_{-m_{n}}, \dots, L_{-1}, L_0, L_1, \dots, L_{m_n} \) so that:
\begin{itemize}
    \item \( 0 \in L_{-m_n} \) and \( N/2 \in L_0 \),
    \item each brick has length between \( r_n \) and \( 4r_n \),
    \item the subdivision is symmetric around \( N/2 \): \( L_i = N - L_{-i} \) for \( -m_n < i < m_n \), and \( L_{-m_n} \setminus \{0\} = N - L_{m_n} \).
\end{itemize}

Next, define a \emph{block} as a nonempty union of consecutive bricks:
\[
J = L_j \cup L_{j+1} \cup \cdots \cup L_{k}, \quad \text{for } -m_n \leq j \leq k \leq m_n.
\]
Let \( M_n := \lceil 100\log^\alpha n \rceil \). We say that a block \( J \) is \emph{admissible} if:
\begin{description}
    \item[(a)] \( L_{-m_n} \nsubseteq J \) and \( L_0 \nsubseteq J \),
    \item[(b)] the number of bricks comprising \( J \), i.e., \( 1 + k - j \), is at most \( M_n \).
\end{description}

The set of all admissible blocks is denoted by \( \mathcal{L} \). This construction of \( \mathcal{L} \) mirrors that in SV13, with two differences: (i) bricks here are  longer (in SV13, \( r_n \sim \log^3 n \), while here \( r_n \sim \log^4 n \)), and (ii) the maximum number of bricks in a block, \( M_n \), grows with \( n \), whereas it is fixed in SV13.

Recall the definition \eqref{active set} of the set \( S \) of active indices. A brick \( L_k \) is said to be \emph{visible} if \( L_k \cap S \neq \varnothing \). Given \( S \), we partition \( \{0,1,\dots,N-1\} \) into disjoint intervals \( J \) by placing cuts between every pair of consecutive \emph{invisible} bricks. Denote the resulting collection of disjoint blocks by \( \Lambda \).

The following proposition is the analogue of Proposition 8 in SV13, adapted to our setting. Its proof is provided in the SM.

\begin{proposition}\label{equivalent prop 8}
    For each division of the interval $\{0,1,\dots,N-1\}$ into bricks, the following holds with high probability (that is, with probability approaching one as $n\rightarrow \infty$): 
    \begin{enumerate}
        \item For each $J\in \Lambda$, if $J\cap S\neq \varnothing$, then $J\in\mathcal{L}$ and
        \item For all $J\in \Lambda\cap\mathcal{L}$, $\#(J\cap S)\leq M_n$.
    \end{enumerate}
\end{proposition}

We are now ready to carry out the reduction.

\subsubsection*{Block-diagonal reduction}
Let $\mb{B}_{q,N}$ be a block-diagonal reduction of $\mb{P}_{q,N}$:
 \[
 (\mb{B}_{q,N})_{kl}=\begin{cases}
     (\mb{P}_{q,N})_{kl}&\text{if }k,l\in J\text{ for some }J\in\Lambda,\\
     0&\text{otherwise.}
 \end{cases}
 \]
 The following lemma, which is analogous to Lemma 9 in SV13, shows that the norms $\|\mb{P}_{p,N}\mb{D}^\epsilon\mb{P}_{n,N}\|$ and $\|\mb{B}_{p,N}\mb{D}^\epsilon\mb{B}_{n,N}\|$ are asymptotically equivalent. A proof of the lemma, adapted to our setting, is provided in the SM.

   \begin{lemma}\label{equivalent of lemma 9}
  \[   \frac{\|\mb{P}_{p,N}\mb{D}^\epsilon\mb{P}_{n,N}\|-\|\mb{B}_{p,N}\mb{D}^\epsilon\mb{B}_{n,N}\|}{\sqrt{(p/N)\log n}}=o_P\left((\log n)^{3\alpha/2-2}\right).
  \]
 \end{lemma}
  Note that \[
\|\mb{B}_{p,N}\mb{D}^\epsilon\mb{B}_{n,N}\|=\max_{J\in\Lambda:J\cap S\neq\varnothing}\|\mb{P}_{p,N}[J]\mb{D}^\epsilon[J] \mb{P}_{n,N}[J]\|,
\]
where $\mb{M}[J]$ denotes the principal minor of an $N\times N$ matrix $\mb{M}$ corresponding to the index set $J$. In the remaining part of the proof, it will be convenient to work with the right hand side of this identity.

 \subsection{Proof of the upper bound}\label{sec: upper bound}
This subsection is devoted to proving the following sharp upper bound.
     \begin{proposition}\label{prop:upper-bound}
Let \( \Lambda \), \( S \), and \( \mb{D}^\epsilon \) be as defined above. Then
\begin{equation}\label{eq: T upper bound}
\max_{J \in \Lambda : J \cap S \neq \varnothing} 
\frac{\| \mb{P}_{p,N}[J] \mb{D}^\epsilon[J] \mb{P}_{n,N}[J] \|}{\sqrt{(p/N)\log n}} 
\leq K_{1,n/p} + o_P(1),
\end{equation}
as \( n \to \infty \).
\end{proposition}

 The following result is a counterpart of Corollary 13 from SV13. Its proof is deferred to the SM.
 \begin{lemma}
     \label{cor: corollary 13}
    Let $M=M_n:=\left\lceil 100 \log^\alpha n\right\rceil $ as defined in subsection \ref{subsection step 6}. For every $\eta>0$, with high probability, for all admissible blocks $L\in\mathcal{L}$ and all distinct $j_1,\dots,j_M\in L$, we have $|d_{j_1}|^2+\dots+|d_{j_M}|^2\leq (1+\eta)\log n.$
 \end{lemma}

By Proposition~\ref{equivalent prop 8}, with high probability the matrix $\mb{D}^\epsilon$ has at most 
$M_n$  non-zero entries in each admissible block. Combining this with Lemma~\ref{cor: corollary 13}, we conclude that for every $\eta>0$, with high probability
\[
\sum_{j\in L}|\mb{D}^\epsilon_{jj}|^2\leq(1+\eta)\log n\qquad\text{for all }L\in\mathcal{L}.
\]
It follows that, with high probability
\begin{eqnarray}\notag
 &&\max_{J\in\Lambda:J\cap S\neq\varnothing}\frac{\|\mb{P}_{p,N}[J]\mb{D}^\epsilon[J] \mb{P}_{n,N}[J]\|}{\sqrt{(p/N)\log n}}\\\label{return back}
 &\leq& \sup\left\{\frac{\|\mb{P}_{p,N}[L]\operatorname{diag}(\delta_1,\dots,\delta_{\#L})\mb{P}_{n,N}[L]\|}{\sqrt{p/N}}:L\in\mathcal{L}, \delta_j\in \mathbb{C}, \sum_{j=1}^{\# L}|\delta_j|^2\leq 1+\eta\right\}.
\end{eqnarray}

Lifting the supremum over blocks in \eqref{return back} to the full matrix level, we obtain
\begin{eqnarray*}
\eqref{return back}\leq \sqrt{\frac{1+\eta}{p/N}}\sup\left\{\|\mb{P}_{p,N}\operatorname{diag}(\delta)\mb{P}_{n,N}\|:\delta\in\mathbb{C}^N, \|\delta\|\leq 1\right\}.
\end{eqnarray*}
On the other hand, by the Cauchy-Schwarz inequality,
\begin{eqnarray*}
\|\mb{P}_{p,N}\operatorname{diag}(\delta)\mb{P}_{n,N}\|&=&\sup\left\{|\mb{v}^\top\mb{P}_{p,N}\operatorname{diag}(\delta)\mb{P}_{n,N}\mb{w}|:\mb{v},\mb{w}\in\mathbb{C}^N,\|\mb{v}\|=\|\mb{w}\|=1\right\}\\
&\leq&\sup\left\{\|(\mb{P}_{p,N}\mb{v})\odot (\mb{P}_{n,N}\mb{w})\|\|\delta\|:\mb{v},\mb{w}\in\mathbb{C}^N,\|\mb{v}\|=\|\mb{w}\|=1\right\},
\end{eqnarray*}
where $\odot$ denotes Hadamard's element-wise product.
Therefore, 
\begin{equation}\label{polynom product}
\eqref{return back}\leq \sqrt{\frac{1+\eta}{p/N}}\sup\left\{\|(\mb{P}_{p,N}\mb{v})\odot(\mb{P}_{n,N}\mb{w})\|:\mb{v},\mb{w}\in\mathbb{C}^N,\|\mb{v}\|=\|\mb{w}\|=1\right\}.
\end{equation}
It is thus enough to bound the norm on the right-hand side of~\eqref{polynom product}. Let us reformulate the problem of finding this supremum in terms of an optimization problem involving a pair of polynomials of degrees $n-1$ and $p-1$.

\paragraph{Polynomial reformulation.}
Recall that $N=n+p$ and  
\[
\mb{P}_{p,N}=\mathcal{F}_N^\ast\begin{pmatrix}
    \mb{I}_p&\mb{0}\\\mb{0}&\mb{0}_{n}
\end{pmatrix}\mathcal{F}_N,\qquad \mb{P}_{n,N}=\mathcal{F}_N^\ast\begin{pmatrix}
    \mb{I}_n&\mb{0}\\\mb{0}&\mb{0}_{p}
\end{pmatrix}\mathcal{F}_N.
\]
Thus, for any unit-norm $\mb{v},\mb{w}$, the vectors $\mb{P}_{p,N}\mb{v}$ and $\mb{P}_{n,N}\mb{w}$ can be written as DFTs of zero-padded vectors  $\mb{v}_p\in\mathbb{C}^p$ and $\mb{w}_n\in\mathbb{C}^n$ with norms no larger than 1, i.e.,
\[
\mb{P}_{p,N}\mb{v}=\mathcal{F}_N^\ast\begin{pmatrix}
    \mb{v}_p\\\mb{0}
\end{pmatrix},\qquad \mb{P}_{n,N}\mb{w}=\mathcal{F}_N^\ast\begin{pmatrix}
    \mb{w}_n\\\mb{0}
\end{pmatrix}.
\]
These represent sampled values of degree-$(p-1)$ and $(n-1)$ polynomials $\mathcal{P}_{p-1}(z)$ and $\mathcal{P}_{n-1}(z)$ on the discrete grid $\left\{e^{-2\pi\rmi j/N}\right\}_{j=0}^{N-1}$, scaled by $1/\sqrt{N}$, with 
\[
\mathcal{P}_{p-1}(z):=\mb{v}_{p1}+\mb{v}_{p2}z+\dots+\mb{v}_{pp}z^{p-1}\qquad\text{and}\qquad \mathcal{P}_{n-1}(z):=\mb{w}_{n1}+\mb{w}_{n2}z+\dots+\mb{w}_{nn}z^{n-1}.
\]
Hence, the Hadamard product in \eqref{polynom product} corresponds to values of the product polynomial $\mathcal{P}_{p-1}(z)\mathcal{P}_{n-1}(z)/N$ sampled on the grid.

Note that, for any polynomial $\mathcal{P}_{q-1}:= \mb{a}_{q1}+\mb{a}_{q2}z+\dots+\mb{a}_{qq}z^{q-1}$ of degree $q-1<N$, its values on the grid are given by $\sqrt{N}\mathcal{F}_N^\ast\begin{pmatrix}
    \mb{a}_q\\\mb{0}
\end{pmatrix}$, where $\mb{a}_{q}=(\mb{a}_{q1},\dots,\mb{a}_{qq})^\top$. Since $\mathcal{F}_N$ is unitary, we have the identity:
\begin{equation}\label{grid identity}
\left\|\sqrt{N}\mathcal{F}_N^\ast\begin{pmatrix}
    \mb{a}_q\\\mb{0}
\end{pmatrix}\right\|^2=N\|\mb{a}_q\|^2=N\|\mathcal{P}_{q-1}\|_2^2,
\end{equation}
where $\|\mathcal{P}_{q-1}\|_2^2:=\frac{1}{2\pi}\int_0^{2\pi}|\mathcal{P}_{q-1}(e^{\rmi \omega})|^2\mathrm{d}\omega.$
Since the degree of $\mathcal{P}_{p-1}(z)\mathcal{P}_{n-1}(z)/N$ is $p+n-2=N-2<N$, the above identity holds for the product polynomial, so that
\[
\|(\mb{P}_{p,N}\mb{v})\odot(\mb{P}_{n,N}\mb{w})\|^2=N\|\mathcal{P}_{p-1}\cdot\mathcal{P}_{n-1}/N\|_2^2=\frac{1}{N}\|\mathcal{P}_{p-1}\cdot\mathcal{P}_{n-1}\|_2^2.
\]
Since $\|\mathcal{P}_{p-1}\|_2,\|\mathcal{P}_{n-1}\|_2\leq 1$, we conclude from \eqref{polynom product} that with high probability,
\begin{equation}\label{polynomial formulation}
\max_{J\in\Lambda:J\cap S\neq\varnothing}\frac{\|\mb{P}_{p,N}[J]\mb{D}^\epsilon[J] \mb{P}_{n,N}[J]\|}{\sqrt{(p/N)\log n}}\leq \sqrt{\frac{1+\eta}{p}}\sup\left\{\|\mathcal{P}_{p-1}\cdot\mathcal{P}_{n-1}\|_2:\|\mathcal{P}_{p-1}\|_2=1, \|\mathcal{P}_{n-1}\|_2=1\right\}.
\end{equation}
Finally, we link the right-hand side of \eqref{polynomial formulation} to the variational quantity 
$K_{1,n/p}$, by expressing the norm of the product polynomial in terms of a pair of positive definite sequences.

\paragraph{Link to $K_{1,n/p}$.}
Recall that a complex-valued sequence $\{\alpha_j\}_{j\in\mathbb{Z}}$ is said to be positive definite if 
\[
\sum_{j,k=1}^t\alpha_{j-k}\xi_j\overline{\xi_k}\geq 0
\]
for any finite collection of complex numbers $\xi_1,\dots,\xi_t$. 
As is well known (see, for example, Sec.~1.4 of \cite{rudin62}) such sequences necessarily satisfy $\alpha_0\geq 0$ and $\alpha_j=\overline{\alpha_{-j}}$ for all $j\in\mathbb{Z}$.

    Let $\mathcal{L}_m$ denote the family of positive definite sequences $\{\alpha_j\}_{j\in\mathbb{Z}}$ such that  $\alpha_0=1$ and $\alpha_j=0$ for all $|j|>m$.
Now define
\begin{equation}\label{Ipn definition}
I_{p,n}:=\max_{\{\alpha_j\}\in \mathcal{L}_{p-1},\{\gamma_j\}\in\mathcal{L}_{n-1}}\left(\sum_{\nu=-p+1}^{p-1} \overline{\alpha_\nu} \gamma_{\nu}\right)^{1/2}.
\end{equation}
Note that the latter sum is always real-valued because $\overline{\alpha_\nu} \gamma_{\nu}=\alpha_{-\nu} \overline{\gamma_{-\nu}}$. The case $p=n$ was previously analyzed  in \cite{garcia69} (see their ``case of the integers'' on p.~806).

We now establish the key connection:

\begin{lemma}\label{lem: link to Ipn}
    \begin{equation*}
    \sup\left\{\|\mathcal{P}_{p-1}\cdot\mathcal{P}_{n-1}\|_2:\|\mathcal{P}_{p-1}\|_2=1, \|\mathcal{P}_{n-1}\|_2=1\right\}=I_{p,n}.\end{equation*}
\end{lemma}
\begin{proof}
As explained in \cite{garcia69}, p.~808, the Fejér–Riesz theorem (e.g.~Theorem 1 in \cite{hussen21}) implies that $\{\alpha_j\} \in \mathcal{L}_{p-1}$ if and only if there exists a polynomial $\mathcal{P}_{p-1}(z) = \sum_{\nu=0}^{p-1} \beta_\nu z^\nu$ with $\|\mathcal{P}_{p-1}\|_2 = 1$ such that
\[
\alpha_j = \sum_{\nu=-\infty}^{\infty} \beta_{j+\nu} \overline{\beta_\nu}, \quad j \in \mathbb{Z}.
\]
Similarly, $\{\gamma_j\} \in \mathcal{L}_{n-1}$ if and only if there exists a polynomial $\mathcal{P}_{n-1}(z) = \sum_{\nu=0}^{n-1} \rho_\nu z^\nu$ with $\|\mathcal{P}_{n-1}\|_2 = 1$ such that
\[
\gamma_j = \sum_{\nu=-\infty}^{\infty} \rho_{j+\nu} \overline{\rho_\nu}, \quad j \in \mathbb{Z}.
\]

Using these representations, we compute
\begin{equation}\label{identity from lem link to Ipn}
\sum_{\nu=-p+1}^{p-1} \overline{\alpha_\nu} \gamma_\nu  = \|\mathcal{P}_{p-1} \cdot \mathcal{P}_{n-1}\|_2^2.
\end{equation}
The identity follows by explicitly evaluating the integral $\frac{1}{2\pi}\int_0^{2\pi}|\mathcal{P}_{p-1}(e^{\rmi x})\mathcal{P}_{n-1}(e^{\rmi x})|^2\mathrm{d} x$. Full details of this evaluation are given in the SM.
\end{proof}

Lemma~\ref{lem: link to Ipn}, together with equation~\eqref{polynomial formulation}, implies that with high probability,
\begin{equation}
\label{integers problem}
\max_{J\in\Lambda:J\cap S\neq\varnothing}\frac{\|\mb{P}_{p,N}[J]\mb{D}^\epsilon[J] \mb{P}_{n,N}[J]\|}{\sqrt{(p/N)\log n}}\leq \sqrt{\frac{1+\eta}{p}}I_{p,n}.
\end{equation}
On the other hand, we now show that $I_{p,n}/\sqrt{p}$ is asymptotically equivalent to $K_{1,n/p}=K_{p,n}/\sqrt{p}$. 
The following result, proved in the Technical Appendix, quantifies this relationship:
\begin{lemma}\label{lem:I is close to K}
    \[
    0\leq I_{p,n}^2-K^2_{p,n}\leq \frac{1}{3}.
    \]
\end{lemma}

Since $K_{1,n/p}=K_{p,n}/\sqrt{p}$ is bounded away from zero --- e.g., by the lower bound \eqref{the lower bound} --- Lemma \ref{lem:I is close to K} implies that
\[
I_{p,n}\leq \sqrt{p}K_{1,n/p}+O\left(\frac{1}{\sqrt{p}}\right).
\]
Therefore,
\[
\max_{J\in\Lambda:J\cap S\neq\varnothing}\frac{\|\mb{P}_{p,N}[J]\mb{D}^\epsilon[J] \mb{P}_{n,N}[J]\|}{\sqrt{(p/N)\log n}}\leq \sqrt{1+\eta}K_{1,n/p}+O\left(\frac{1}{p}\right).
\]
Since $\eta>0$ is arbitrary, this yields the desired upper bound \eqref{eq: T upper bound}, completing the proof of Proposition \ref{prop:upper-bound}.

\subsection{Proof of the lower bound}
This subsection finishes the proof of Theorem \ref{T theorem} by establishing the corresponding sharp lower bound.
\begin{proposition}\label{prop: lower bound prop}
    For every fixed $\tau>0$, with high probability,
    \begin{equation}\label{eq: lower bound}
    \max_{J\in\Lambda:J\cap S\neq\varnothing}\frac{\|\mb{P}_{p,N}[J]\mb{D}^\epsilon[J] \mb{P}_{n,N}[J]\|}{\sqrt{(p/N)\log n}}\geq K_{1,n/p}-\tau.
    \end{equation}
\end{proposition}
Let us briefly outline the main idea of the proof, which parallels the argument in Proposition 10 of SV13.
As $n\rightarrow\infty$, the minors  $\mb{P}_{r,N}[J]$  are well approximated by finite sections of certain ``limiting operators'' $\Pi_{r/N}:\ell^2\rightarrow\ell^2$. The strategy is to first show that $K_{1,n/p}$ can be approximated by 
\begin{equation}\label{approximating quantity}
\|\Pi_{p/N}^{[k]}\operatorname{diag}\{\mb{u}\}\Pi_{n/N}^{[k]}\|/\sqrt{p/N}
\end{equation}
for some $\mb{u}\in\mathbb{C}^k$, where $\Pi_{r/N}^{[k]}$ denotes a $k$-dimensional section of the operator $\Pi_{r/N}$. Then, the goal is to demonstrate that this approximation provides a lower bound for the left hand side of \eqref{eq: lower bound} with high probability. 

In the square case studied in SV13, $p/N=n/N=1/2$ so that the analogue of the approximating quantity \eqref{approximating quantity} does not vary with $n$. In our rectangular setting, the quantity depends on the ratio $n/p$, which is allowed to fluctuate as $n\rightarrow\infty$. This added generality introduces new challenges, which require modifications of the approach in SV13. In what follows, we focus on these novel aspects of the argument, while deferring proofs of the more standard parts to the SM.

First, we define the ``limiting operator'' 
\begin{equation}
\label{definition of limiting operator}
\Pi_{r/N}: \ell^2(\mathbb{C}) \overset{U_{r/N}}{\longrightarrow} \ell^2(\mathbb{C}) 
\overset{\varphi}{\longrightarrow} L^2(\mathbb{T}) 
\overset{\times \mathbf{1}_{r/(2N)}}{\longrightarrow} L^2(\mathbb{T}) 
\overset{\varphi^{-1}}{\longrightarrow} \ell^2(\mathbb{C}) 
\overset{U_{r/N}^{-1}}{\longrightarrow} \ell^2(\mathbb{C}),
\end{equation}
where
\begin{itemize}
   \item $L^2(\mathbb{T}):=\left\{f:[-1/2,1/2]\rightarrow\mathbb{C}:f(-1/2)=f(1/2)\text{ and }\int_{-1/2}^{1/2}|f(x)|^2\mathrm{d}x<\infty\right\}$,
    \item $U_{r/N}$ is a unitary operator sending the coordinate vector $e_m, m\in\mathbb{Z}$, to $e^{\pi \rmi m r/N}e_m$,
    \item $\varphi:\ell^2(\mathbb{C})\rightarrow L^2(\mathbb{T})$ sends $e_m, m\in\mathbb{Z}$, to the function $x\rightarrow e^{2\pi \rmi mx}\in L^2(\mathbb{T})$ (reconstructing the Fourier series from coefficients),
    \item For any $a>0$, $\mathbf{1}_{a}$ denotes the indicator function of the interval $\left[-a,a\right]$, and $\times \mathbf{1}_{a}$ denotes the operation of pointwise multiplication by this function.
\end{itemize}
The matrix entries of $\Pi_{r/N}$ are given by
\begin{equation}\label{entries of Pi}
\Pi_{r/N}(k,l):=\langle e_k,\Pi_{r/N} e_l\rangle =e^{\frac{\pi\rmi}{N}(l-k)r}\int_{-r/(2N)}^{r/(2N)}e^{2\pi\rmi (l-k)x}\mathrm{d}x= \begin{cases}
    r/N&\text{if }k=l\\
    \frac{1}{2\pi\rmi(l-k)}\left(e^{\frac{2\pi\rmi}{N}(l-k)r}-1\right)&\text{if }k\neq l.
\end{cases}
\end{equation}
Comparing this to the explicit formula \eqref{explicit formula for entries of P} for the entries of $\mb{P}_{r,N}$, we arrive at the following lemma (see SM for a brief proof).
\begin{lemma}\label{lem: comparison with limiting}
    For all integers $0\leq k,l\leq N-1$, such that $|k-l|<N/2$, and for all $1\leq r\leq N-1$,
    \[
    \left|(\mb{P}_{r,N})_{k,l}-\Pi_{r/N}(k,l)\right|<2/N.
    \]
\end{lemma}

The following lemma establishes a key connection between the operators $\Pi_{p/N}$, $\Pi_{n/N}$, and the constant $K_{1,n/p}$. A detailed proof is provided in the Technical Appendix.
\begin{lemma}\label{lem: link between Pi and K} The constant $K_{1,n/p}$ satisfies
    \begin{equation}\label{link between Pi and K}
K_{1,n/p}=\frac{\sup_{\mathbf{c}_1,\mathbf{c}_2\in\ell^2(\mathbb{C}):\|\mathbf{c}_1\|=\|\mathbf{c}_2\|=1}\|\Pi_{p/N}\mathbf{c}_1\odot \Pi_{n/N}\mathbf{c}_2\|}{\sqrt{p/N}},
\end{equation}
where $\odot$ denotes the entrywise multiplication of vectors from $\ell^2(\mathbb{C})$.
\end{lemma}
\begin{remark}\label{remark: deficiency}
    Essentially the same proof leads to the conclusion that for any $0<a\leq b$ such that $a+b\leq 1$,
    \[
    K_{a,b}=\sup_{\mathbf{c}_1,\mathbf{c}_2\in\ell^2(\mathbb{C}):\|\mathbf{c}_1\|=\|\mathbf{c}_2\|=1}\|\Pi_{a}\mathbf{c}_1\odot \Pi_{b}\mathbf{c}_2\|.
    \]
\end{remark}

Next, we approximate $K_{1,n/p}$ by an expression of form \eqref{approximating quantity}, where $\Pi_{r/N}^{[k]}$ denotes the $k\times k$ matrix with $(i,j)$-th entry $\langle e_i,\Pi_{r/N}e_j\rangle$. For any $\tau\in(0,1)$ and a positive integer $T$, consider a uniform grid $\tau_0,\dots,\tau_T$ over $[\tau,1]$. That is, set $\tau_t:=\tau+(1-\tau)t/T$.
As we show in the Technical Appendix, Lemma \ref{lem: link between Pi and K} implies the following result.
\begin{lemma}\label{lem: grid}
    For any $\tau\in(0,3/4)$, there exist positive integers $T=T(\tau), k=k(\tau)$, and vectors $\mb{u}^{(1)},\dots,\mb{u}^{(T)}$ from the interior of the unit ball in $\mathbb{C}^k$ such that, for any $t\in\{1,\dots,T\}$ and all $p/n\in\left(\tau_{t-1},\tau_t\right]$,
    \begin{equation}\label{grid bound}
        \frac{\|\Pi_{p/N}^{[k]}\operatorname{diag}\{\mb{u}^{(t)}\}\Pi_{n/N}^{[k]}\|}{\sqrt{p/N}}>K_{1,n/p}-\tau/2.
    \end{equation}
    Moreover, for all $p/n\in\left[\log^{-\alpha} n,\tau\right]$ and all sufficiently large $n$, we have:
    \begin{equation}\label{first grid interval}
        \frac{\|\Pi_{p/N}^{[2L_n+1]}\operatorname{diag}\{\mb{u}^{(0)}\}\Pi_{n/N}^{[2L_n+1]}\|}{\sqrt{p/N}}>K_{1,n/p}-3\tau/4,
    \end{equation}
    where  $L_n:=\lceil \log n\rceil$ and $\mb{u}^{(0)}=(0,\dots,0,1-\tau/15,0,\dots,0)^\top$ with $L_n$ zeros on both sides of the term $1-\tau/15$.
\end{lemma}
\begin{remark}\label{remark: nonzero u}
   By slightly decreasing the right-hand side of \eqref{grid bound} --- for instance, to \(K_{1,n/p} - 3\tau/4\) --- we can ensure that all entries of the vectors \(\mathbf{u}^{(t)}\), \(t = 1, \dots, T\), are nonzero.
\end{remark}

We now turn to the final step in the proof of Proposition~\ref{prop: lower bound prop}, linking the left-hand side of \eqref{eq: lower bound} with those of \eqref{grid bound} and \eqref{first grid interval}. This step requires a technical proposition, which we will state after introducing the necessary definitions and notation.

Fix an integer $k\geq 1$ and let $u_1,\dots,u_k$ be fixed nonzero complex numbers such that
\[
|u_1|^2+\cdots+|u_k|^2<1.
\]
Define the parameters
\[
q=q(N):=100\lceil\log N\rceil^4,\quad  b=b(N):=12\lceil\log N\rceil^4,\quad m=m(N):=\lfloor N/(4q)\rfloor.
\]
Let $B(x,R)$ denote the closed ball in $\mathbb{C}$ centered at $x\in\mathbb{C}$ with radius $R\in\mathbb{R}_+$. Define the balls
\[
B_j=\begin{cases}
    B(0,\epsilon_n),&\text{for }-b+1\leq j\leq 0,\\
    B(u_j,\eta),&\text{for }1\leq j\leq k,\\
    B(0,\epsilon_n),&\text{for }k+1\leq j\leq k+b.
\end{cases}
\]
Here, $\epsilon_n:=(\log n)^{-\alpha/2}$, as defined in \eqref{epsilon_n definition}. The parameter $\eta>0$ is fixed to satisfy
\begin{equation}\label{new equation}
    \eta<\min\left\{\frac{1}{2}\min_{1\leq i\leq k}|u_i|,\frac{1}{3k}\left(\frac{1}{2}-\frac{1}{\gamma}\right)\right\}.
\end{equation}

Recall that $d_j(\mb{a})$ denotes the $j+1$-th diagonal entry of $\mb{D}(\mb{a})$ with its explicit form given in \eqref{def of d again}. 
For each $i=1,\dots,m,$ let $A_i$ be the event that the following three conditions hold:
\begin{description}
    \item[(i)] $|d_{iq+j}(\mb{a})|\in \sqrt{\log n} \times B_j$ for all $j=-b+1,\dots,0$,
    \item[(ii)] $d_{iq+j}(\mb{a})\in \sqrt{\log n} \times B_j$ for all $j=1,\dots,k$,
    \item[(iii)] $|d_{iq+j}(\mb{a})|\in \sqrt{\log n} \times B_j$ for all $j=k+1,\dots,k+b$.
\end{description}

The following proposition is an analogue of Proposition 15 from SV13. Its proof closely follows the argument given there and is deferred to the SM.
\begin{proposition}\label{prop: equivalent of prop 15}
Let $A_i$ be as above. Then, the probability that at least one of the events $A_1,A_2,\dots,A_m$ occurs converges to one as $n\rightarrow\infty$.
\end{proposition}

We now complete the proof of Proposition~\ref{prop: lower bound prop} by showing that, for any small $\tau,\varepsilon>0$, there exists $n_{\tau,\varepsilon}$ such that, for all $n\geq n_{\tau,\varepsilon}$, inequality \eqref{eq: lower bound} holds with probability at least $1-\varepsilon$. We proceed in a manner similar to SV13 (pp.~4073–4074). For the reader's convenience, we will break up the argument into several steps.

\paragraph{Step 1 - Setup.}
Let \(\mathbf{u}^{(t)} = (u^{(t)}_1, \dots, u^{(t)}_k)^\top\) for  \(t=1,\dots,T\) be as in \eqref{grid bound}. By Remark~\ref{remark: nonzero u}, all coordinates of these vectors can be taken non-zero.

    Choose \(\eta > 0\) sufficiently small and \(n_{\tau,\varepsilon}\) sufficiently large so that the following \textbf{conditions} hold:

\begin{description}
    \item[1. Asymptotic regime:] \(p/n > \log^{-\alpha} n\) and  \(\eta>\epsilon_n=\log^{-\alpha/2} n\)  for all \(n \geq n_{\tau,\varepsilon}\).
    
    \item[2. Inequality \eqref{new equation}:] holds for \(u_i=u_i^{(t)} (t=1,\dots,T)\)  and also for \(k = 1,\) \(u_1 = 1 - \tau/15\).
    
    \item[3. Operator norm bounds:] For each \(t = 1, \dots, T\) and \(p/n \in (\tau_{t-1}, \tau_t]\), if \(\mathbf{v}^{(t)} = (v_1^{(t)}, \dots, v_k^{(t)}) \in B(u_1^{(t)}, \eta) \times \dots \times B(u_k^{(t)}, \eta)\), then
    \[
    \frac{\|\Pi_{p/N}^{[k]} \operatorname{diag}(\mathbf{v}^{(t)}) \Pi_{n/N}^{[k]}\|}{\sqrt{p/N}} > K_{1, n/p} - 5\tau/6.
    \]
    For \(t=0\) and \(p/n \in (\log^{-\alpha} n, \tau]\), if \(v_1 \in B(1 - \tau/15, \eta)\) and 
    \(\mathbf{v}^{(0)} = (0, \dots, 0, v_1, 0, \dots, 0)^\top\) with \(L_n\) zeros on both sides of the \(v_1\) entry, then
    \[
    \frac{\|\Pi_{p/N}^{[2L_n+1]} \operatorname{diag}(\mathbf{v}^{(0)}) \Pi_{n/N}^{[2L_n+1]}\|}{\sqrt{p/N}} > K_{1, n/p} - 5\tau/6.
    \]
    
    \item[4. Norm control:] For each \(t = 1, \dots, T\), 
    \[
    \max\{\|\mathbf{v}^{(t)}\| : \mathbf{v}^{(t)} \in B(u_1^{(t)}, \eta) \times \dots \times B(u_k^{(t)}, \eta)\} < 1.
    \]
    Also \(1 - \tau/15 + \eta < 1.\)
\end{description}
    
\paragraph{Step 2 - Events.}
For each \(t \in \{1, \dots, T\}\) define \(A_i^{(t)}\) as \(A_i\) with \(u_j = u^{(t)}_j\). For $t=0$, set \(k = 1\) and \(u_1 = 1 - \tau/15\). All these definitions use the same $\eta$ chosen above.

By Proposition~\ref{prop: equivalent of prop 15} and the finiteness of \(T \)  we may enlarge \(n_{\tau,\varepsilon}\) so that, for all \(n \geq n_{\tau,\varepsilon}\), 
\[
\mathbb{P}\left[\text{for each }t \in \{0,1,\dots,T\}, \text{ some } A_i^{(t)}\text{ occurs}\right]\geq 1-\varepsilon.
\]

\paragraph{Step 3 - Case $p/n\in(\tau_{t-1},\tau_t], t\geq 1$.}
If \(A_i^{(t)}\) occurs, then, for each $j=1,\dots,k$, $d_{iq+j}(\mb{a})\in B(u_j,\eta)$. Conditions (1) and (2) imply that $|d_{iq+j}(\mb{a})|>\epsilon_n$ for all $j=1,\dots,k$. Recall that by construction, each partition block $J\in\Lambda$ has to start and end with an invisible brick, and the minimum possible length of any brick is $\lceil \log N\rceil^4$. 
Therefore, if \(A_i^{(t)}\) occurs, then the points $iq-\lceil \log N\rceil^4,\dots,iq+k+\lceil \log N\rceil^4$ must be contained in a single block, which we call $J^{(t)}$.

Note that $J^{(t)}$ cannot contain any visible points other than $iq+1,\dots,iq+k$. Indeed, event \(A_i^{(t)}\) implies that points $iq-12\lceil\log N\rceil^4,\dots,iq$ and $iq+k+1,\dots,iq+k+12\lceil\log N\rceil^4$ are invisible. On the other hand, two consecutive invisible bricks cannot belong to the same partition block, and the maximum possible length of a brick is $4\lceil \log N\rceil^4$.

Let $F^{(t)}:=\{iq+1,\dots,iq+k\}$. Since $F^{(t)}\subset J^{(t)}$, 
\begin{eqnarray*}
    &&\frac{\|\mb{P}_{p,N}[J^{(t)}]\mb{D}^\epsilon[J^{(t)}] \mb{P}_{n,N}[J^{(t)}]\|}{\sqrt{(p/N)\log n}} \geq \frac{\|\mb{P}_{p,N}[F^{(t)}]\mb{D}^\epsilon [F^{(t)}]\mb{P}_{n,N}[F^{(t)}])\|}{\sqrt{(p/N)\log n}}\\
    &\geq&\inf\left\{\frac{\|\mb{P}_{p,N}[F^{(t)}]\operatorname{diag}\{\mb{v}^{(t)}\}\mb{P}_{n,N}[F^{(t)}]\|}{\sqrt{(p/N)}}:\mb{v}^{(t)}\in B(u_1^{(t)},\eta)\times\dots\times B(u_k^{(t)},\eta)\right\} . 
\end{eqnarray*}
By Lemma \ref{lem: comparison with limiting}, 
\[
\|\mb{P}_{p,N}[F^{(t)}]-\Pi_{p/N}^{[k]}\|\leq 2k/N, \qquad \|\mb{P}_{n,N}[F^{(t)}]-\Pi_{n/N}^{[k]}\|\leq 2k/N.
\]
Therefore, increasing $n_{\tau,\varepsilon}$ if needed, ensures that the latter infimum is no smaller than
\[
\inf\left\{\frac{\|\Pi_{p/N}^{[k]}\operatorname{diag}\{\mb{v}^{(t)}\}\Pi_{n/N}^{[k]}\|}{\sqrt{(p/N)}}:\mb{v}^{(t)}\in B(u_1^{(t)},\eta)\times\dots\times B(u_k^{(t)},\eta)\right\}-\tau/6.
\]
Condition (3) implies that this is no smaller than $K_{1,n/p}-\tau$, which gives \eqref{eq: lower bound}.

\paragraph{Step 4 - Case $p/n\in(\log^{-\alpha}n, \tau]$.}
The argument is identical but with $k=1$, $u_1=1-\tau/15$, and $F^{(0)}:=\{iq+1-L_n,\dots,iq+1+L_n\}$.
Condition (3) again yields the bound $K_{1,n/p}-\tau$, completing the proof of Proposition \ref{prop: lower bound prop} and of Theorem \ref{T theorem}.

\section{Proof of Theorems \ref{C theorem} and \ref{S theorem}}\label{sec: proof of the other theorems}
\subsection{Proof of Theorem \ref{C theorem}}


As discussed in the Introduction, the key difference between Theorems~\ref{C theorem} and~\ref{T theorem} arises from the simpler representation of \(\|\mb{C}_{p\times n}\|\) in the form~\eqref{PDP matrix}:
\[
\|\mb{C}_{p\times n}\| = \sqrt{N}\,\|\mb{P}_{p,N}\,\mb{D}(\mb{a})\,\mb{P}_{n,N}\|.
\]
In contrast to the Toeplitz case, where \(N = n + p\), here \(N = n\), and consequently \(\mb{P}_{n,N} = \mb{I}_n\).

Keeping this distinction in mind, the proof of Theorem~\ref{C theorem} proceeds in the same manner as that of Theorem~\ref{T theorem} up to inequality~\eqref{polynom product}, which in the present setting takes the simpler form
\[
\max_{\substack{J\in\Lambda \\ J\cap S\neq\varnothing}}
\frac{\|\mb{P}_{p,n}[J]\,\mb{D}^\epsilon[J]\|}{\sqrt{(p/n)\log n}}
\leq \sqrt{\frac{1+\eta}{p/n}}
\sup_{\substack{\mb{v},\mb{w}\in\mathbb{C}^n \\ \|\mb{v}\|=\|\mb{w}\|=1}}
\|(\mb{P}_{p,n}\mb{v})\odot \mb{w}\|.
\]
Since
\[
\|(\mb{P}_{p,n}\mb{v})\odot \mb{w}\| \leq \|\mb{P}_{p,n}\mb{v}\|_{\infty}\,\|\mb{w}\|,
\]
the supremum on the right-hand side reduces to \(\|\mb{P}_{p,n}\|_{2\rightarrow\infty}\), the maximum Euclidean norm of the rows of \(\mb{P}_{p,n}\).

As \(\mb{P}_{p,n}\) is a self-adjoint projection with all diagonal entries equal to \(p/n\), each row has Euclidean norm \(\sqrt{p/n}\). Therefore,
\[
\max_{\substack{J\in\Lambda \\ J\cap S\neq\varnothing}}
\frac{\|\mb{P}_{p,n}[J]\,\mb{D}^\epsilon[J]\|}{\sqrt{(p/n)\log n}}
\leq \sqrt{1+\eta},
\]
which yields the desired tight upper bound on \(\|\mb{C}_{p\times n}\|\).

The proof of the lower bound is nearly identical to that for the Toeplitz case, with the following lemma replacing Lemma~\ref{lem: grid}:

\begin{lemma}\label{lem: grid lemma for circulant}
For any $\tau \in (0, 3/4)$, there exist positive integers $T = T(\tau)$, $k = k(\tau)$, and vectors $\mb{u}^{(1)}, \dots, \mb{u}^{(T)}$ from the interior of the unit ball in $\mathbb{C}^k$ such that, for any $t \in \{1, \dots, T\}$ and all $p/n \in (\tau_{t-1}, \tau_t]$,
\begin{equation}\label{C grid bound}
    \frac{\|\Pi_{p/n}^{[k]}\operatorname{diag}\{\mb{u}^{(t)}\}\|}{\sqrt{p/n}} > 1 - \tau/2.
\end{equation}
Moreover, for all $p/n \in [\log^{-\alpha} n, \tau]$ and all sufficiently large $n$, we have
\begin{equation}\label{C first grid interval}
    \frac{\|\Pi_{p/n}^{[2L_n+1]}\operatorname{diag}\{\mb{u}^{(0)}\}\|}{\sqrt{p/n}} > 1 - 3\tau/4,
\end{equation}
where $L_n := \lceil \log n \rceil$ and $\mb{u}^{(0)} = (0, \dots, 0,\, 1 - \tau/15,\, 0, \dots, 0)^\top$ with $L_n$ zeros on each side of the entry $1 - \tau/15$.
\end{lemma}

The proof of this lemma follows the same structure as the proof of Lemma~\ref{lem: grid}, but is considerably simpler. We therefore defer its proof to the SM.

\subsection{Proof of Theorem \ref{S theorem}}
In what follows, we concentrate on the symmetric Toeplitz setting of Theorem~\ref{S theorem}, where the proof naturally parallels that of Theorem~\ref{T theorem} with certain key differences that will be emphasized. The arguments for the symmetric circulant case are analogous to those used in proving Theorem~\ref{C theorem} and thus will not be detailed here.

It is natural to embed the rectangular symmetric Toeplitz matrix $\mb{T}_{p\times n}^{(s)}$ into the square symmetric circulant $\mb{C}_{N\times N}^{(s)}$ with $N=2n$ (in contrast to $N=n+p$ in Theorem \ref{T theorem}).  This embedding guarantees that the diagonal entries, $d_j$, of $\mathbf{D}(\mathbf{a}) := \operatorname{diag}(\mathcal{F}_N \mathbf{a})$ are real-valued rather than complex-valued. 

Arguing as in Lemma 5 of SV13 to justify the replacement of $a_0,a_n$ by $\sqrt{2}a_0,\sqrt{2}a_n$, we obtain the same formula for $d_j$ as on p.~4053 of SV13:
\[
d_j=\frac{1}{\sqrt{N}}\left(\sqrt{2}a_0+(-1)^j\sqrt{2}a_n+2\sum_{k=1}^{n-1}a_k\cos\left(\frac{2\pi jk}{N}\right)\right).
\]

The fact that, both in SV13 and in the present symmetric setting, $d_j$ with $0\leq j\leq N/2$ are real-valued, uncorrelated, and standardized random variables imply that $\max_{0\leq j\leq N/2} |d_j|$ behaves asymptotically as $\sqrt{2\log n}$. This contrasts with Theorem \ref{T theorem}, where the corresponding maximum behaves as $\sqrt{\log n}$. This difference ultimately leads to the appearance of an ``extra'' factor of $\sqrt{2}$ in the denominator of $\|\mb{T}^{(s)}_{p\times n}\|/\sqrt{2 p\log n}$, which is absent from its non-symmetric version $\|\mb{T}_{p\times n}\|/\sqrt{p\log n}$ studied in Theorem \ref{T theorem}.

On the technical level, in the proof of the upper bound, Lemma \ref{cor: corollary 13} should be modified to require
\[
|d_{j_1}|^2+\dots+|d_{j_M}|^2\leq (1+\eta)2\log n,
\]
with an additional factor of $2$ on the right-hand side. Similarly, in the proof of the lower bound, all three conditions defining the event $A_i$ (immediately following \eqref{new equation}) should now involve $\sqrt{2\log n}$ instead of $\sqrt{\log n}$. With these changes, the proof proceeds with only minor modifications due to the fact that $N=2n$ rather than $N=p+n$ as in Theorem \ref{T theorem}.

\section{Values of $K_{1,n/p}$}\label{sec: K values}
In this section, we first estimate values of the constant $K_{1,n/p}$ for $p/n$ on the grid $0.01:0.01:1$. 

Lemma \ref{lem:I is close to K}, together with the identity \eqref{K reparameterization}, implies that
\begin{equation}\label{K brackets}
\sqrt{\frac{I_{p,n}^2}{p}-\frac{1}{3p}}\leq K_{1,n/p}\leq \frac{I_{p,n}}{\sqrt{p}},
\end{equation}
where, as shown in Lemma \ref{lem:I is close to K},
\[
I_{p,n}=\max\left\{\|\mathcal{P}_{p-1}\cdot\mathcal{P}_{n-1}\|_2:\|\mathcal{P}_{p-1}\|_2=1, \|\mathcal{P}_{n-1}\|_2=1\right\}.
\]

The latter maximization problem is studied in \cite{bunger11}. Corollary 2 of that paper implies that if the unit vectors $\mb{w}\in\mathbb{C}^n$ and $\mb{v}\in\mathbb{C}^p$ satisfy
\[
\mathcal{P}_{n-1}(z)=\mb{w}_{1}+\mb{w}_{2}z+\dots+\mb{w}_{n}z^{n-1}\qquad\text{and}\qquad \mathcal{P}_{p-1}(z)=\mb{v}_{1}+\mb{v}_{2}z+\dots+\mb{v}_{p}z^{p-1},
\]
where $\mathcal{P}_{n-1}$ and $\mathcal{P}_{p-1}$
are the extremal polynomials for the maximization problem, then $I_{p,n}$ equals  the common spectral norm of the $(p+n-1)\times p$ and $(p+n-1)\times n$  matrices
\[
M_\mb{w}:=\underbrace{\begin{pmatrix}
    \mb{w}_1&&&\\
    \vdots &\mb{w}_1&&\\
    \mb{w}_n&\vdots&\ddots&\\
    &\mb{w}_n&&\mb{w}_{1}\\
    &&\ddots&\vdots\\
    &&&\mb{w}_n
\end{pmatrix}}_{p}\qquad\text{and}\qquad M_\mb{v}:=\underbrace{\begin{pmatrix}
    \mb{v}_1&&&\\
    \vdots &\mb{v}_1&&\\
    \mb{v}_p&\vdots&\ddots&\\
    &\mb{v}_p&&\mb{v}_{1}\\
    &&\ddots&\vdots\\
    &&&\mb{v}_p
\end{pmatrix}}_{n}.
\]
Moreover, $\mb{w}$ is the right principal singular vector of $M_\mb{v}$, and  $\mb{v}$ is the right principal singular vector of $M_\mb{w}$.

Based on this result, we propose the following algorithm for numerically approximating \(K_{1,n/p}\):

\begin{enumerate}
    \item Initialize the algorithm by choosing large \(p\) and \(n\), and set the unit vectors \(\mathbf{w}^{(0)} \in \mathbb{C}^n\) and \(\mathbf{v}^{(0)} \in \mathbb{C}^p\) proportional to vectors of ones.
    \item Given \(\mathbf{w}^{(i)}\) and \(\mathbf{v}^{(i)}\), compute \(\mathbf{w}^{(i+1)}\) and \(\mathbf{v}^{(i+1)}\) as the right principal singular vectors of the matrices \(M_{\mathbf{v}^{(i)}}\) and \(M_{\mathbf{w}^{(i)}}\), respectively.
    \item Iterate until convergence. Estimate \(K_{1,n/p}\) by \(\|M_{\mathbf{w}^{(\text{end})}}\| / \sqrt{p}\).
\end{enumerate}

In practice, we stop the iteration when the difference between \(\|M_{\mathbf{w}^{(i+1)}}\|\) and \(\|M_{\mathbf{w}^{(i)}}\|\) becomes smaller than the tolerance level \(10^{-16}\). In our calculations, convergence was achieved within a few seconds on a standard laptop computer.

Table~\ref{tab: table_intro} reports the estimates of \(K_{1,n/p}\) obtained using this algorithm. We started with \(p = n = 1000\) and then proceeded to smaller ratios \(p/n\) by repeatedly decreasing \(p\) by 10, each time after convergence was achieved at the previous value of \(p\). For the new starting values of \(\mathbf{w}\) and \(\mathbf{v}\), we used the converged vector \(\mathbf{w}\) from the previous step and the right principal singular vector of the corresponding matrix \(M_{\mb{w}}\), respectively.

Assuming that convergence to \(I_{p,n}\) was achieved at each step, equation~\eqref{K brackets} describes the guaranteed accuracy of the approximations reported in Table~\ref{tab: table_intro}. This guaranteed accuracy varies with different values of \(p/n\) shown in the table. Indeed, the first-order Taylor approximation to the gap between the right- and left-hand sides of the inequalities in \eqref{K brackets} equals \(\frac{1}{6 I_{p,n} \sqrt{p}}\). This value is relatively small, \(0.0002\), for \(p/n = 1\), but much larger, approximately \(1/60 \approx 0.167\), for \(p/n = 0.01\) (which corresponded to \(p=10\), \(n=1000\) in our calculations).

The actual accuracy of the reported values might be much higher. For example, as mentioned in the introduction, the value of \(K_{1,1}\) is known up to 24 decimal places. The difference between this known value and the computed \(K_{1,1}\) (reported in Table~\ref{tab: table_intro} only up to three decimal places) was approximately \(-1.6 \times 10^{-7}\).

For \(p/n = 0.01\), the lower bound \eqref{the lower bound} on \(K_{1,n/p}\) equals 0.9983, while the theoretical upper bound is 1. Hence, the difference between the true value of \(K_{1,100}\) and our estimate \(1.000\) cannot exceed 0.002 in absolute value.

Of course, the above discussion of accuracy is valid only if convergence to \(I_{p,n}\) has been achieved. \cite{garcia69}, who computed \(K_{1,1}\) up to 24 decimal places, used a numerical algorithm similar to the one described above and proved its convergence for \(p = n\). We believe that their method can be extended to prove convergence for cases where \(p/n\) is relatively large. However, we leave a rigorous analysis of convergence for \(p/n \in (0,1]\) to future research.

\section{Monte Carlo and a higher order bound}\label{sec: MC}
In this section, we use Monte Carlo experiments to investigate the behavior of $\|\mathbf{T}_{p\times n}\|$ and $\|\mathbf{C}_{p\times n}\|$ for the case $p < n$, which is the focus of this paper. While the derived asymptotic limits provide a useful benchmark, numerical evidence shows a considerable level of dispersion across all aspect ratios and a bias that deteriorates markedly as $p/n$ decreases.

To investigate the underlying causes of this deterioration, we derive a higher-order asymptotic lower bound for $\|\mathbf{C}_{p\times n}\|$ in the Gaussian case. This bound shows that the asymptotic distribution of the scaled and centered norm stochastically dominates a shifted Gumbel distribution, with a shift that vanishes when $p=n$ but grows as $p/n$ decreases. This explains the deterioration of the bias and, at the same time, the dispersion of the shifted Gumbel distribution aligns reasonably well with the dispersion observed in the Monte Carlo experiments, especially when $p/n$ is not too small.

\paragraph{Monte Carlo: }
We simulate 10,000 independent copies of $\frac{\|\mathbf{T}_{p\times n}\|}{\sqrt{p\log n}}$ and $\frac{\|\mathbf{C}_{p\times n}\|}{\sqrt{p\log n}}$ with standard normal entries of the corresponding matrices,  $p=500$ and $n=\left\lfloor \frac{10}{k}p\right\rfloor$, $k=1,\dots,10$. Figure \ref{fig: histograms} shows the histograms corresponding to the case $k=5$ (so that $p/n=1/2$). The vertical red lines indicate the corresponding asymptotic limits: $K_{1,2}$ for the Toeplitz case, and $1$ for the circulant case. We see that most of the Monte Carlo draws are larger than the corresponding limit. The bias is larger in the Toeplitz case.

\begin{figure}
\centering
    \includegraphics[width=\linewidth]{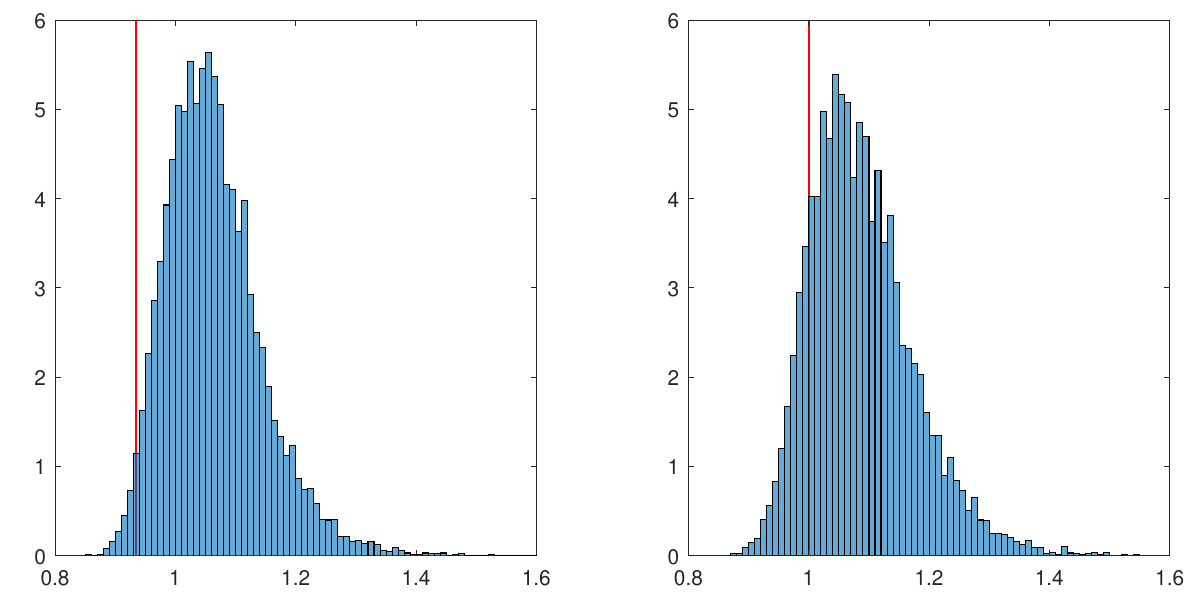}
    \caption{PDF-normalized histograms of 10,000 samples of 
$\frac{\|\mathbf{T}_{p\times n}\|}{\sqrt{p\log n}}$ (left) and 
$\frac{\|\mathbf{C}_{p\times n}\|}{\sqrt{p\log n}}$ (right), 
with standard normal entries, $p=500$, $n=1000$. 
Vertical red lines indicate the asymptotic limits: $K_{1,2}$ (Toeplitz) and $1$ (circulant).}
    \label{fig: histograms}
\end{figure}

Figure \ref{fig: summuries} presents results for all studied ratios $p/n\approx 0.1,0.2,\dots,1$. 
Solid lines show Monte Carlo medians as functions of $p/n$, dashed lines indicate the corresponding asymptotic limits, and dotted lines mark the 0.05 and 0.95 Monte Carlo quantiles. 
For $p/n>0.6$, the asymptotic approximation is consistent with the simulations: in the Toeplitz case, the asymptotic limit remains within the 90\% Monte Carlo interval, while in the circulant case it closely tracks the Monte Carlo median. The quality of the approximation substantially deteriorates for $p/n<0.6$.

\begin{figure}
\centering
    \includegraphics[width=\linewidth]{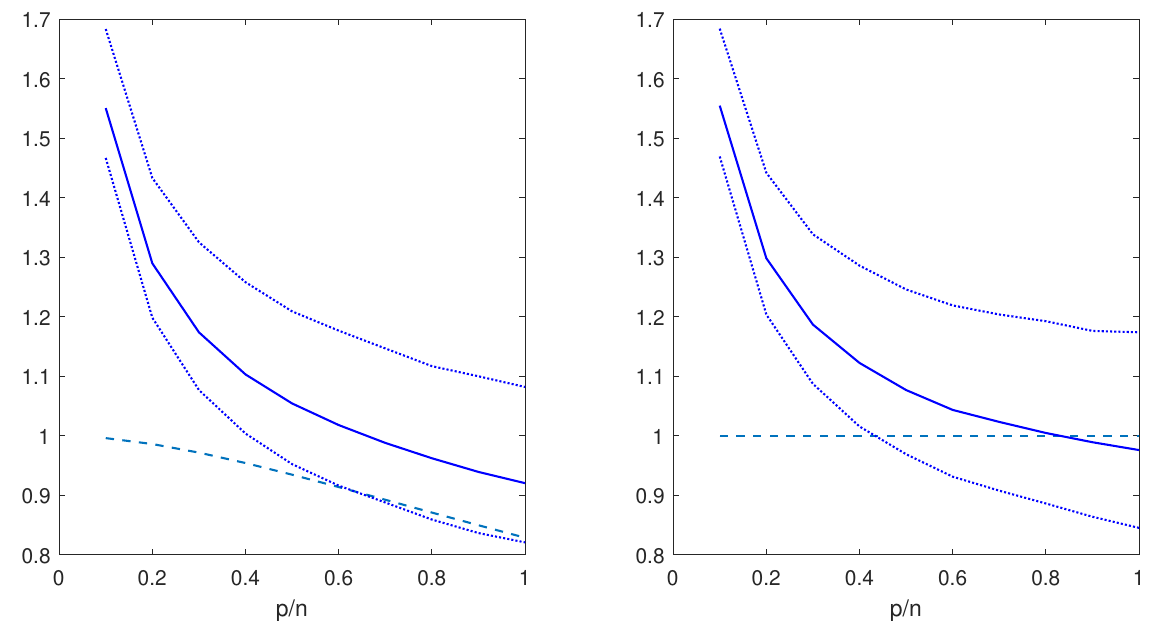}
    \caption{Solid lines show the Monte Carlo medians of $\frac{\|\mathbf{T}_{p\times n}\|}{\sqrt{p\log n}}$ (left) and 
$\frac{\|\mathbf{C}_{p\times n}\|}{\sqrt{p\log n}}$ (right) as functions of $p/n=0.1,0.2,\dots,1$. Dotted lines mark the 0.05 and 0.95 Monte Carlo quantiles. Dashed lines show the asymptotic limits: $K_{1,n/p}$ (Toeplitz) and $1$ (circulant). }
    \label{fig: summuries}
\end{figure}

\paragraph{Higher-order asymptotic lower bound: } 
To gain analytic insight into the deterioration of the finite-sample performance of the asymptotic approximation as $p/n$ decreases, we consider a rectangular Gaussian circulant matrix $\mb{C}_{p\times n}$ where $p$ grows proportionally with $n$, so that the ratio $p/n=c\in(0,1]$ remains fixed as $n$ increases, with $n$ even. 

Representation \eqref{PDP matrix} yields
\[
\|\mb{C}_{p\times n}\|^2=n\|\mb{P}_{p,n}\mb{D}\overline{\mb{D}}\mb{P}_{p,n}\|,
\]
where, in the Gaussian case, all diagonal entries $|d_j|^2$ of $\mb{D}\overline{\mb{D}}$ are standard exponential random variables, except $|d_0|^2$ and $|d_{n/2}|^2$, which are distributed according to a $\chi^2$ distribution with one degree of freedom. Moreover, the random variables $|d_j|^2$ with $0\leq j\leq n/2$ are independent, and $|d_{n-j}|^2=|d_j|^2$.

Our proof of Theorem \ref{C theorem} shows that the first-order behavior of $\|\mb{C}_{p\times n}\|^2$ is governed by values of $|d_j|^2$ of order $\log n$. Moreover, if $|d_j|^2 \approx \log n$, then
\[
\|\mb{C}_{p\times n}\|^2\approx n\overline{\mb{v}_j}^\top \mb{P}_{p,n}\mb{D}\overline{\mb{D}}\mb{P}_{p,n} \mb{v}_j=n\overline{\mb{v}_j}^\top \mb{D}\overline{\mb{D}} \mb{v}_j,
\]
where $\mb{v}_j$ denotes the normalized $j+1$-th column of the projection matrix $\mb{P}_{p,n}$.
This observation naturally suggests the following lower bound for $\|\mb{C}_{p\times n}\|^2$:
\begin{equation}\label{second-order bound}
B_{p, n}:=n\max_{j=0,\dots,n-1}\overline{\mb{v}_j}^\top \mb{D}\overline{\mb{D}} \mb{v}_j\leq \|\mb{C}_{p\times n}\|^2.
\end{equation}

Using the explicit formula \eqref{bound on the entries of P} for the entries of $\mb{P}_{p,n}$, we obtain
\begin{equation}\label{stationary moving average}
n\overline{\mb{v}_j}^\top \mb{D}\overline{\mb{D}} \mb{v}_j=p|d_j|^2+\sum_{\substack{i=0\\i\neq j}}^{n-1}\frac{1}{p}\frac{\sin^2\frac{\pi(j-i)p}{n}}{\sin^2\frac{\pi(j-i)}{n}}|d_{i}|^2.
\end{equation}
Applying this identity together with the relation $|d_{i}|^2=|d_{n-i}|^2$, we conclude that 
\[
\overline{\mb{v}_{n-j}}^\top\mb{D}\overline{\mb{D}}
\mb{v}_{n-j}=\overline{\mb{v}_{j}}^\top\mb{D}\overline{\mb{D}}
\mb{v}_{j},
\]
and hence, we can reduce the domain of the maximum in the definition of $B_{p,n}$:
\[
B_{p,n}=n\max_{j=0,\dots,n/2}\overline{\mb{v}_{j}}^\top\mb{D}\overline{\mb{D}}\mb{v}_j.
\]

Further, the right hand side of \eqref{stationary moving average}
is well approximated by a stationary moving average of i.i.d. standard exponential random variables. 
Consequently, the asymptotic distribution of $B_{p,n}$ can be analyzed using standard results on extrema of stationary sequences (see, e.g., Chapter 3 of \cite{leadbetter83}). This leads to the following lemma, whose proof is given in the Technical Appendix.
\begin{lemma}\label{lem: gumbel}
   Suppose that $n$ is even, and $p=p(n)$ is such that $p/n$ remains constant, $p/n=c\in(0,1]$, when $n\rightarrow\infty$. Then
   \[
   \frac{B_{p,n}}{p}-\log\frac{n}{2}\overset{d}\rightarrow \operatorname{Gumbel}(\theta_c,1)\qquad\text{as }n\rightarrow \infty,
   \]
   where 
   \begin{equation}\label{definition of theta_c}
\theta_c:=-2\sum_{j=1}^\infty\log\left(1-\left(\frac{\sin(c\pi j)}{c\pi j}\right)^2\right),
\end{equation}
$\operatorname{Gumbel}(\theta_c,1)$ is the extreme value distribution of type 1 with the cumulative distribution function $F(x)=\exp\{-e^{-(x-\theta_c)}\}$, and $\overset{d}\rightarrow$ denotes the convergence in distribution.
\end{lemma}

Lemma \ref{lem: gumbel}, together with the inequality $\|\mb{C}_{p\times n}\|^2\geq B_{p,n}$, yields the following second–order asymptotic bound. 
\begin{corollary}\label{cor: gumbel}
    Under the assumptions of Lemma \ref{lem: gumbel}, the Gaussian rectangular circulant matrix $\mb{C}_{p\times n}$ satisfies the asymptotic inequality
    \[
    \limsup_{n\rightarrow\infty}\mathbb{P}\left\{\,\frac{\|\mb{C}_{p\times n}\|^2}{p}-\log\frac{n}{2} \leq x \,\right\}
    \;\leq\; \exp\{-e^{-(x-\theta_c)}\}.
    \]
    Equivalently, the sequence
    \[
    \left\{\,\frac{\|\mb{C}_{p\times n}\|^2}{p}-\log\frac{n}{2}\,\right\}
    \]
    asymptotically stochastically dominates the shifted Gumbel distribution 
    $\operatorname{Gumbel}(\theta_c,\,1)$.
\end{corollary}

\begin{remark}
   For the square case, $c=p/n=1$, we have $\theta_c=0$, so there is no shift in the Gumbel law. Moreover, Theorem 1.2 in \cite{barrera22} implies that the sequence described in Corollary \ref{cor: gumbel} converges in distribution to the Gumbel law, even without assuming the Gaussianity. 
\end{remark}

Table \ref{tab: thatac} shows that $\theta_c$ grows rapidly as $c$ decreases, leading to substantial shifts of the Gumbel distribution.

\begin{table}[h]
    \centering
    \begin{tabular}{c|c c c c c c c c c c}
       c  & 0.1 & 0.2 & 0.3 & 0.4 & 0.5 & 0.6 & 0.7 & 0.8 & 0.9 & 1 \\
        $\theta_c$ & 18.42 & 7.05 & 3.61 & 2.06 & 1.23 & 0.75 & 0.45 & 0.25 & 0.11 & 0 
    \end{tabular}
    \caption{Values of $\theta_c$ on the grid $c=0.1:0.1:1$.}
    \label{tab: thatac}
\end{table}

\begin{figure}
\centering
    \includegraphics[width=0.5\linewidth]{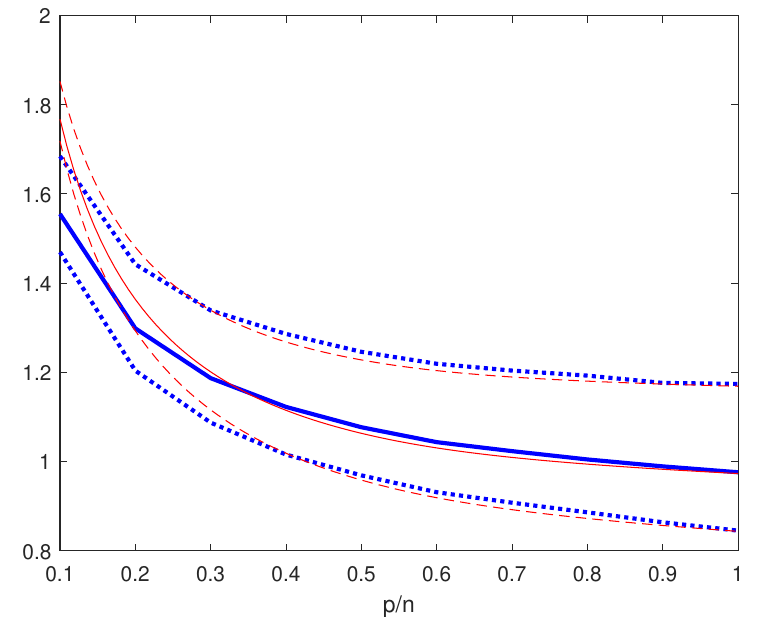}
    \caption{Monte Carlo quantiles of $\|\mb{C}_{p\times n}\|/\sqrt{p\log n}$ reported in Figure \ref{fig: summuries} (fat blue lines) superimposed with the corresponding quantiles of $\mathcal{G}_c$ (thin red lines). Solid lines- medians. Dotted lines - 5\% and 95\% quantiles.}
    \label{fig: bad quantiles}
\end{figure}

Figure \ref{fig: bad quantiles} overlays the Monte Carlo results for $\|\mb{C}_{p\times n}\|/\sqrt{p\log n}$, shown in the right panel of Figure \ref{fig: summuries}, with the quantiles of the distribution
\[
\mathcal{G}_c := \sqrt{\frac{\operatorname{Gumbel}(\theta_c,1)+\log (n/2)}{\log n}},
\]
which, by Corollary \ref{cor: gumbel}, provides asymptotic lower bounds for the corresponding quantiles of $\|\mb{C}_{p\times n}\|/\sqrt{p\log n}$. 

The theoretical quantiles are closely aligned with the Monte Carlo quantiles for moderate and large values of $c$. For smaller $c$, however, a noticeable discrepancy appears; this can be explained by the corresponding values of $\tfrac{n}{p\log n}$ being relatively large (1.17 for $c=0.1$ and 0.64 for $c=0.2$), which do not fit well with the asymptotic regime \eqref{new asymptotic regime}, where $\tfrac{n}{p\log n}$ is required to converge to zero. 

Overall, the observed alignment for larger $c$ is consistent with the conjecture
\begin{equation}\label{conjecture}
    \frac{\|\mb{C}_{p\times n}\|^2}{p}-\log\frac{n}{2}\overset{d}{\longrightarrow} \operatorname{Gumbel}(\theta_c,1),
\end{equation}
although a rigorous proof is currently lacking. Establishing this result would require additional work, including demonstrating a high concentration of the principal eigenvector of $\mb{C}_{p\times n}^\top \mb{C}_{p\times n}$. We leave this as an open problem for future study.


\section{Technical Appendix}

\subsection{Proof of Lemma \ref{lem:I is close to K}}
First, we show that 
\begin{equation}\label{GRR}
K_{p,n}^2\leq I_{p,n}^2.
\end{equation}
The argument follows the same idea as in the case $p=n$ on p.~812 of \cite{garcia69}.

Let $\Lambda_p\in \mathcal{F}_{p}$ and $\Lambda_n\in\mathcal{F}_{n}$, and define
\[
F(u):=\int e^{-\rmi ux}\overline{\Lambda_p\left(x\right)}\Lambda_n\left(x\right) \mathrm{d}x.
\]
By the Poisson summation formula (see e.g.~Sec.~13, Ch.~II of \cite{zygmund03}), 
\[
\sum_{k=-\infty}^{+\infty}F(2\pi k)=\sum_{k=-\infty}^{+\infty} \overline{\Lambda_p\left(k\right)}\Lambda_n\left(k\right).
\]
Observe that the sequences $\{\Lambda_p(k)\}_{k\in\mathbb{Z}}\in \mathcal{L}_{p-1}$ and $\{\Lambda_n(k)\}_{k\in\mathbb{Z}}\in \mathcal{L}_{n-1}$. Therefore,
\[
\sum_{k=-\infty}^{+\infty}F(2\pi k)\leq I_{p,n}^2.
\]
Next, since a product of positive definite functions is positive definite, the product $\overline{\Lambda_p(x)}\Lambda_n(x)$ is a positive definite function. By Bochner's theorem (see e.g.~p.~19 in \cite{rudin62}), its Fourier transform has non-negative values, which implies that $F(2\pi k)\geq 0$ for any $k\in\mathbb{Z}$. In particular,
\[
F(0)\leq \sum_{k=-\infty}^{+\infty}F(2\pi k)\leq I_{p,n}^2.
\]
Finally, observe that $F(0)=\int\overline{\Lambda_p(x)}\Lambda_n(x)\mathrm{d}x$. Taking the supremum over all admissible $\Lambda_p\in\mathcal{F}_p$ and $\Lambda_n\in\mathcal{F}_n$, we obtain
\[
\sup_{\Lambda_p\in\mathcal{F}_p,\Lambda_n\in\mathcal{F}_n} \int\overline{\Lambda_p(x)}\Lambda_n(x)\mathrm{d}x\leq I_{p,n}^2, 
\]
which establishes \eqref{GRR}.

It remains to show that 
\begin{equation}\label{belov bound}
K_{p,n}^2 \geq I_{p,n}^2 - \frac{1}{3}.
\end{equation}
To establish this bound, we build on a technique developed by \cite{belov13}, who introduced a novel approach to extremal problems involving positive definite functions. His method relies on extending positive definite sequences to positive definite piecewise linear functions, allowing the analysis of quantities such as
\[
\sup_{\Lambda \in \mathcal{F}_c} \int \Lambda^m(x)\, \mathrm{d}x.
\]
In the special case $m = 2$, this supremum coincides with $K_{c,c}$, which directly connects to our setup when $p = n$. Belov used this framework to establish a bound analogous to ours (see Theorem~5 in~\cite{belov13}), and we adapt his ideas to handle the general case $p \neq n$.

We begin by establishing an inequality analogous to Lemma 4.4 of~\cite{belov13}; see equation~\eqref{belov GRR} below. Let $\{\alpha_j\}\in\mathcal{L}_{p-1}, \{\gamma_j\}\in\mathcal{L}_{n-1}$, and define $\Lambda_{p}$ and $\Lambda_{n}$  as the piecewise linear functions interpolating the nodes $\{(j,\alpha_j)\}$ and $\{(j,\gamma_j)\}$, respectively. By Theorem 1 of \cite{belov13}, $\Lambda_{p}\in\mathcal{F}_{p}$ and $\Lambda_{n}\in\mathcal{F}_{n}$. Therefore, we have
\[
K_{p,n}^2\geq \int\overline{\Lambda_{p}(x)}\Lambda_{n}(x)\mathrm{d}x=\frac{1}{3}\sum_{j=-\infty}^\infty\left(\overline{\alpha_j}\gamma_j+\frac{1}{2}\overline{\alpha_j}\gamma_{j+1}+\frac{1}{2}\overline{\alpha_{j+1}}\gamma_j+\overline{\alpha_{j+1}}\gamma_{j+1}\right),
\]
where the final equality follows from evaluating the integral explicitly over each interval $x \in [j, j+1]$.

Taking the maximum over all such sequences, we obtain
\begin{equation}\label{belov GRR}
    K_{p,n}^2\geq J_{p,n}^2,
\end{equation}
where
\begin{equation}\label{J definition}
J_{p,n}^2:=\max_{\{\alpha_j\}\in\mathcal{L}_{p-1}, \{\gamma_j\}\in\mathcal{L}_{n-1}}\frac{1}{3}\sum _{j=-\infty}^{\infty}\left(\overline{\alpha_j}\gamma_j+\frac{1}{2}\overline{\alpha_j}\gamma_{j+1}+\frac{1}{2}\overline{\alpha_{j+1}}\gamma_j+\overline{\alpha_{j+1}}\gamma_{j+1}\right).
\end{equation}
Hence, to establish inequality \eqref{belov bound}, it suffices to show that
\begin{equation}\label{I-J comparison}
I_{p,n}^2-J_{p,n}^2\leq \frac{1}{3}.
\end{equation}

By definition \eqref{Ipn definition},
\begin{equation}\label{maximisation Ipn2}
I_{p,n}^2=\max_{\{\alpha_j\}\in \mathcal{L}_{p-1},\{\gamma_j\}\in\mathcal{L}_{n-1}}\sum_{\nu=-p+1}^{p-1} \overline{\alpha_\nu} \gamma_{\nu}.
\end{equation}
Suppose that there exist admissible extremal sequences $\{\alpha_j\}$ and $\{\gamma_j\}$ solving this maximization problem, such that all their elements are real and
\begin{equation}\label{monotonicity}
1=\alpha_0\geq\cdots\geq \alpha_{p-1}\geq 0,\qquad 1=\gamma_0\geq\cdots\geq \gamma_{n-1}\geq 0.
\end{equation}
The existence of such real-valued, symmetric, monotone extremal sequences is established in Lemma \ref{lem: theorem 5.1} below. Then we have
\begin{eqnarray*}
    I^2_{p,n}&=&\sum_{j=-\infty}^{\infty}\alpha_j\gamma_j,\\
    J^2_{p,n}&\geq& \frac{1}{3}\sum _{j=-\infty}^{\infty}\left(\alpha_j\gamma_j+\frac{1}{2}\alpha_j\gamma_{j+1}+\frac{1}{2}\alpha_{j+1}\gamma_j+\alpha_{j+1}\gamma_{j+1}\right).
\end{eqnarray*}
Taking the difference and using the symmetry $\alpha_{j}=\alpha_{-j}$ and  $\gamma_{j}=\gamma_{-j}$, we obtain
\begin{eqnarray*}
I^2_{p,n}-J^2_{p,n}&\leq& \frac{1}{3}\sum_{j=-\infty}^{\infty}\left(\alpha_j\gamma_j-\frac{1}{2}\alpha_j\gamma_{j+1}-\frac{1}{2}\alpha_{j+1}\gamma_j\right)\\ &=&\frac{1}{3}\alpha_0\gamma_0+\frac{1}{3}\sum_{j=0}^\infty\left(2\alpha_{j+1}\gamma_{j+1}-\alpha_j\gamma_{j+1}-\gamma_j\alpha_{j+1}\right)\\
    &=&\frac{1}{3}\alpha_0\gamma_0-\frac{1}{3}\sum_{j=0}^\infty\left((\alpha_j-\alpha_{j+1})\gamma_{j+1}+(\gamma_j-\gamma_{j+1})\alpha_{j+1}\right)\leq \frac{1}{3},
\end{eqnarray*}
since all terms in the final sum are non-negative.
This completes the proof of \eqref{I-J comparison}, adapting the argument of Lemma 5.9 in \cite{belov13}.

Finally, the claimed existence of the extremal sequences that are real-valued, symmetric, and monotone is established in the following lemma. The proof closely follows the reasoning of Theorem 5.1 in \cite{belov13} and is therefore deferred to the SM.

\begin{lemma}\label{lem: theorem 5.1} 
There exist real-valued, symmetric sequences $\{\alpha_j\}\in\mathcal{L}_{p-1}$ and $\{\gamma_j\}\in\mathcal{L}_{n-1}$ that attain the maximum in \eqref{maximisation Ipn2} and satisfy the monotonicity condition \eqref{monotonicity}.
\end{lemma}

\subsection{Proof of Lemma \ref{lem: link between Pi and K}}
By the identity \eqref{K reparameterization},
\[
\sqrt{p/N}K_{1,n/p}=K_{p/N,n/N}.
\]
Hence, it suffices to prove that
\[
K_{p/N,n/N}=\sup_{\mathbf{c}_1,\mathbf{c}_2\in\ell^2(\mathbb{C}):\|\mathbf{c}_1\|=\|\mathbf{c}_2\|=1}\|\Pi_{p/N}\mathbf{c}_1\odot \Pi_{n/N}\mathbf{c}_2\|.
\]
Recall that, by definition,
\begin{equation}\label{K again}
K_{p/N,n/N}^2=\sup_{\Lambda_1\in \mathcal{F}_{p/N},\Lambda_2\in\mathcal{F}_{n/N}}\int_{-p/N}^{p/N} \overline{\Lambda_1(x)}\Lambda_2(x) \mathrm{d}x.
\end{equation}

By Lemma \ref{lem: Boas and Kac}, any $\Lambda_1\in \mathcal{F}_{p/N},\Lambda_2\in\mathcal{F}_{n/N}$ can be represented in the form 
\[
\Lambda_1(x)=\int_{-\frac{p}{2N}}^{\frac{p}{2N}}\overline{\phi_1(t)}\phi_1(x+t)\mathrm{d}t,\qquad  \Lambda_2(x)=\int_{-\frac{n}{2N}}^{\frac{n}{2N}}\overline{\phi_2(t)}\phi_2(x+t)\mathrm{d}t,
\]
where $\phi_1\in \mathcal{G}_{p/(2N)}$ and $\phi_2\in \mathcal{G}_{n/(2N)}$. 
We will assume, without loss of generality, that these functions equal zero at the support boundary, that is, $\phi_1(t)=0$ for $t=p/(2N)$, and $\phi_2(t)=0$ for $t=n/(2N)$. Note that, since $N=p+n$, the supports of both $\phi_1$ and $\phi_2$ belong to $[-1/2,1/2]$. 

Now let $\tilde{\phi}_1(t)$ and $\tilde{\phi}_2(t)$ be periodic functions with period $1$, equal to $\phi_1(t)$ and $\phi_2(t)$, respectively, for $t\in [-1/2,1/2]$. Further, define
\begin{eqnarray*}
\tilde{\Lambda}_1(x)&:=&\int_{-1/2}^{1/2}\overline{\tilde{\phi}_1(t)}\tilde{\phi}_1(x+t)\mathrm{d}t=(\tilde{\phi}_1^\ast\star\tilde{\phi}_1)(x)\quad \text{and}\\
\tilde{\Lambda}_2(x)&:=&\int_{-1/2}^{1/2}\overline{\tilde{\phi}_2(t)}\tilde{\phi}_2(x+t)\mathrm{d}t=(\tilde{\phi}_2^\ast\star\tilde{\phi}_2)(x).
\end{eqnarray*}

Let us show that
\begin{eqnarray}
    \label{lambda equality} \Lambda_j(x)=\tilde{\Lambda}_j(x)\text{ for all }x\in[-p/N,p/N],\text{ and }j=1,2. 
\end{eqnarray}
It suffices to prove that, for all $t,x\in\mathbb{R}$ such that $|t|\leq 1/2$ and $|x|\leq p/N$, the integrands in the integrals defining $\Lambda_j$ and $\tilde{\Lambda}_j$ coincide:
\begin{equation}\label{equality of integrands}
\overline{\phi_j(t)}\phi_j(x+t)=\overline{\tilde{\phi}_j(t)}\tilde{\phi}_j(x+t),\quad j=1,2.
\end{equation}

First, consider the case $j=1$. For $p/(2N)\leq |t|\leq 1/2$ and $j=1$, the identity \eqref{equality of integrands} holds because $\phi_1(t)=\tilde{\phi}_1(t)=0$.
For $|t|< p/(2N)$, the common value of $\phi_1(t)=\tilde{\phi}_1(t)$ is not zero, in general. Therefore, for such $t$, we need to show that $\phi_1(x+t)=\tilde{\phi}_1(x+t)$.

Observe that, for $|x+t|\leq 1/2$, the latter equality holds by definition. For $1/2<|x+t|\leq  1-p/(2N)$, it holds because, then, both $\phi_1(x+t)=0$ and $\tilde{\phi}_1(x+t)=0$. 

Therefore, it suffices to show that, for all $t,x\in\mathbb{R}$ such that $|t|\leq p/(2N)$ and $|x|\leq p/N$, we must have $|x+t|\leq 1-p/(2N)$. This follows immediately from the triangle inequality: 
\[
|x+t|\leq|x|+|t|\leq p/N+p/(2N)\leq 1-p/(2N),
\]
where the final inequality holds since $p/N=p/(p+n)\leq 1/2$.

Turning to the case $j=2$, and arguing similarly to the case $j=1$, we now reduce the problem to verifying that for all $t,x\in\mathbb{R}$ such that $|t|\leq n/(2N)$ and $|x|\leq p/N$, we must have $|x+t|\leq 1-n/(2N)$. Applying the triangle inequality, 
\[
|x+t|\leq|x|+|t|\leq p/N+n/(2N)= 1-n/(2N),
\]
since $N=p+n$. 
This finishes our proof of \eqref{equality of integrands}, and hence \eqref{lambda equality} holds, as claimed. 

Note that if $N$ had been larger than $p+n$, we would have had the inequality $p/N+p/(2N)<1-p/(2N)$, which would still have been sufficient for our purposes. This justifies the validity of Remark \ref{remark: deficiency}.

The identities \eqref{lambda equality} and \eqref{K again} yield the following representation:
\[
K_{p/N,n/N}^2=\sup_{\phi_1\in \mathcal{G}_{p/(2N)},\phi_2\in\mathcal{G}_{n/(2N)}}\int_{-p/N}^{p/N} \overline{(\tilde{\phi}_1^\ast\star\tilde{\phi}_1)(x)}(\tilde{\phi}_2^\ast\star\tilde{\phi}_2)(x) \mathrm{d}x.
\]
Let $\mathcal{G}_{a/2}^0$ be the set of complex-valued functions $\phi\in L^2(\mathbb{R})$, supported on $[-a/2,a/2]$ and satisfying $\int|\phi(t)|^2\mathrm{d}t\leq 1$. The difference from $\mathcal{G}_{a/2}$ is that the functions in $\mathcal{G}_{a/2}^0$ are allowed to have norms smaller than $1$. 
Then, the domain of the supremum in the above expression can be extended, yielding:
\[
K_{p/N,n/N}^2=\sup_{\phi_1\in \mathcal{G}_{p/(2N)}^0,\phi_2\in\mathcal{G}_{n/(2N)}^0}\int_{-p/N}^{p/N} \overline{(\tilde{\phi}_1^\ast\star\tilde{\phi}_1)(x)}(\tilde{\phi}_2^\ast\star\tilde{\phi}_2)(x) \mathrm{d}x.
\]
Note that for any $\delta<1$, the class $\mathcal{G}_{\delta/2}^0$ coincides with the set $\mathcal{G}_{1/2}\cdot \mathbf{1}_{\delta/2}$, that is, the pointwise products of functions from $\mathcal{G}_{1/2}$ and the indicator function $\mathbf{1}_{\delta/2}$ of the interval $[-\delta/2,\delta/2]$. Therefore, we have:
\begin{eqnarray}
\label{intermediate definition of C}
K_{p/N,n/N}^2=\sup_{\psi_1\in \mathcal{G}_{1/2},\psi_2\in\mathcal{G}_{1/2}}\int_{-p/N}^{p/N} \overline{\left([\tilde{\psi}_1\tilde{\mathbf{1}}_{\frac{p}{2N}}]^\ast\star \tilde{\psi}_1\tilde{\mathbf{1}}_{\frac{p}{2N}}\right)(x)}\left([\tilde{\psi}_2\tilde{\mathbf{1}}_{\frac{n}{2N}}]^\ast\star \tilde{\psi}_2\tilde{\mathbf{1}}_{\frac{n}{2N}}\right)(x) \mathrm{d}x,
\end{eqnarray}
where all tilded functions denote the period-1 extensions of their respective functions.

Now consider the $k$-th Fourier coefficients, $k\in\mathbb{Z}$, of functions $\psi_1, \psi_2,$  and their truncated versions $\psi_1\mathbf{1}_{p/(2N)}$, $\psi_2\mathbf{1}_{n/(2N)}$:
\begin{align*}
    c_{k}^{(1)}&:=\int_{-1/2}^{1/2}\psi_1(x)e^{-2\pi\rmi k x}\mathrm{d}x, &  c_{k}^{(2)}&:=\int_{-1/2}^{1/2}\psi_2(x)e^{-2\pi\rmi k x}\mathrm{d}x,\\
    \check{c}_{k}^{(1)}&:=\int_{-p/(2N)}^{p/(2N)}\psi_1(x)e^{-2\pi\rmi k x}\mathrm{d}x,  & \check{c}_{k}^{(2)}&:=\int_{-n/(2N)}^{n/(2N)}\psi_2(x)e^{-2\pi\rmi k x}\mathrm{d}x.
\end{align*}
These are, of course, also the Fourier coefficients of the period-1 extensions $\tilde{\psi}_1, \tilde{\psi}_2, \tilde{\psi}_1\tilde{\mathbf{1}}_{p/(2N)},$ and $\tilde{\psi}_2\tilde{\mathbf{1}}_{n/(2N)}$, respectively.
Therefore, the Fourier coefficients of the convolution $[\tilde{\psi}_1\tilde{\mathbf{1}}_{\frac{p}{2N}}]^\ast\star \tilde{\psi}_1\tilde{\mathbf{1}}_{\frac{p}{2N}}$ are given by  $|\check{c}_k^{(1)}|^2$, and those of  $[\tilde{\psi}_2\tilde{\mathbf{1}}_{\frac{n}{2N}}]^\ast\star \tilde{\psi}_2\tilde{\mathbf{1}}_{\frac{n}{2N}}$ by $|\check{c}_k^{(2)}|^2$ for all $k\in\mathbb{Z}$.
By Plancherel's theorem,
\begin{equation}\label{Plancherel identity}
\int_{-p/N}^{p/N} \overline{\left([\tilde{\psi}_1\tilde{\mathbf{1}}_{\frac{p}{2N}}]^\ast\star \tilde{\psi}_1\tilde{\mathbf{1}}_{\frac{p}{2N}}\right)(x)}\left([\tilde{\psi}_2\tilde{\mathbf{1}}_{\frac{n}{2N}}]^\ast\star \tilde{\psi}_2\tilde{\mathbf{1}}_{\frac{n}{2N}}\right)(x) \mathrm{d}x=\sum_{k=-\infty}^\infty|\check{c}_k^{(1)}\check{c}_k^{(2)}|^2.
\end{equation}
This equality allows us to link $K_{p/N,n/N}^2$ to a maximization problem in $\ell^2(\mathbb{C})$ space as follows.

For $j=1,2$, define the vectors in $\ell^2(\mathbb{C})$:
\begin{eqnarray*}
\mb{c}_{j}&:=&(\dots,c_{-2}^{(j)},c_{-1}^{(j)},c_{0}^{(j)},c_{1}^{(j)},c_{2}^{(j)},\dots),\text{ and}\\
\check{\mb{c}}_{j}&:=&(\dots,\check{c}_{-2}^{(j)},\check{c}_{-1}^{(j)},\check{c}_{0}^{(j)},\check{c}_{1}^{(j)},\check{c}_{2}^{(j)},\dots).
\end{eqnarray*}
Observe that 
\begin{equation}\label{check identity}
\check{\mb{c}}_{1}=\tilde{\Pi}_{p/N}\mb{c}_{1},\qquad \check{\mb{c}}_{2}=\tilde{\Pi}_{n/N}\mb{c}_{2},
\end{equation}
where operator $\tilde{\Pi}_{r/N}:\ell^2(\mathbb{C})\rightarrow\ell^2(\mathbb{C})$ is defined as follows (cf. \eqref{definition of limiting operator}):
\[
\tilde{\Pi}_{r/N}: \ell^2(\mathbb{C})\overset{\varphi}{\rightarrow}L^2(\mathbb{T})\overset{\times \mathbf{1}_{r/(2N)}}{\longrightarrow}L^2(\mathbb{T})\overset{\varphi^{-1}}{\rightarrow}\ell^2(\mathbb{C}).
\]
In the notations of \eqref{definition of limiting operator},
\[
\tilde{\Pi}_{q/N}=U^{-1}_{q/N}\Pi_{q/N}U_{q/N}.
\]
In particular, operators $\tilde{\Pi}_{q/N}$ and $\Pi_{q/N}$ are unitarily equivalent.

As $\psi_j$ spans $\mathcal{G}_{1/2}$, $\mb{c}_{j}$ spans the unit sphere in $\ell^2(\mathbb{C})$ (recall that $\|\psi_j\|_2=1$ by definition). Therefore, identities \eqref{Plancherel identity} and \eqref{check identity} imply that \eqref{intermediate definition of C} is equivalent to 
\begin{eqnarray}\notag
K_{p/N,n/N}^2&=&\sup_{\mathbf{c}_1,\mathbf{c}_2\in\ell^2(\mathbb{C}):\|\mathbf{c}_1\|=\|\mathbf{c}_2\|=1}\|\tilde{\Pi}_{p/N}\mathbf{c}_1\odot\tilde{\Pi}_{n/N}\mathbf{c}_2\|^2\\\label{l2 formulation of C}
&=&\sup_{\mathbf{c}_1,\mathbf{c}_2\in\ell^2(\mathbb{C}):\|\mathbf{c}_1\|=\|\mathbf{c}_2\|=1}\|\Pi_{p/N}\mathbf{c}_1\odot \Pi_{n/N}\mathbf{c}_2\|^2.
\end{eqnarray}
This concludes the proof of Lemma \ref{lem: link between Pi and K}.

\subsection{Proof of Lemma \ref{lem: grid}}
First, we establish the following auxiliary result:
\begin{lemma}\label{lem: K continuity}
    For any $a,b,\alpha,\beta\in\mathbb{R}$ such that $1\geq a\geq\alpha>0$ and $1\geq b\geq\beta b>0$, we have:
    \[
    \left|K_{a,b}-K_{\alpha,\beta}\right|\leq 1-\sqrt{c},\qquad c:=\min\{\alpha/a,\beta/b\}.
    \]
\end{lemma}
\begin{proof}
    By definition,
    \[
    K_{a,b}=\sup_{\Lambda_1 \in \mathcal{F}_a,\; \Lambda_2 \in \mathcal{F}_b} \left(\int \overline{\Lambda_1(x)}\Lambda_2(x)\,\mathrm{d}x\right)^{1/2}.
    \]
    For any $\Lambda_1 \in \mathcal{F}_a,\; \Lambda_2 \in \mathcal{F}_b$, define
    \[
    \widetilde{\Lambda}_1(x):=\Lambda_1(x/c),\qquad \widetilde{\Lambda}_2(x):=\Lambda_2(x/c).
    \]
    By construction, $\widetilde{\Lambda}_1 \in \mathcal{F}_\alpha,\; \widetilde{\Lambda}_2 \in \mathcal{F}_\beta$, and therefore,
    \[
    K_{\alpha,\beta}\geq \left(\int \overline{\widetilde{\Lambda}_1(x)}\widetilde{\Lambda}_2(x)\,\mathrm{d}x\right)^{1/2}=\sqrt{c}\left(\int \overline{\Lambda_1(x)}\Lambda_2(x)\,\mathrm{d}x\right)^{1/2}.
    \]
    Since this inequality holds for any $\Lambda_1 \in \mathcal{F}_a,\; \Lambda_2 \in \mathcal{F}_b$, we obtain
    $K_{\alpha,\beta}\geq \sqrt{c}K_{a,b}$.    On the other hand, since $\mathcal{F}_\alpha\subset \mathcal{F}_a$ and $\mathcal{F}_\beta\subset \mathcal{F}_b$, we have
    $
    K_{\alpha,\beta}\leq K_{a,b}$.
        Combining the latter two inequalities, we arrive at
    \[
    |K_{a,b}-K_{\alpha,\beta}|\leq (1-\sqrt{c})K_{a,b}.
    \]
    It remains to note that
    \[
    K_{a,b}\leq \sup_{\Lambda \in \mathcal{F}_1} \left( \int \Lambda(x)\, \mathrm{d}x \right)^{1/2}=1,
\]
with the final equality established in \cite{boas45}.   
\end{proof}

Now we turn to establishing \eqref{grid bound}. Multiplying both parts of that inequality by $\sqrt{p/N}$ and using identity \eqref{K reparameterization}, we see that \eqref{grid bound} is equivalent to
\begin{equation}\label{another form of grid bound}
\|\Pi_{p/N}^{[k]}\operatorname{diag}\{\mb{u}^{(t)}\}\Pi_{n/N}^{[k]}\|>K_{p/N,n/N}-\sqrt{\frac{p}{N}}\frac{\tau}{2}.
\end{equation}
Note that, for $p/n>\tau$, we have 
\begin{equation}\label{easier margin}
\sqrt{\frac{p}{N}}\frac{\tau}{2}\geq \sqrt{\frac{p}{n}}\frac{\tau}{2}>\frac{\tau^{3/2}}{2}.
\end{equation}

Fix an integer $T>3/\tau^3$ and suppose that  $p/n\in (\tau_{t-1},\tau_t]$ for some $t\in\{1,\dots,T\}$. Then we have
\[
\frac{p}{N}=\frac{p/n}{1+p/n}>\frac{\tau_{t-1}}{1+\tau_{t-1}}=:\alpha_t,\qquad \frac{n}{N}=\frac{1}{1+p/n}\geq \frac{1}{1+\tau_t}=:\beta_t.
\]
By Lemma \ref{lem: K continuity},
\[
|K_{p/N,n/N}-K_{\alpha_t,\beta_t}|\leq 1-\sqrt{\min\left\{\frac{\alpha_t}{p/N},\frac{\beta_t}{n/N}\right\}}.
\]
On the other hand,
\[
\frac{\alpha_t}{p/N}\geq\frac{\alpha_t}{\alpha_{t+1}}\geq\frac{\tau_{t-1}}{\tau_t}\geq 1-\frac{1}{T\tau},\qquad \frac{\beta_t}{n/N}\geq\frac{\beta_t}{\beta_{t-1}}\geq 1-\frac{1}{T}\geq 1-\frac{1}{T\tau}.
\]
Therefore,
\begin{equation}\label{tolerance}
|K_{p/N,n/N}-K_{\alpha_t,\beta_t}|\leq 1-\sqrt{1-\frac{1}{T\tau}}\leq \frac{\tau^{3/2}}{4},
\end{equation}
where the last inequality holds for $T>3/\tau^3$ and $\tau\in(0,1)$ (see SM for an elementary proof). 

Using \eqref{tolerance} and \eqref{easier margin} in \eqref{another form of grid bound}, we conclude that it is sufficient to establish 
\[
\|\Pi_{p/N}^{[k]}\operatorname{diag}\{\mb{u}^{(t)}\}\Pi_{n/N}^{[k]}\|>K_{\alpha_t,\beta_t}-\tau^{3/2}/4
\]
for each $t\in\{1,\dots,T\}$ and all $p/n\in(\tau_{t-1},\tau_t]$.

By Remark \ref{remark: deficiency},
\[
K_{\alpha_t,\beta_t}=\sup_{\mathbf{c}_1,\mathbf{c}_2\in\ell^2(\mathbb{C}):\|\mathbf{c}_1\|=\|\mathbf{c}_2\|=1}\|\Pi_{\alpha_t}\mathbf{c}_1\odot \Pi_{\beta_t}\mathbf{c}_2\|.
\]
Hence, there exist unit vectors $\mb{c}_1^{(t)},\mb{c}_2^{(t)}\in\ell^2(\mathbb{C})$ such that
\[
\left|K_{\alpha_t,\beta_t}-\|\Pi_{\alpha_t}\mathbf{c}_1^{(t)}\odot \Pi_{\beta_t}\mathbf{c}_2^{(t)}\|\right|<\tau^{3/2}/8.
\]

Recall the definition \eqref{definition of limiting operator} of the projection operator $\Pi_{r/N}$, and the related diagonal unitary operator $U_{r/N}$. Since $\alpha_t\leq p/N$, the image of the projection $U_{\alpha_t}\Pi_{\alpha_t}U_{\alpha_t}^{-1}$ is a subspace of the image of the projection $U_{p/N}\Pi_{p/N}U_{p/N}^{-1}$. Consequently,
\[
\Pi_{\alpha_t}\mb{c}_1^{(t)}=U_{\alpha_t}^{-1}U_{\alpha_t}\Pi_{\alpha_t}\mb{c}_1^{(t)}=U_{\alpha_t}^{-1}U_{p/N}\Pi_{p/N}U^{-1}_{p/N}U_{\alpha_t}\Pi_{\alpha_t}\mb{c}_1^{(t)}=:U_{\alpha_t}^{-1}U_{p/N}\Pi_{p/N}\widetilde{\mb{c}}_1^{(t)}.
\]
Similarly,
\[
\Pi_{\beta_t}\mb{c}_2^{(t)}=U_{\beta_t}^{-1}U_{n/N}\Pi_{n/N}\widetilde{\mb{c}}_2^{(t)},\qquad \widetilde{\mb{c}}_2^{(t)}:=U^{-1}_{n/N}U_{\beta_t}\Pi_{\beta_t}\mb{c}_2^{(t)}.
\]
Since $U_{\alpha_t}^{-1}U_{p/N}$ and $U_{\beta_t}^{-1}U_{n/N}$ are unitary and diagonal, we have
\[
\|\Pi_{\alpha_t}\mathbf{c}_1^{(t)}\odot \Pi_{\beta_t}\mathbf{c}_2^{(t)}\|=\|\Pi_{p/N}\widetilde{\mathbf{c}}_1^{(t)}\odot \Pi_{n/N}\widetilde{\mathbf{c}}_2^{(t)}\|.
\]
Note also that $\|\widetilde{\mb{c}}_1^{(t)}\|\leq 1$ and $\|\widetilde{\mb{c}}_2^{(t)}\|\leq 1$. 

Summing up, we have shown that there exist vectors $\widetilde{\mb{c}}_1^{(t)},\widetilde{\mb{c}}_2^{(t)}$ from the unit ball in $\ell^2(\mathbb{C})$  such that
\[
\left|K_{\alpha_t,\beta_t}-\|\Pi_{p/N}\widetilde{\mathbf{c}}_1^{(t)}\odot \Pi_{n/N}\widetilde{\mathbf{c}}_2^{(t)}\|\right|<\tau^{3/2}/8.
\]
Thus, we have reduced the problem of establishing \eqref{grid bound} to showing that there exist $k=k(\tau)$ and $\mb{u}^{(1)},\dots,\mb{u}^{(T)}$ from the interior of the unit ball in $\mathbb{C}^k$ such that
\begin{equation}\label{almost final}
\|\Pi_{p/N}^{[k]}\operatorname{diag}\{\mb{u}^{(t)}\}\Pi_{n/N}^{[k]}\|>\|\Pi_{p/N}\widetilde{\mathbf{c}}_1^{(t)}\odot \Pi_{n/N}\widetilde{\mathbf{c}}_2^{(t)}\|-\tau^{3/2}/8
\end{equation}
for each $t\in\{1,\dots,T\}$ and all $p/n\in(\tau_{t-1},\tau_t]$.

A proof of the following result parallels the proof of Lemma 16 of SV13. We defer it to SM.
\begin{lemma}\label{lem: lem16}
    There exists a finite integer $k=k(\tau)$ and vectors $\mb{x}_1^{(t)},\mb{x}_2^{(t)}$ from the unit ball in $\mathbb{C}^{k}$, such that for each $t\in\{1,\dots,T(\tau)\}$ and all $p/n\in(\tau_{t-1},\tau_t]$,
    \[
    \left|\|\Pi_{p/N}\widetilde{\mathbf{c}}_1^{(t)}\odot \Pi_{n/N}\widetilde{\mathbf{c}}_2^{(t)}\|-\|\Pi_{p/N}^{[k]}\mathbf{x}_1^{(t)}\odot \Pi_{n/N}^{[k]}\mathbf{x}_2^{(t)}\|\right|<\tau^{3/2}/16.
    \]
\end{lemma}

Now denote vectors $\overline{\Pi_{p/N}^{[k]}\mathbf{x}_1^{(t)}}\odot \Pi_{n/N}^{[k]}\mathbf{x}_2^{(t)}$ as $\mb{v}^{(t)}$, and define
$\widetilde{\mb{u}}^{(t)}:=\overline{\mb{v}^{(t)}}/\|\mb{v}^{(t)}\|$.
We have:
\begin{eqnarray*}
\|\Pi_{p/N}^{[k]}\mathbf{x}_1^{(t)}\odot \Pi_{n/N}^{[k]}\mathbf{x}_2^{(t)}\|&=&\|\overline{\Pi_{p/N}^{[k]}\mathbf{x}_1^{(t)}}\odot \Pi_{n/N}^{[k]}\mathbf{x}_2^{(t)}\|\\
&=&\left(\overline{\Pi_{p/N}^{[k]}\mathbf{x}_1^{(t)}}\right)^\top\operatorname{diag}\{\widetilde{\mb{u}}^{(t)}\}\left(\Pi_{n/N}^{[k]}\mathbf{x}_2^{(t)}\right)\\
&=&\left(\overline{\mathbf{x}_1^{(t)}}\right)^\top\Pi_{p/N}^{[k]}\operatorname{diag}\{\widetilde{\mb{u}}^{(t)}\}\Pi_{n/N}^{[k]}\left(\mathbf{x}_2^{(t)}\right)\\
&\leq& \|\Pi_{p/N}^{[k]}\operatorname{diag}\widetilde{\mb{u}}^{(t)}\}\Pi_{n/N}^{[k]}\|.
\end{eqnarray*}
This inequality, together with Lemma \ref{lem: lem16}, yields
\[
\|\Pi_{p/N}^{[k]}\operatorname{diag}\{\widetilde{\mb{u}}^{(t)}\}\Pi_{n/N}^{[k]}\|>\|\Pi_{p/N}\widetilde{\mathbf{c}}_1^{(t)}\odot \Pi_{n/N}\widetilde{\mathbf{c}}_2^{(t)}\|-\tau^{3/2}/16,
\]
where $\widetilde{\mb{u}}^{(t)}$, $t=1,\dots,T$ are unit vectors. Defining 
\[
\mb{u}^{(t)}:=(1-\varepsilon)\widetilde{\mb{u}}^{(t)},\qquad \varepsilon\in(0,1),
\]
ensures that $\mb{u}^{(t)}$ lie in the interior of the unit ball. Choosing $\varepsilon$ sufficiently small yields \eqref{almost final}, completing the proof of \eqref{grid bound}.

It remains to prove \eqref{first grid interval}. Note that
\begin{equation}\label{explicit}
\frac{\|\Pi_{p/N}^{[2L_n+1]}\operatorname{diag}\{\mb{u}^{(0)}\}\Pi_{n/N}^{[2L_n+1]}\|}{1-\tau/15}=\|\mb{y}_{p/N}(\overline{\mb{y}_{n/N}})^\top\|=\|\mb{y}_{p/N}\|\|\mb{y}_{n/N}\|,
\end{equation}
where $\mb{y}_{r/N}\in\mathbb{C}^{2L_n+1}$ denotes the $L_n+1$-th column of matrix $\Pi_{r/N}^{[2L_n+1]}$.  
Using the explicit formula \eqref{entries of Pi} for the entries of $\Pi_{r/N}$, we obtain:
\begin{equation}\label{Fourier calculation}
\|\mb{y}_{r/N}\|^2=\left(\frac{r}{N}\right)^2+2\sum_{j=1}^{L_n}\left(\frac{\sin\frac{\pi j r}{N}}{\pi j}\right)^2=\frac{r}{N}-2\sum_{j=L_n+1}^{\infty}\left(\frac{\sin\frac{\pi j r}{N}}{\pi j}\right)^2.
\end{equation}
The latter equality follows from expanding the $2\pi$-periodic function $g:\mathbb{R}\rightarrow\mathbb{R}$, defined by $g(x):=\pi|x|-x^2$ for $x\in[-\pi,\pi]$, into its Fourier series, and evaluating it at $x=\pi r/N$ (see the SM for an elementary derivation).

From \eqref{Fourier calculation} we have
\[
\frac{r}{N}\geq \|\mb{y}_{r/N}\|^2\geq \frac{r}{N}-\frac{2}{\pi^2 L_n},
\]
In particular,
\[
\|\mb{y}_{p/N}\|\|\mb{y}_{n/N}\|\leq \sqrt{\frac{p}{N}\frac{n}{N}}
\]
and using the lower bound above,
\begin{eqnarray*}
 \|\mb{y}_{p/N}\|\|\mb{y}_{n/N}\|\geq \left(\sqrt{\frac{p}{N}}-\frac{2}{\pi^2 L_n\sqrt{p/N}}\right)\left(\sqrt{\frac{n}{N}}-\frac{2}{\pi^2 L_n\sqrt{n/N}}\right).
\end{eqnarray*}
Expanding the product yields
\[
\|\mb{y}_{p/N}\|\|\mb{y}_{n/N}\|\geq\sqrt{\frac{p}{N}\frac{n}{N}}-\frac{\sqrt{n/N}}{L_n\sqrt{p/N}}.
\]
Finally, since
\[
L_n p/N\geq L_n p/(2n)\geq \log^{1-\alpha}(n)/2\rightarrow \infty\quad\text{as }n\rightarrow\infty,
\]
the error term is $o(1)$, and hence,
\[
\|\mb{y}_{p/N}\|\|\mb{y}_{n/N}\|=\sqrt{\frac{p}{N}\frac{n}{N}}\left(1-o(1)\right).
\]

Using this in \eqref{explicit} gives
\[
\frac{\|\Pi_{p/N}^{[2L_n+1]}\operatorname{diag}\{\mb{u}^{(0)}\}\Pi_{n/N}^{[2L_n+1]}\|}{\sqrt{p/N}}=\left(1-\frac{\tau}{15}\right)\sqrt{\frac{n}{N}}+o(1).
\]
If $p/n\leq \tau$, then $\sqrt{n/N}=\sqrt{1-p/N}\geq\sqrt{1-\tau}$. 

Since Lemma \ref{lem: grid} assumes that $\tau\in(0,3/4)$, we have $\sqrt{1-\tau}>1-2\tau/3$. Therefore,
\[
\frac{\|\Pi_{p/N}^{[2L_n+1]}\operatorname{diag}\{\mb{u}^{(0)}\}\Pi_{n/N}^{[2L_n+1]}\|}{\sqrt{p/N}}>\left(1-\frac{\tau}{15}\right)\left(1-\frac{2\tau}{3}\right)+o(1)>1-\frac{3\tau}{4}
\]
for all sufficiently large $n$. 

Finally, since $K_{1,n/p}\leq \sup_{\Lambda\in\mathcal{F}_1}\left(\int\Lambda(x)\mathrm{d}x\right)^{1/2}=1$ (by \cite{boas45}), we obtain \eqref{first grid interval}, completing the proof of Lemma~\ref{lem: grid}.

\subsection{Proof of Lemma \ref{lem: gumbel}}
First, we recall some facts about sub-gamma random variables (see e.g.~chapter~{2.4} of \cite{boucheron2013concentration}). A  centered random variable $X$ belongs to the sub-gamma
 family
${SG}(v,u)$ for $v, u > 0$ if
\begin{equation*}
    \log \mathbb{E} e^{tX} \leq \frac{t^2 v}{2(1 - tu)},
    \qquad
    \forall t: \; 0<t < \frac{1}{u} \, .
  \end{equation*}
   If $X \in SG(v,u)$ then $X \in SG(v',u')$ for each $v' \geq v$,$u' \geq u$, and for arbitrary $C \in \mathbb{R}$, we have 
 \begin{equation}
 \label{SG multiplication}
 CX\in SG\left(C^{2}v,|C| u\right) .
\end{equation}
If  $X\in SG\left( v_{X},u_{X} \right) $ and $Y\in SG\left( v_{Y},u_{Y}\right) 
$ are independent, then 
\begin{equation}
\label{SG independent}
X+Y\in SG\left( v_{X}+v_{Y},\max \left\{
u_{X},u_{Y}\right\} \right).
\end{equation}

  If $X\sim \chi^2(1)$, that is $X$ is distributed according to $\chi^2$ with one degrees of freedom, then $X-1\in SG(2,2)$. If $X\sim\operatorname{Exp}(1)$, then $X-1\in SG(1,1)\subset SG(2,2)$.  Thus, for a linear combination of independent $\operatorname{Exp}(1)$ and $\chi^2(1)$ random variables, we have
\begin{equation}\label{lin comb of exponentials}
\sum_{i=0}^k \omega_i(X_i-1)\in SG(2\|\omega\|^2_2,2\|\omega\|_\infty),
\end{equation}
where $\omega=(\omega_0,\dots,\omega_k)^\top\in\mathbb{R}^{k+1}$.
  Finally, recall that if $X\in SG(v,u)$, then for every $t \geq 0$,
\begin{equation}
    \label{exponential ineq}
    \mathbb{P}(X>\sqrt{2vt}+ut)\leq e^{-t}.
\end{equation}

With these preliminaries in place, we write \eqref{stationary moving average} as
\[
z_j:=\frac{n}{p}\,\overline{\mb{v}_{j}}^\top\mb{D}\overline{\mb{D}} \mb{v}_{j}
   = \sum_{i=0}^{n/2} \omega_i^{(j)} |d_i|^2,
\]
with
\[
\omega_i^{(j)}=\begin{cases}
    \frac{1}{c}\left(|(\mb{v}_j)_i|^2+|(\mb{v}_j)_{n-i}|^2\right)&\text{for }0<i<n/2\\
    \frac{1}{c}|(\mb{v}_j)_i|^2&\text{for }i\in\{0,n/2\}
\end{cases}.
\]
Note that
\[
\|\omega^{(j)}\|_2^2\leq \frac{4}{c^2}\|\mb{v}_j\|_4^4\leq\frac{4}{c^2},\qquad \|\omega^{(j)}\|_\infty\leq \frac{2}{c}\|\mb{v}_j\|_\infty^2\leq 2, \qquad \|\omega^{(j)}\|_1\leq \frac{2}{c}\|\mb{v}_j\|_2^2=\frac{2}{c}.
\]
This, together with \eqref{lin comb of exponentials}, implies that
\[
z_j-\|\omega^{(j)}\|_1\in SG\left(8/c^2,4\right).
\]
Therefore, from \eqref{exponential ineq}, for every $t \geq 0$,
\begin{equation}
    \label{exponential ineq specific}
    \Pr\left(z_j>2/c+4\sqrt{t}/c+4t\right)\leq e^{-t}.
\end{equation}

The latter inequality implies that it is sufficient to prove Lemma \ref{lem: gumbel} with $B_{p,n}$ replaced by
\[
\tilde{B}_{p,n}:=p\max_{ j\in J_n}z_j ,\qquad J_n:=\{j:\lceil\log n\rceil\leq j\leq n/2-\lceil\log n\rceil\}.
\]
Indeed, let $E_{1n}$ be the event that $\frac{1}{p}\tilde{B}_{p,n}>\sqrt{\log n}$ and $E_{2n}$ the event that
\[
\max_{j\in J_n^c} z_j\leq \sqrt{\log n},\qquad J_n^c:=\{j:0\leq j\leq n/2\}\setminus J_n.
\]
If Lemma \ref{lem: gumbel} holds for $\tilde{B}_{p,n}$, $\mathbb{P}\{E_{1n}\}\rightarrow 1$ as $n\rightarrow \infty$. Moreover, setting $t = \sqrt{\log n}/8$ in \eqref{exponential ineq specific}, we obtain for all sufficiently large $n$,
\[
\mathbb{P}\{E_{2n}^c\}\leq 2(\log n+1)e^{-\sqrt{\log n}/8},
\]
and hence $\mathbb{P}\{E_{2n}\}\rightarrow 1$ as $n\rightarrow \infty$. Finally, since $E_{1n} \cap E_{2n}$ implies $B_{p,n} = \tilde{B}_{p,n}$, it remains to establish Lemma \ref{lem: gumbel} with $\tilde{B}_{p,n}$ in place of $B_{p,n}$.

Let $\{\xi_i\}_{i\in\mathbb{Z}}$ be a sequence of i.i.d. $\operatorname{Exp}(1)$ random variables with the identification $\xi_i = |d_i|^2$ for $1 \le i \le n/2-1$. We will also assume that the sequence $\{\xi_i\}_{i\in\mathbb{Z}}$ is independent from $d_0$ and $d_{n/2}$. Further, let $\hat{\omega}^{(j)}$ with $j\in J_n$ be vectors in $\ell^2(\mathbb{R})$ with the $i$-th coordinate defined by
\begin{equation}\label{omega hat ij}
\hat{\omega}^{(j)}_i:=\frac{|\langle e_j,\Pi_c e_i\rangle|^2}{c^2}=\begin{cases}
    1&\text{if }i=j\\
    \left(\frac{\sin [ c\pi(i-j)]}{c\pi(i-j)}\right)^2&\text{if }i\neq j
\end{cases},
\end{equation}
and let
\begin{equation}\label{definition of zeta}
\zeta_j:=\sum_{i\in\mathbb{Z}}\hat{\omega}^{(j)}_i\xi_i.
\end{equation}
It then suffices to establish Lemma \ref{lem: gumbel} with $\tilde{B}_{p,n}$ replaced by
\[
\hat{B}_{p,n} := p \max_{j\in J_n} \zeta_j.
\]

Indeed, denote $z_j-\zeta_j$ as $Z_j$. Since $\mathbb{E}z_j=1/c$ and $\mathbb{E}\zeta_j=1/c$, we have $\mathbb{E}Z_j=0$.
Further, by definition,
\[
Z_j=\omega_0^{(j)}|d_0|^2+\omega_{n/2}^{(j)}|d_{n/2}|^2+\sum_{i\in\mathbb{Z}}\check{\omega}^{(j)}_i\xi_i,
\]
where 
\[
\check{\omega}^{(j)}_i:=\begin{cases}
    \omega^{(j)}_i-\hat{\omega}^{(j)}_i&\text{for } i=1,\dots,n/2-1\\
    \hat{\omega}^{(j)}_i&\text{for }i\leq 0
\end{cases}.
\]

Since $Z_j$ is a linear combination of independent sub-gamma random variables, it is itself sub-gamma. The proof the following lemma is elementary, and is  deferred to the SM.
\begin{lemma}\label{lem: Xj}In the above notations,
\[
Z_j\in SG(v_j,u_j),
\]
where, for any $j\in J_n$,
\[
v_j\leq \frac{1}{c^4\log^3 n}\quad\text{and}\quad u_j\leq  \frac{1}{c^2\log^2 n}.
\]    
\end{lemma}

By \eqref{SG multiplication}, $-Z_j\in SG(v_j,u_j)$ too, and hence, for every $t>0$, we have
\[
\mathbb{P}(|Z_j|>\sqrt{2v_jt}+u_jt)\leq 2e^{-t}.
\]
Setting $t=\log^{3/2}n$ and noting that, for all $j\in J_n$ and all sufficiently large $n$,
\[
\sqrt{2v_jt}+u_jt\leq \frac{\sqrt{2}}{c^2\log^{3/4} n}+\frac{1}{c^2\log^{1/2}n}\leq \frac{2}{c^2\log^{1/2}n},
\]
we obtain
\[
\mathbb{P}\left\{\max_{j\in J_n}|Z_j|>\frac{2}{c^2\sqrt{\log n}}\right\}\leq 2ne^{-\log^{3/2}n}\rightarrow 0\qquad\text{as }n\rightarrow\infty.
\]
Since 
\[
\frac{|\tilde{B}_{p,n}-\hat{B}_{p,n}|}{p}\leq\max_{j\in J_n}|Z_j|,
\]
we conclude that it is sufficient to establish Lemma \ref{lem: gumbel} with $\hat{B}_{p,n}$ in place of $B_{p,n}$. In other words, it suffices to establish that
\begin{equation}\label{raw result for stationary sequences}
\max_{j\in J_n}\zeta_j-\log\frac{n}{2}\overset{d}{\longrightarrow}\operatorname{Gumbel}(\theta_c,1).
\end{equation}

Let $m:=|J_n|=n/2-2\lceil\log n\rceil+1$. Since $|\log(n/2)-\log m|\rightarrow 0$ as $n\rightarrow \infty$, replacing $\log (n/2)$ by $\log m$ does not change the limit. Changing the index $j\mapsto j-\lceil\log n\rceil +1$, we obtain
\[
\max_{j\in J_n}\zeta_j\mapsto\max_{j=1,\dots,m}\zeta_j,
\]
where the re-indexed $\zeta_j$ satisfies (cf. \eqref{definition of zeta}) 
\begin{equation}\label{definition of reindexed zeta}
\zeta_j=\sum_{i\in\mathbb{Z}}\hat{\omega}^{(j+\lceil\log n\rceil-1)}_i\xi_i=\sum_{i\in\mathbb{Z}}\hat{\omega}^{(j)}_{i-\lceil\log n\rceil+1}\xi_i=\sum_{i\in\mathbb{Z}}\hat{\omega}^{(j)}_i\xi_i',
\end{equation}
where $\xi'_i:=\xi_{i+\lceil\log n\rceil -1}$.
Recalling that $\operatorname{Gumbel}(\theta_c,1)$ is is simply the standard Gumbel law $\operatorname{Gumbel}(0,1)$ shifted by $\theta_c$,
we can restate \eqref{raw result for stationary sequences} in the equivalent form
\begin{equation}\label{result for stationary sequences}
\max_{j=1,\dots,m}\zeta_j- x_m\overset{d}\rightarrow\operatorname{Gumbel}(0,1),
\end{equation}
where
\[
x_m:=x+\log m+\theta_c.
\]

Observe that $\{\zeta_j\}$ is a strictly stationary sequence. By Theorem 3.5.2 in \cite{leadbetter83}, the convergence in \eqref{result for stationary sequences} follows from the corresponding convergence for an i.i.d. sequence with the same marginal distribution, provided that two additional conditions hold. 

\paragraph{(i) Weak dependence (eq.~(3.2.2) in \cite{leadbetter83}):} Denote $\mathbb{P}\{\zeta_{i_1}\leq x_m,\dots,\zeta_{i_k}\leq x_m\}$ as $F_{i_1,\dots,i_k}(x_m)$. Then, for any $x\in\mathbb{R}$, there exists a sequence $l_m=o(m)$ such that, for any integers
\[
1\leq i_1<\cdots<i_r<j_1<\cdots<j_{r'}\leq m
\]
with spacing $j_1-i_r\geq l_m$, we have
\[
|F_{i_1,\dots,i_r,j_1,\dots,j_{r'}}(x_m)-F_{i_1,\dots,i_r}(x_m)F_{j_1,\dots,j_{r'}}(x_m)|\rightarrow 0\quad\text{as }m\rightarrow\infty.
\]

\paragraph{(ii) No clustering of extremes (eq.~(3.4.3) in \cite{leadbetter83}):}
\begin{equation}\label{no clustering}
\limsup_{m\rightarrow \infty}\quad m\sum_{j=2}^{\lfloor m/k\rfloor} \mathbb{P}\{\zeta_1>x_m,\zeta_j>x_m\}\rightarrow 0\quad\text{as }k\rightarrow\infty.
\end{equation}

Before verifying conditions (i) and (ii), let us first establish the analogue of \eqref{result for stationary sequences} for an i.i.d. sequence $\{\hat{\zeta}_j\}_{j=1}^m$ with the same common distribution as $\zeta_j$. In particular, we show that  
\begin{equation}\label{convergence for iid}
\max_{1\le j\le m}\hat{\zeta}_j - x_m \;\overset{d}{\longrightarrow}\; \operatorname{Gumbel}(0,1).
\end{equation}
By Theorem 1.5.1 in \cite{leadbetter83}, it suffices to show that
\begin{equation}\label{condition marginal}
m(1-F_\zeta(x_m))\rightarrow e^{-x}\quad\text{as }m\rightarrow\infty,
\end{equation}
where $F_\zeta(x)$ is the cumulative distribution function of $\zeta_1$.

Observe that a standard $\operatorname{Exp}(1)$ random variable $\xi_i'$, when multiplied by $2$, has a $\chi^2$ distribution with $2$ degrees of freedom. Using this in \eqref{definition of reindexed zeta}, we obtain a representation
\begin{equation}\label{zeta hat}
\zeta_1=\frac{1}{2}\chi^2_0+\frac{1}{2}\sum_{j=1}^\infty\hat{\omega}_j\chi^2_j,
\end{equation}
where $\hat{\omega}_j=\hat{\omega}_j^{(0)}$, the variables $\chi^2_j$, $j=0,1,2,...$ are independent, with $\chi^2_0\sim\chi^2(2)$ and $\chi^2_j\sim\chi^2(4)$ for $j\ge 1$.

By the asymptotic formula (6) in \cite{zolotarev61}, the tail of $F_\zeta$ satisfies
\begin{equation}\label{zolotarev cdf}
1-F_{\zeta}(x_m)=e^{\theta_c-x_m}(1+\epsilon(x_m)),
\end{equation}
where  $\epsilon(x_m)\rightarrow 0$ when $x_m\rightarrow\infty$. Since $\theta_c-x_m=-x-\log m$, this establishes \eqref{condition marginal}. Consequently, \eqref{convergence for iid} follows.

We now verify conditions (i) and (ii), beginning with (i). Introduce the truncated version of $\zeta_j$ (cf.~\eqref{definition of reindexed zeta}):
\[
\bar{\zeta}_j:=\sum_{i\in\mathbb{Z}:|i-j|\leq\log m}\hat{\omega}_{i}^{(j)}\xi'_{i}.
\]
For any integers $1\leq s_1<\cdots<s_R\leq m$, let \[
\bar{F}_{s_1,\dots,s_R}(u_m):=\mathbb{P}\{\bar{\zeta}_{s_1}\leq u_m,\dots,\bar{\zeta}_{s_R}\leq u_m\}.
\]
We will need the following lemma.
\begin{lemma}\label{lem: validity of D}
    For any $x\in\mathbb{R}$, any integers $1\leq s_1<\cdots<s_R\leq m$, and all sufficiently large $m$, we have
    \begin{equation}\label{F and Fbar}
    |F_{s_1,\dots,s_R}(x_m)-\bar{F}_{s_1,\dots,s_R}(x_m)|\leq \frac{R}{m\log^{1/5} m}.
    \end{equation}
\end{lemma}

The lemma implies condition (i). Indeed, set $l_m=3\log m$. 
By construction, the variables $\bar{\zeta}_{i_1},\dots,\bar{\zeta}_{i_r}$ are independent from $\bar{\zeta}_{j_1},\dots,\bar{\zeta}_{j_{r'}}$ (since the truncation windows do not overlap). Hence,
\[
|\bar{F}_{i_1,\dots,i_r,j_1,\dots,j_{r'}}(u_m)-\bar{F}_{i_1,\dots,i_r}(u_m)\bar{F}_{j_1,\dots,j_{r'}}(u_m)|=0.
\]
On the other hand, using Lemma \ref{lem: validity of D}, replacing $\bar{F}$ with $F$ changes the expression by at most 
\[
\frac{r+r'}{m\log^{1/5}m}+\frac{r}{m\log^{1/5}m}+\frac{r'}{m\log^{1/5}m}\leq \frac{2}{\log^{1/5}m}\rightarrow 0,\quad\text{as }m\rightarrow\infty.
\] 
Hence, condition (i) indeed holds. 

Let us establish Lemma \ref{lem: validity of D}. Observe that for any $\delta_m>0$ we have
\begin{eqnarray}\label{truncation estimate}
F_{s_1,\dots,s_R}(x_m)\leq \bar{F}_{s_1,\dots,s_R}(x_m)&\leq& F_{s_1,\dots,s_R}(x_m+\delta_m)+\sum_{i=1}^R\mathbb{P}\{\zeta_{s_i}-\bar{\zeta}_{s_i}>\delta_m\}\\\notag
&=&F_{s_1,\dots,s_R}(x_m+\delta_m)+R\mathbb{P}\{\zeta_1-\bar{\zeta}_1>\delta_m\}.
\end{eqnarray}
The representation \eqref{zeta hat} for $\zeta_1$ yields a similar representation for $\bar{\zeta}_1$:
\[
\bar{\zeta}_1=\frac{1}{2}\chi^2_0+\frac{1}{2}\sum_{j=1}^{\lfloor\log m\rfloor}\hat{\omega}_j\chi^2_j.
\]
Since, for $j\geq 1$, $\frac{1}{2}\chi_j^2-2\in SG(2,1)$ (because  $\chi_j^2\sim\chi^2(4)$), we have $\zeta_1-\bar{\zeta}_1-\mathbb{E}(\zeta_1-\bar{\zeta}_1)\in SG (u,v)$ with
\begin{eqnarray*}
\mathbb{E}(\zeta_1-\bar{\zeta}_1)&=&2\sum_{j\geq\lfloor\log m\rfloor+1} \hat{\omega}_j\leq \frac{1}{c^2\pi^2\lfloor \log m\rfloor},\\
v&=&2\sum_{j\geq \lfloor\log m\rfloor+1}\hat{\omega}_j^2\leq \frac{2}{3c^4\pi^4\lfloor\log m\rfloor^3},\\
u&=&\max_{j\geq \lfloor\log m\rfloor+1}\hat{\omega}_j\leq \frac{1}{c^2\pi^2\lfloor \log m\rfloor^2}.
\end{eqnarray*}
Inequality \eqref{exponential ineq} then implies that, for any $t>0$,
\[
\mathbb{P}\left\{\zeta_1-\bar{\zeta}_1>\frac{1}{c^2\pi^2\lfloor \log m\rfloor}+\sqrt{\frac{4t}{3c^4\pi^4\lfloor\log m\rfloor^3}}+\frac{t}{c^2\pi^2\lfloor \log m\rfloor^2}\right\}\leq e^{-t}.
\]
Choosing $t=\log^{3/2}m$ and $\delta_m=\log^{-1/4} m$, we obtain for all sufficiently large $m$,
\begin{equation}\label{residual 1}
\mathbb{P}\left\{\zeta_1-\bar{\zeta}_1>\delta_m\right\}\leq e^{-\log^{3/2}m}.
\end{equation}

Further,
\begin{equation}\label{small ball}
F_{s_1,\dots,s_R}(x_m+\delta_m)-F_{s_1,\dots,s_R}(x_m)\leq \sum_{i=1}^R\mathbb{P}\{\zeta_{s_{i}}\in(x_m,x_m+\delta_m]\}=R\mathbb{P}\{\zeta_1\in(x_m,x_m+\delta_m]\}.
\end{equation}
Equation (5) in \cite{zolotarev61} shows that the density $f_\zeta(x)$ of $\zeta_1$  satisfies 
\[
f_{\zeta}(z)=e^{\theta_c-z}(1+\epsilon(z)),
\]
where $\epsilon(z)\rightarrow 0$ when $z\rightarrow \infty$. Therefore, for all sufficiently large $m$,
\[
\mathbb{P}\{\zeta_1\in(x_m,x_m+\delta_m]\}\leq 2\delta_m e^{\theta_c-x_m}=\frac{2e^{-x}}{m\log^{1/4} m}.
\]
Using this and \eqref{residual 1}-\eqref{small ball} in \eqref{truncation estimate}, we obtain \eqref{F and Fbar}. This completes the proof of Lemma \ref{lem: validity of D} and thus verifies condition (i).

It remains to establish condition (ii). Recall the representation \eqref{definition of reindexed zeta} of $\zeta_j$ and let
\[
r_{1j}:=\sum_{i\in\mathbb{Z}}\min\left\{\hat{\omega}_i^{(1)},\hat{\omega}_i^{(j)}\right\}\xi_i',\qquad r_{2j}:=\zeta_1-r_{1j},\qquad r_{3j}:=\zeta_j-r_{1j}.
\]
Note that $r_{2j}$ and $r_{3j}$ are independent random variables. For any $y_m>0$, we have
\begin{eqnarray*}
\mathbb{P}\{\zeta_1>x_m,\zeta_j>x_m\}&=&\mathbb{P}\{r_{1j}+r_{2j}>x_m,r_{1j}+r_{3j}>x_m\}\\
&\leq&\mathbb{P}\{r_{2j}>x_m-y_m,r_{3j}>x_m-y_m\}+\mathbb{P}\{r_{1j}>y_m\}\\
&=&\mathbb{P}\{r_{2j}>x_m-y_m\}\mathbb{P}\{r_{3j}>x_m-y_m\}+\mathbb{P}\{r_{1j}>y_m\}\\
&\leq&\mathbb{P}\{\zeta_1>x_m-y_m\}\mathbb{P}\{\zeta_j>x_m-y_m\}+\mathbb{P}\{r_{1j}>y_m\}\\
&=&(1-F_\zeta(x_m-y_m))^2+\mathbb{P}\{r_{1j}>y_m\}.
\end{eqnarray*}
Using the definition \eqref{omega hat ij} of $\hat{\omega}_{i}^{(j)}$, we obtain
\[
\min\left\{\hat{\omega}_i^{(1)},\hat{\omega}_i^{(j)}\right\}\leq \min\left\{\frac{1}{c^2\pi^2(i-1)^2},\frac{1}{c^2\pi^2(j-i)^2}\right\},
\]
which implies that
\[
r_{1j}\leq \bar{r}_{1j}:=\sum_{i:i<(j+1)/2}\frac{\xi'_i}{c^2\pi^2(j-i)^2}+\sum_{i:i\geq(j+1)/2}\frac{\xi'_i}{c^2\pi^2(i-1)^2}.
\]
Combining these results, we obtain
\begin{equation}\label{to use for S2 and S3}
\mathbb{P}\{\zeta_1>x_m,\zeta_j>x_m\}\leq (1-F_\zeta(x_m-y_m))^2+\mathbb{P}\{\bar{r}_{1j}>y_m\}.
\end{equation}

Now split the sum in \eqref{no clustering} into three parts:
\[
S_1+S_2+S_3:=\left(m\sum_{j=2}^{2\lceil 1/c\rceil}+m\sum_{j=2\lceil 1/c\rceil+1}^{2\lceil \log m\rceil}+ m\sum_{j=2\lceil \log m\rceil+1}^{\lfloor m/k\rfloor}\right)\mathbb{P}\{\zeta_1>x_m,\zeta_j>x_m\}.
\]
First consider the part $S_3$, where we have $j\geq 2\lceil\log m\rceil +1$. Since $\xi'_i-1\in SG(1,1)$, properties \eqref{SG multiplication} and \eqref{SG independent} of sub-gamma random variables imply that $\bar{r}_{1j}-\mathbb{E}\bar{r}_{1j}\in SG(v_r,u_r)$ with
\begin{equation}\label{r1bar SG}
     \mathbb{E}\bar{r}_{1j}\leq \frac{2}{c^2\pi^2(\lceil\log m\rceil-1)}, \quad v_r\leq \frac{2}{3c^4\pi^4(\lceil\log m\rceil-1)^3},\quad u_r\leq \frac{1}{c^2\pi^2\lceil\log m\rceil^2}.
\end{equation}
Therefore, \eqref{exponential ineq} yields
\[
\mathbb{P}\{\bar{r}_{1j}>1\}\leq \exp\{-\log^{3/2}m\},
\]
for all sufficiently large $m$. Using this in \eqref{to use for S2 and S3} with $y_m=1$, we obtain
\[
\mathbb{P}\{\zeta_1>x_m,\zeta_j>x_m\}\leq (1-F_\zeta(x_m-1))^2+\exp\{-\log^{3/2}m\}.
\]
Using the asymptotic formula \eqref{zolotarev cdf} (with $x_m-1$ instead of $x_m$), we conclude that for all sufficiently large $m$,
\[
\mathbb{P}\{\zeta_1>x_m,\zeta_j>x_m\}\leq \frac{2e^{2(1-x)}}{m^2}+\exp\{-\log^{3/2}m\}\leq \frac{3e^{2(1-x)}}{m^2}.
\]
Therefore,
\[
S_3\leq \frac{m^2}{k}\frac{3e^{2(1-x)}}{m^2}=\frac{3e^{2(1-x)}}{k}\rightarrow 0,\quad\text{as }k\rightarrow\infty.
\]

Consider now $S_2$, where we have $j\geq g:=2\lceil 1/c\rceil+1$. Note that, for such $j$, 
\begin{equation}\label{just rbar}
\mathbb{P}\{\bar{r}_{1j}>y_m\}\leq \mathbb{P}\{\bar{r}>y_m\},
\end{equation}
where
\[
\bar{r}:=\sum_{i:i<\lceil 1/c\rceil+1}\frac{\xi'_i}{c^2\pi^2(1/c+\lceil 1/c\rceil+1-i)^2}+\sum_{i:i\geq \lceil 1/c\rceil+1}\frac{\xi'_i}{c^2\pi^2(1/c+i-\lceil 1/c\rceil-1)^2}.
\]
For $\bar{r}$, we have the following representation, similar to \eqref{zeta hat}:
\[
\bar{r}=\frac{1}{2\pi^2}\chi^2_0+\sum_{j=1}^\infty\frac{1}{2\pi^2(1+cj)^2}\chi^2_j.
\]
Therefore, the asymptotic formula (6) in \cite{zolotarev61} yields
\begin{equation*}
\mathbb{P}\{\bar{r}>y_m\}=O(e^{-\pi^2y_m})\quad\text{as }y_m\rightarrow\infty
\end{equation*}
The same asymptotic formula yields
\begin{equation*}
(1-F_\zeta(x_m-y_m))^2=O(e^{-2x_m+2y_m})\quad\text{as }x_m-y_m\rightarrow\infty.
\end{equation*}
Setting $y_m=\frac{1}{3}\log m$ and recalling that $x_m=x+\log m+\theta_c$, we obtain
\[
(1-F_\zeta(x_m-y_m))^2+\mathbb{P}\{\bar{r}>y_m\}=O(e^{-\frac{4}{3}\log m})+O(e^{-\frac{\pi^2}{3}\log m})=O(m^{-4/3}).
\]
Using this and \eqref{just rbar} in \eqref{to use for S2 and S3}, we obtain the uniform bound:
\[
\mathbb{P}\{\zeta_1>x_m,\zeta_j>x_m\}\leq O(m^{-4/3}).
\]
Therefore,
\[
S_2\leq 2m\lceil\log m\rceil O(m^{-4/3})\rightarrow 0\quad\text{as }m\rightarrow\infty.
\]

Finally, consider $S_1$, where we have $j\leq 2\lceil 1/c\rceil$. Note that
\[
\mathbb{P}\{\zeta_1>x_m,\zeta_j>x_m\}\leq \mathbb{P}\left\{\frac{\zeta_1+\zeta_j}{2}>x_m\right\}.
\]
Equation \eqref{definition of reindexed zeta} implies that, for any $j\geq 2$, there exists $\delta_j>0$ such that the random variable $\frac{\zeta_1+\zeta_j}{2}$ admits an upper bound
\[
\frac{\zeta_1+\zeta_j}{2}\leq \frac{1}{2}\frac{1}{1+\delta_j}\xi_0^2+\frac{1}{2}\sum_{i=1}^\infty\sigma_i\xi_i^2
\]
with $0\leq \sigma_i<\frac{1}{1+\delta_j}$. Therefore, 
by (6) in \cite{zolotarev61}, 
\[
\mathbb{P}\left\{\frac{\zeta_1+\zeta_j}{2}>x_m\right\}=O(m^{-1-\delta_j}),
\]
and hence,
\[
S_1\leq m\sum_{j=2}^{2\lceil1/c\rceil}O(m^{-1-\delta_j})\rightarrow 0\qquad\text{as }m\rightarrow\infty.
\]
This completes the proof of condition (ii), yielding Lemma \ref{lem: gumbel}.

\section{Supplementary Material}
\subsection{Proof of Lemma \ref{lem:reductions} (Reductions and standardizations)}

\subsubsection*{Part (i) (truncation and standardization)}

Similarly to SV13, we introduce truncated random variables
\[
\hat{a}_j := a_j \mathbf{1}\left\{ |a_j| \leq \tfrac{1}{2} n^{1/\gamma} \right\}.
\]
The mean and variance of $\hat{a}_j$ are not exactly zero and one, though they approach these values as $n$ becomes large. We therefore standardize:
\[
\bar{a}_j := \frac{\hat{a}_j - \mathbb{E} \hat{a}_j}{\sqrt{\operatorname{Var} \hat{a}_j}}.
\]
By construction, $|\bar{a}_j| < n^{1/\gamma}$ for all sufficiently large $n$.

Let $\bar{\mb{a}} := (\bar{a}_0, \dots, \bar{a}_{N-1})^\top$. Then, to establish part (i) of Lemma \ref{lem:reductions}, it is sufficient to show that 
\begin{equation}\label{part (i) of reduction}
    \frac{\sqrt{N}\|\mb{P}_{p,N}\mb{D}(\mb{a})\mb{P}_{n,N}\|-\sqrt{N}\|\mb{P}_{p,N}\mb{D}(\bar{\mb{a}})\mb{P}_{n,N}\|}{\sqrt{p\log n}} \overset{L^\gamma}{\rightarrow} 0,
\end{equation}
where $\overset{L^\gamma}{\rightarrow}$ denotes convergence in $L^\gamma$.

Since the spectral norm of the projection matrices $\mb{P}_{r,N}$ is $1$, we have
\begin{align*}
\|\mb{P}_{p,N}\mb{D}(\mb{a})\mb{P}_{n,N}\| - \|\mb{P}_{p,N}\mb{D}(\bar{\mb{a}})\mb{P}_{n,N}\|
&\leq \|\mb{P}_{p,N}(\mb{D}(\mb{a}) - \mb{D}(\bar{\mb{a}}))\mb{P}_{n,N}\| \\
&\leq \|\mb{D}(\mb{a}) - \mb{D}(\bar{\mb{a}})\|.
\end{align*}
Hence, it suffices to prove that
\[
\frac{\|\sqrt{N}(\mb{D}(\mb{a}) - \mb{D}(\bar{\mb{a}}))\|}{\sqrt{p\log n}} \overset{L^\gamma}{\rightarrow} 0.
\]

Let $\hat{\mb{a}} := (\hat{a}_0, \dots, \hat{a}_{N-1})^\top$ and $\tilde{\mb{a}} := (\tilde{a}_0, \dots, \tilde{a}_{N-1})^\top$, where $\tilde{a}_j := \hat{a}_j - \mathbb{E} \hat{a}_j$. To establish this convergence, we will show that
\begin{align}
    \label{part 1 lemma 2}
    \frac{\|\sqrt{N}(\mb{D}(\mb{a}) - \mb{D}(\hat{\mb{a}}))\|}{\sqrt{p\log n}} &\overset{L^\gamma}{\rightarrow} 0, \\
    \label{part 2 lemma 2}
    \frac{\|\sqrt{N}(\mb{D}(\hat{\mb{a}}) - \mb{D}(\tilde{\mb{a}}))\|}{\sqrt{p\log n}} &\overset{L^\gamma}{\rightarrow} 0, \\
    \label{part 3 lemma 2}
    \frac{\|\sqrt{N}(\mb{D}(\tilde{\mb{a}}) - \mb{D}(\bar{\mb{a}}))\|}{\sqrt{p\log n}} &\overset{L^\gamma}{\rightarrow} 0.
\end{align}

\paragraph{Step 1. Proof of \eqref{part 1 lemma 2}:} By definition of $\mb{D}(\mb{a})$,
\[
\|\sqrt{N}(\mb{D}(\mb{a}) - \mb{D}(\hat{\mb{a}}))\| = \max_{j=0,\dots,N-1} \left| \sum_{k=0}^{N-1} e^{2\pi\rmi jk/N} a_k \mathbf{1}\left\{ |a_k| > \tfrac{1}{2} n^{1/\gamma} \right\} \right| \leq \sum_{k=0}^{N-1} |a_k| \mathbf{1}\left\{ |a_k| > \tfrac{1}{2} n^{1/\gamma} \right\}.
\]
Using Rosenthal’s inequality (as in SV13), the $\gamma$'s moment of the right hand side of the latter inequality is bounded by a constant times $N\mathbb{E}|a_0|^\gamma$ , which is $O(N)$ under \eqref{Lyapunov}: 
\[
\mathbb{E} \left( \sum_{k=0}^{N-1} |a_k| \mathbf{1}\left\{ |a_k| > \tfrac{1}{2} n^{1/\gamma} \right\} \right)^\gamma = O(N) = O(n).
\]
Therefore,
\[
\mathbb{E} \left| \frac{\|\sqrt{N}(\mb{D}(\mb{a}) - \mb{D}(\hat{\mb{a}}))\|}{\sqrt{p\log n}} \right|^\gamma = \frac{O(n)}{(p \log n)^{\gamma/2}} \rightarrow 0,
\]
since $\gamma > 2$ and $(p \log n)/n \gg 1$ by assumption. This completes the proof of \eqref{part 1 lemma 2}.

\paragraph{Step 2. Proof of \eqref{part 2 lemma 2}:}
To prove \eqref{part 2 lemma 2}, observe that
\begin{eqnarray*}
         \mathbb{E}\left|\frac{\|\sqrt{N}\mb{D}(\hat{\mb{a}})-\sqrt{N}\mb{D}(\tilde{\mb{a}})\|}{\sqrt{p\log n}}\right|^\gamma&\leq&\left|\frac{\sum_{k=0}^{N-1}|\mathbb{E}a_k\mathbf{1}\{|a_k|\leq \tfrac{1}{2}n^{1/\gamma}\}|}{\sqrt{p\log n}}\right|^\gamma\\
     &=&\left|\frac{\sum_{k=0}^{N-1}|\mathbb{E}a_k\mathbf{1}\{|a_k|\leq \tfrac{1}{2}n^{1/\gamma}\}-\mathbb{E}a_k|}{\sqrt{p\log n}}\right|^\gamma\\
     &\leq &\left|\frac{\sum_{k=0}^{N-1}\mathbb{E}|a_k|\mathbf{1}\{|a_k|> \tfrac{1}{2}n^{1/\gamma}\}}{\sqrt{p\log n}}\right |^\gamma=\frac{O(n)}{(p\log n)^{\gamma/2}}\rightarrow 0.
\end{eqnarray*}

\paragraph{Step 3. Proof of \eqref{part 3 lemma 2}:}
We write
\begin{equation}\label{preparation for part 3}
 \|\mb{D}(\tilde{\mb{a}})-\mb{D}(\bar{\mb{a}})\|^\gamma = \max_{0\leq j \leq \lfloor N/2 \rfloor} \left| d_j(\tilde{\mb{a}} - \bar{\mb{a}}) \right|^\gamma,
\end{equation}
where
\begin{equation}\label{d definition}
 d_j(\mb{x}) := \frac{1}{\sqrt{N}} \sum_{k=0}^{N-1} e^{2\pi\rmi kj/N} x_k
\end{equation}
for any $N$-vector $\mb{x}$ with independent real-valued entries.

We use the following tail bound:
\begin{lemma}\label{lem: bernstein for d}
Let $\mb{x} = (x_0, \dots, x_{N-1})^\top$ be a vector of independent real random variables such that $\mathbb{E}x_k = 0$, $\operatorname{Var}(x_k) \leq \sigma^2$, and $|x_k| \leq M$ for all $k$. Then for all $j$ and $t > 0$,
\[
\mathbb{P}(|d_j(\mb{x})| > t) \leq 4 \exp\left(-\frac{t^2}{4\sigma^2 + Mt/\sqrt{N}}\right).
\]
\end{lemma}
\begin{proof}
Split $d_j$ into real and imaginary parts:
\[
\mathbb{P}(|d_j| > t) \leq \mathbb{P}(|\Re d_j| > t/\sqrt{2}) + \mathbb{P}(|\Im d_j| > t/\sqrt{2}).
\]
Each summand is a sum of independent, bounded, mean-zero variables with variance at most $\sigma^2$, so we apply Bernstein’s inequality (e.g., Theorem 2.8.4 in \cite{vershynin2018}). Noting that $2\sqrt{2}/3 < 1$, we absorb constants into the denominator:
\[
\mathbb{P}(|d_j| > t) \leq 4 \exp\left(-\frac{t^2}{4\sigma^2 + Mt/\sqrt{N}}\right).
\]
\end{proof}

To apply Lemma \ref{lem: bernstein for d} to $x_k = \tilde{a}_k - \bar{a}_k$, observe that $\mathbb{E}x_k = 0$, and as in SV13 (top of p.~4059), $\operatorname{Var}(x_k) = O(n^{-2+4/\gamma})$, while $|x_k| < 2n^{1/\gamma}$. Thus,
\[
\mathbb{P}(|d_j(\tilde{\mb{a}} - \bar{\mb{a}})| > t^{1/\gamma}) \leq 4 \exp\left( -\Omega(n^{1/2 - 1/\gamma}) t^{1/\gamma} \right).
\]
Using this in a union bound (as in (11) of SV13),
\begin{align*}
\mathbb{E} \left[ \max_{0 \leq j \leq \lfloor N/2 \rfloor} |d_j(\tilde{\mb{a}} - \bar{\mb{a}})|^\gamma \right] 
&\leq 1 + (N/2 + 1) \int_1^\infty 4 \exp\left( -\Omega(n^{1/2 - 1/\gamma}) t^{1/\gamma} \right) \mathrm{d}t \\
&\leq 1 + N \cdot \Omega(n^{1/\gamma - 1/2}) \exp\left( -\Omega(n^{1/2 - 1/\gamma}) \right) \rightarrow 1.
\end{align*}
Using \eqref{preparation for part 3} and $N = O(n)$, we obtain
\[
\mathbb{E}\left| \frac{\|\sqrt{N}(\mb{D}(\tilde{\mb{a}}) - \mb{D}(\bar{\mb{a}}))\|}{\sqrt{p \log n}} \right|^\gamma = O\left( \frac{n}{p \log n} \right)^{\gamma/2} \rightarrow 0,
\]
which completes the proof of \eqref{part 3 lemma 2}, and thus of part (i) of Lemma \ref{lem:reductions}.

\subsubsection*{Part (ii) (strengthening to $L^\gamma$ convergence)}

To simplify notation, we now drop the bar over $a_j$ and assume without loss of generality that $a_j$, $j \in \mathbb{Z}$, satisfies
\begin{equation}
\label{new a}
|a_j| < n^{1/\gamma}.
\end{equation}
The following lemma proves part (ii) of Lemma~\ref{lem:reductions}, showing that convergence in probability strengthens to $L^\gamma$ convergence in Theorem~\ref{T theorem}.

\begin{lemma}\label{P implies Lgamma}
If, in addition to \eqref{Lyapunov} and \eqref{new asymptotic regime}, we assume that $\limsup_{n\rightarrow\infty}n/p <\infty$, then the following implication holds:
\[
\frac{\|\mb{T}_{p\times n}\|}{\sqrt{p\log n}} - K_{1,n/p} \overset{P}{\rightarrow} 0  
\Longrightarrow   
\frac{\|\mb{T}_{p\times n}\|}{\sqrt{p\log n}} - K_{1,n/p} \overset{L^\gamma}{\rightarrow} 0.
\]
\end{lemma}

\begin{proof}
     We adapt the approach of Lemma 7 in SV13. 
     Let 
\[
X_n := Y_n - K_{1,n/p}, \qquad Y_n := \frac{\sqrt{N} \| \mb{P}_{p,N} \mb{D}(\mb{a}) \mb{P}_{n,N} \|}{\sqrt{p\log n}}.
\]
By Theorem 4.5.4 of \cite{chung01}, if $X_n\overset{P}\rightarrow 0$ and $\mathbb{E}|X_n|^\gamma<\infty$ for all $n$, then $X_n\overset{L^\gamma}\rightarrow 0$ as long as $\{X_n^\gamma\}$ is uniformly integrable, that is if
 \[
 \lim_{b\rightarrow \infty}\sup_n\mathbb{E}\left\{|X_n|^\gamma\mathbf{1}\{|X_n|^\gamma>b\}\right\}=0.
 \]
 Note that, for any $b>0$,
 \[
\mathbb{E}\{|X_n|^\gamma\mathbf{1}\left\{|X_n|^\gamma>b\}\right\}=\int_{b}^\infty \mathbb{P}\{|X_n|^\gamma>t\}\mathrm{d}t.
 \]

The elementary inequality
\[
|Y_n-K_{1,n/p}|^\gamma\leq 2^\gamma(|Y_n|^\gamma+|K_{1,n/p}|^\gamma)
\]
and the fact that $\sup_{c\geq 1}|K_{1,c}|<\infty$ imply that the uniform integrability of $\{X_n^\gamma\}$ follows from that of $\{Y_n^\gamma\}$. Let us establish the uniform integrability of $\{Y_n^\gamma\}$.

We have 
\[
|Y_n|\leq \max_{0\leq j\leq \lfloor N/2\rfloor}\frac{\left|d_j(\mb{a})\right|}{\sqrt{\frac{p}{N}\log n}}.
\]
Therefore, 
\begin{eqnarray}\label{uniform integrability}
\mathbb{E}\left\{|Y_n|^\gamma\mathbf{1}\{|Y_n|^\gamma>b\}\right\}\leq  \sum_{j=0}^{\lfloor N/2\rfloor}\int_b^\infty\mathbb{P}\left\{\frac{\left|d_j(\mb{a})\right|}{\sqrt{\frac{p}{N}\log n}}>t^{1/\gamma}\right\}\mathrm{d}t.
\end{eqnarray}
On the other hand, by Lemma \ref{lem: bernstein for d}, for all $t\geq 0$
 \begin{eqnarray}\label{prob in uniform integrability}
     \mathbb{P}\left\{\frac{|d_j(\mb{a})|}{\sqrt{\frac{p}{N}\log 
 n}}>t^{1/\gamma}\right\}\leq 4\exp\left(-\frac{t^{2/\gamma} \frac{p}{n}\log n }{4N/n+t^{1/\gamma}n^{1/\gamma-1/2}\sqrt{\frac{p}{n}\log n }}\right).
 \end{eqnarray}
Let $t_n=n^{\gamma/2-1}\left(\frac{p}{n}\log n\right)^{-\gamma/2}$, then the denominator in \eqref{prob in uniform integrability} evaluated at $t=t_n$ satisfies inequality
\[
4N/n+t_n^{1/\gamma}n^{1/\gamma-1/2}\sqrt{\frac{p}{n}\log n }=4N/n+1\leq 9.
\]
This inequality and the assumption of the lemma that $\limsup_{n\rightarrow\infty}(n/p)<\infty$ imply that there exists a constant $C>0$ such that, for all sufficiently large $n$, we have
 \begin{equation}
 \label{integral split}
 \int_b^\infty\mathbb{P}\left\{\frac{\left|d_j(\mb{a})\right|}{\sqrt{\frac{p}{N}\log n}}>t^{1/\gamma}\right\}\mathrm{d}t\leq \int_b^{t_n}\exp\left(-Ct^{2/\gamma}\log n\right)\mathrm{d}t+\int_{t_n}^{\infty}\exp\left(-Ct^{1/\gamma}\right)\mathrm{d}t.
 \end{equation}
 Both integrals can be linked to the incomplete gamma function $\Gamma(r,x):=\int_x^\infty s^{r-1}e^{-s}\mathrm{d}s$. Indeed, for the second integral, we have
 \[
 \int_{t_n}^{\infty}\exp\left(-Ct^{1/\gamma}\right)\mathrm{d}t=\frac{\gamma}{C^\gamma}\Gamma(\gamma,Ct_n^{1/\gamma})=\frac{\gamma}{C}t_n^{1-1/\gamma}\exp\left(-Ct_n^{1/\gamma}\right)(1+o(1)),
 \]
 where for the last equality we used the standard asymptotic formula for the incomplete Gamma function (see e.g.~\cite{olver97}, p.~66). Note that under the assumptions of the lemma, $t_n$ grows as a (possibly small) positive power of $n$. Therefore, the right hand side of the last equality converges to zero faster than any power of $n$.
 
 For the first integral in \eqref{integral split}, we have
 \begin{eqnarray*}
 \int_b^{t_n}\exp\left(-Ct^{2/\gamma}\log n\right)\mathrm{d}t&\leq& \int_b^{\infty}\exp\left(-Ct^{2/\gamma}\log n\right)\mathrm{d}t\\
 &=&\frac{\gamma}{2C^{\gamma/2}\log^{\gamma/2} n}\Gamma\left(\frac{\gamma}{2},Cb^{2/\gamma}\log n\right)\\
 &=&\frac{\gamma b^{1-2/\gamma}}{2C\log n}\exp\left(-Cb^{2/\gamma}\log n\right)(1+o(1)).
 \end{eqnarray*}
 Combining the last two displays, we see that for sufficiently large $b$, we have
 \[
\int_b^\infty\mathbb{P}\left\{\frac{\left|d_j(\mb{a})\right|}{\sqrt{\frac{p}{N}\log n}}>t^{1/\gamma}\right\}\mathrm{d}t=O(n^{-100}),
 \]
 say. Using this in \eqref{uniform integrability}, we obtain, for sufficiently large $b$,
 \[
 \mathbb{E}\left\{|Y_n|^\gamma\mathbf{1}\{|Y_n|^\gamma>b\}\right\}=0.
 \]
 Hence, $\{Y_n^\gamma\}$ is uniformly integrable. 
 
 To show that  $X_n\overset{P}\rightarrow 0$ implies  $X_n\overset{L^\gamma}\rightarrow 0$, it remains to demonstrate that $\mathbb{E}|X_n|^\gamma<\infty$ for all sufficiently large $n$. But
 \[
 \mathbb{E}|X_n|^\gamma\leq 2^\gamma(\mathbb{E}|Y_n|^\gamma+|K_{1,n/p}|^\gamma)\leq 2^\gamma\left(\mathbb{E}\left\{|Y_n|^\gamma\mathbf{1}\{|Y_n|^\gamma>b\}\right\}+b+\sup_{c\geq 1}K_{1,c}\right).
 \]
 As we have just shown, the first component of the latter sum is finite for sufficiently large $b$ and all sufficiently large $n$. Hence, indeed, $\mathbb{E}|X_n|^\gamma<\infty$ for all sufficiently large $n$. 

\end{proof}

Note that the argument of the above proof breaks down when 
$n/p$ is allowed to diverge to infinity, because the inequality
\[
Y_n:=\frac{\sqrt{N}\|\mb{P}_{p,N}\mb{D}(\mb{a})\mb{P}_{n,N}\|}{\sqrt{p\log n}}\leq \frac{\sqrt{N}\|\mb{D}(\mb{a})\|}{\sqrt{p\log n}},
\]
used to establish the uniform integrability of $Y_n$, becomes too crude in that regime. Extending the result to the regime $n/p\rightarrow\infty$ remains open.

\subsubsection*{Part (iii) (Assuming that $N$ is even)}

 Suppose $N$ is odd. Then, instead of the matrix $\mb{T}_{p\times n}$, consider matrix
 $\mb{T}_{p\times (n+1)}$ that adds an extra column, call it $\mb{v}_T$, to the original matrix. Note that
 \[
|\|\mb{T}_{p\times n}\|-\|\mb{T}_{p\times (n+1)}\||\leq \|\mb{v}_T\|.
 \]
 On the other hand, the identity $\mathbb{E}\|\mb{v}_T\|^2=p$ and Markov's inequality yield $ \|\mb{v}_T\|=o_P(\sqrt{p\log n})$. As a consequence,
 \[
 \frac{\|\mb{T}_{p\times n}\|}{\sqrt{p\log n}}-\frac{\|\mb{T}_{p\times (n+1)}\|}{\sqrt{p\log (n+1)}}=o_P(1).
 \]
 The latter identity and the fact that $K_{1,c}$ is continuous in $c$ immediately imply that Theorem \ref{T theorem} holds for odd $N$ as long as it holds for even $N+1$. 

 \subsection{Proof of Proposition \ref{equivalent prop 8} (Properties of the partition of $\{0,\dots,N-1\}$ )}
This proof adapts the proof of Proposition 8 in SV13 to our setting. For any $s_n=O(\log^\alpha n)$ and $0\leq j_1<j_2<\cdots<j_{s_n}\leq N/2$, we have 
\begin{eqnarray*}
    \mathbb{P}\left\{|d_{j_i}|>\epsilon_n \sqrt{\log n}, 1\leq i\leq s_n\right\}&\leq& \sum_{\xi_i\in\{\Re d_{j_i},\Im d_{j_i}\}}\mathbb{P}\left\{|\xi_{i}|>\epsilon_n\sqrt{(\log n)/2}, 1\leq i\leq s_n\right\}\\
    &\leq&\sum_{\xi_i\in\{\Re d_{j_i},\Im d_{j_i}\}}\sum_{\beta_i\in\{-1,+1\}}\mathbb{P}\left\{\sum_{i=1}^{s_n}\beta_i\xi_{i}>s_n\epsilon_n\sqrt{(\log n)/2}\right\}.
\end{eqnarray*}
Using the definition \eqref{def of d again} of $d_j$, we can represent $\sum_{i=1}^{s_n}\beta_i\xi_{i}$ as a linear combination $\sum_{k=0}^{N-1}\theta_k a_k$ of random variables $a_k$ with weights $\theta_k$ of order $s_n\times O(n^{-1/2})$. Since $|a_k|<n^{1/\gamma}$, Bernstein's inequality yields, for any $t\geq 0$:
\begin{equation}\label{general bernstein}
\mathbb{P}\left\{\left|\sum_{k=0}^{N-1} \theta_k a_k\right| \geq t \right\} \leq 2 \exp\left\{ -\frac{t^2/2}{\sum_{k=0}^{N-1} \theta_k^2 + \max_{0 \leq k < N} |\theta_k| n^{1/\gamma} t / 3} \right\}.
\end{equation}
Therefore,
\begin{equation}\label{for (17)}
\mathbb{P}\left\{\sum_{i=1}^{s_n}\beta_i\xi_{i}>s_n\epsilon_n\sqrt{\frac{\log n}{2}}\right\}\leq 2\exp\left(-\frac{s_n^2\epsilon_n^2\log n/4}{ Var(\sum_{i=1}^{s_n}\beta_i\xi_i)+|O(s_nn^{1/\gamma-1/2})|s_n\epsilon_n\sqrt{\log n}}\right).
\end{equation}
On the other hand,
\[
\operatorname{Var}\left(\sum_{i=1}^{s_n}\beta_i\xi_i\right)=\sum_{i=1}^{s_n}\operatorname{Var}\xi_i\leq \sum_{i=1}^{s_n}\operatorname{Var}d_{j_i} \leq s_n,
\]
and therefore, for all sufficiently large $n$, we have
\[
\operatorname{Var}\left(\sum_{i=1}^{s_n}\beta_i\xi_i\right)+|O(s_nn^{1/\gamma-1/2})|s_n\epsilon_n\sqrt{\log n}\leq 2s_n.
\]
This implies that the right hand side of \eqref{for (17)} is $O\left(n^{-s_n\epsilon_n^2/8}\right)$.
The sums over $\xi_i\in\{\Re d_{j_i},\Im d_{j_i}\}$ and $\beta_i\in\{-1,+1\}$ increase this at most by a factor $2^{2s_n}=n^{o(s_n\epsilon_n^2)}$. Thus, finally, we have the following equivalent of SV13's (17):
\begin{equation}
\label{equivalent of (17)}
\mathbb{P}\left\{|d_{j_i}|>\epsilon_n\sqrt{\log n}, 1\leq i\leq s_n\right\}\leq O(n^{-s_n\epsilon_n^2/9}).
\end{equation}

Continuing to closely follow SV13's proof of their proposition 8, we note that if one of the statements of Proposition \ref{equivalent prop 8} fails, then one of the following holds:
\begin{itemize}
    \item[(i)] either of the bricks $L_0$ or $L_{-m_n}$ is visible (this takes into account a possibility that $J\cap S\neq \varnothing$ but $J\notin\mathcal{L}$ because it contains a visible brick $L_0$ or a visible brick $L_{-m_n}$);
    \item[(ii)] there exists a stretch of $M_n$ consecutive bricks from $L_{-m_n+1},\dots,L_{-1}$ such that at least $\lfloor M_n/2\rfloor-1$ of them are visible (this corresponds to a possibility that a block $J\in \Lambda$ is too long to be admissible);
    \item[(iii)] there exists a stretch of $M_n$ consecutive bricks ($L_a,\dots,L_{a+M_n-1}$) from $L_{-m_n+1},\dots,L_{-1}$ such that \newline $\sum_{i=0}^{M_n-1}\#(L_{a+i}\cap S)\geq M_n$ (this takes into account a possibility that the second statement of the proposition is violated).
\end{itemize}
By \eqref{equivalent of (17)}, the probability of the event (i) is $O(n^{-\epsilon_n^2/10})=o(1)$. Further, both (ii) and (iii) are contained in 
\begin{itemize}
    \item[(iv)] there exists a stretch of $M_n$ consecutive bricks ($L_a,\dots,L_{a+M_n-1}$) from $L_{-m_n+1},\dots,L_{-1}$ such that \newline $\sum_{i=0}^{M_n-1}\#(L_{a+i}\cap S)\geq \lfloor M_n/2\rfloor-1$.
\end{itemize}
If we fix such $a$, let $s_n=\lfloor M_n/2\rfloor -1$, and fix positions $j_1,\dots,j_{s_n}$ within the block $L_a\cup L_{a+1}\cup\dots\cup L_{a+M_n-1}$, then by \eqref{equivalent of (17)}, the probability that $j_1,\dots,j_{s_n}\in S$ is bounded above by
\[
O(n^{-(\lfloor M_n/2\rfloor -1)\epsilon_n^2 /9})=O\left(n^{-\frac{\lfloor 50\log^\alpha n\rfloor -1}{9\log^\alpha n}}\right)=O(n^{-5}). 
\]
By union bound,
\[
\mathbb{P}(\text{event (iv)})\leq m_n\times\binom{4M_nr_n}{s_n}O(n^{-5}).
\]
On the other hand 
\[
\binom{4M_nr_n}{s_n}=O((4M_nr_n)^{M_n})=O((\log N)^{5M_n})
\]
and $m_n<N=O(n)$. Therefore, 
\[
\mathbb{P}(\text{event (iv)})\leq O(n^{-4}(\log n)^{1000\log^\alpha n})=o(1).
\]
This completes the proof.

\subsection{Proof of Lemma \ref{equivalent of lemma 9} (block-diagonal reduction)}
     This proof is only marginally different from the proof of Lemma 9 in SV13. Note that the absolute value of the numerator in the ratio of interest is bounded above by
     \begin{equation}\label{eq from lemma 9}
 \|\mb{P}_{p,N}\mb{D}^\epsilon \mb{P}_{n,N}-\mb{B}_{p,N}\mb{D}^\epsilon \mb{B}_{n,N}\|\leq \|(\mb{P}_{p,N}-\mb{B}_{p,N})\mb{R}\|\|\mb{D}\|+\|\mb{D}\|\|\mb{R}(\mb{P}_{n,N}-\mb{B}_{n,N})\|,
 \end{equation}
 where $\mb{R}:=\operatorname{diag} (\mathbf{1}_{\{j\in S\}})$.
The maximal column sum of matrix $(\mb{P}_{p,N}-\mb{B}_{p,N})\mb{R}$ is obviously not larger than the maximum column sum of $\mb{P}_{p,N}$, which is no larger than $C\log n$ for some constant $C$. Such an upper bound follows from \eqref{bound on the entries of P}.

Consider now the row sums. Let $k\in J$, where $J\in\Lambda$, then 
\[
\sum_{l=0}^{N-1}|((\mb{P}_{p,N}-\mb{B}_{p,N})\mb{R})_{kl}|=\sum_{l\notin J}|(\mb{P}_{p,N})_{kl}|\mb{R}_{ll}.
\]
By Proposition \ref{equivalent prop 8}, with probability approaching one (with high probability), each $J\in\Lambda$ has at most $M_n$ elements $j$ where $R_{jj}$ is non-zero. Furthermore, different blocks with visible bricks in them, start and end with invisible bricks (each having at least $r_n=\lceil\log N\rceil^4$ elements). This and \eqref{bound on the entries of P} imply (in the same way as a similar inequality is implied in SV13's proof of their Lemma 9) that, with high probability
\[
\sum_{l\notin J}|(\mb{P}_{p,N})_{kl}|\mb{R}_{ll}\leq \sum_{k=1}^{\#\Lambda-1}\frac{CM_n}{k(\log N)^4}\leq C(\log n)^{\alpha-3}.
\]
Here and elsewhere, $C$ may denote different constants from one appearance to another.
Hence, as in SV13, bounding the spectral norm by the geometric mean of the maximal row and columns sums, we obtain with high probability,
\[
\|(\mb{P}_{p,N}-\mb{B}_{p,N})\mb{R}\|\leq C(\log n)^{\alpha-2}.
\]
Finally, by Bernstein's inequality, $\|\mb{D}\|\leq C\sqrt{\log n}$ with high probability. Therefore, 
\[
\|(\mb{P}_{p,N}-\mb{B}_{p,N})\mb{R}\|\|\mb{D}\|=O_{P}\left((\log n)^{\alpha-3/2} \right).
\]
One shows similarly that 
\[
\|\mb{D}\|\|\mb{R}(\mb{P}_{n,N}-\mb{B}_{n,N})\|=O_{P}\left((\log n)^{\alpha-3/2} \right).
\] Using the latter two displays in \eqref{eq from lemma 9} yields
\[
\frac{\|\mb{P}_{p,N}\mb{D}^\epsilon\mb{P}_{n,N}\|-\|\mb{B}_{p,N}\mb{D}^\epsilon\mb{B}_{n,N}\|}{\sqrt{(p/N)\log n}}=\frac{O_{P}\left((\log n)^{\alpha-2} \right)}{\sqrt{p/n}}=o_P\left((\log n)^{3\alpha/2-2}\right),
\]
where the last equality follows from \eqref{new asymptotic regime}.

\subsection{Proof of Lemma \ref{cor: corollary 13} (an upper bound on $|d_{j_1}|^2+\dots+|d_{j_{M}}|^2$ from admissible blocks)}
We begin by establishing counterparts of Lemmas 11 and 12 from SV13. Lemma \ref{cor: corollary 13} will then be a simple consequence of these results. 

  \begin{lemma}\label{lem: refined equiv of lemma 11}
  Let $k_n=o(\log n/\log \log n)$ be positive integers, and let $\beta=(\beta_1,\dots,\beta_{k_n})^\top$ with each $\beta_i\geq 0$. Then for any fixed small $\tau>0$ and all $0<j_1<\dots<j_{k_n}<N/2$, for all sufficiently large $n$, we have
     \begin{equation}
     \label{equivalent of lemma 11}
     \mathbb{P}\{|d_{j_i}|>\beta_i \sqrt{\log n}, i=1,\dots,k_n\}\leq \max\{n^{-\log n/4},n^{-\|\beta\|^2+\tau}\},
     \end{equation}
     where $d_j$ are as defined in \eqref{def of d again}.     
 \end{lemma}
 \begin{proof}
     For $\|\beta\|^2\leq \tau$ the statement of the lemma is trivial, so we will assume $\|\beta\|^2>\tau$. To bound $\mathbb{P}\{|d_{j_i}|>\beta_i \sqrt{\log n}, i=1,\dots,k_n\}$ from above using Bernstein's inequality, we are going to use a simple trick that we learned from \cite{hannan80} (p.~1076--1077). SV13 do not need this trick because their $d_{j_i}$ are real-valued.

     Consider a grid of $m$ equidistant points on $[0,2\pi]$: $2\pi s/m, s=1,\dots, m$. Let 
$\alpha_j$ be the closest point on the grid to the argument of the complex number $d_j$. Then,
\begin{equation}\label{Hannan trick}
|d_j|^2=\frac{|\Re d_j\times \cos \alpha_j+\Im d_j\times\sin\alpha_j|^2}{\cos^2(\alpha_j-\arg d_j)}.
\end{equation}
Let $m:=\lceil 2\pi\|\beta\|\sqrt{\log n}\rceil$.
Then
\[
 |\alpha_{j_i}-\arg d_{j_i}|\leq \frac{2\pi}{m}\leq \frac{1}{\|\beta\|\sqrt{\log n}}.
 \]
 Since, for any $x\in\mathbb{R}$, $\cos x\geq 1-x^2/2$, we have
 \[
 |\cos (\alpha_{j_i}-\arg d_{j_i})|\geq 1-\frac{1}{2 \|\beta\|^2\log n}=:\frac{t_n}{\|\beta\|\sqrt{\log n}}.
 \]
Using this in \eqref{Hannan trick} yields
\[
|d_{j_i}|\leq \frac{\|\beta\|\sqrt{\log n}}{t_n}\left|\Re d_{j_i}\times \cos \alpha_{j_i}+\Im d_{j_i}\times \sin \alpha_{j_i}\right|,\qquad i=1,\dots,k_n.
\]
Therefore, 
\begin{eqnarray*}
&&\mathbb{P}\{|d_{j_i}|>\beta_i \sqrt{\log n}, \quad i=1,\dots,k_n\}\\
&&\leq \mathbb{P}\left\{|\Re d_{j_i}\cos \alpha_{j_i}+\Im d_{j_i}\sin \alpha_{j_i}|>\beta_i t_n/\|\beta\|, \quad i=1,\dots,k_n\right\}\\
&&\leq\sum_{1\leq s_1,\dots,s_{k_n}\leq m}\mathbb{P}\left\{\left|\Re d_{j_i}\cos \frac{2\pi s_i}{m}+\Im d_{j_i}\sin \frac{2\pi s_i}{m}\right|>\beta_i t_n/\|\beta\|, \quad i=1,\dots,k_n\right\}\\
&&\leq \sum_{1\leq s_1,\dots,s_{k_n}\leq m}\sum_{\beta_i\in\{-1,+1\}}\mathbb{P}\left\{\sum_{i=1}^{k_n}\beta_i\psi_{i}\beta_i/\|\beta\|>t_n\right\},
\end{eqnarray*}
where 
\begin{eqnarray*}
\psi_{i}&:=&\Re d_{j_i}\cos \frac{2\pi s_i}{m}+\Im d_{j_i}\sin \frac{2\pi s_i}{m}\\
&=&\Re \left(e^{-\frac{2\pi\rmi s_i}{m}}d_{j_i}\right)=\frac{1}{\sqrt{N}}\sum_{r=0}^{N-1}\cos\left\{2\pi\left(\frac{j_i r}{N}-\frac{s_i}{m}\right)\right\}a_r.
\end{eqnarray*}
The facts about $d_j$ described in Section \ref{sec:circulant embedding} imply that $\psi_{i}, i=1,\dots,k_n$, are mutually uncorrelated and have variance $1/2$.
The sum $\sum_{i=1}^{k_n}\beta_i\psi_{i}\beta_i/\|\beta\|$ has variance $1/2$ as well, and can be represented as a linear combination of $a_r$ with coefficients $\theta_r$ such that $|\theta_r|\leq k_n/\sqrt{N}$ and $\sum_{r=0}^{N-1}\theta_r^2=1/2$. Therefore, Bernstein's inequality yields
\begin{equation}\label{Bernstein for lemma 11}
\mathbb{P}\{|d_{j_i}|>\beta_i \sqrt{\log n}, i=1,\dots,k_n\}\leq m^{k_n} 2^{k_n}\exp\left(-\frac{t_n^2}{1+\psi_n t_n}\right),
\end{equation}
with $\psi_n:=2k_n n^{1/\gamma}/(3\sqrt{N})$.

Recall that we have restricted attention to cases with $\|\beta\|^2>\tau$. Suppose, in addition, that $\|\beta\|^2\leq \log n$. In such cases,  we have $\psi_n t_n^r=o(1)$ for any fixed $r>0$ because $\psi_n$ decays as a power of $1/n$ while  $t_n$ grows as a power of $\log n$. Using this fact we observe that, for all sufficiently large $n$,
\begin{eqnarray*}
\exp\left(-\frac{t_n^2}{1+\psi_n t_n}\right)&\leq& \exp\left(-t_n^2+\psi_n t_n^3\right)\\
&\leq& \exp\left(-\|\beta\|^2\log n+1+o(1)\right).
\end{eqnarray*}
On the other hand, since $m=\lceil 2\pi\|\beta\|\sqrt{\log n}\rceil$ and that $k_n=o(\log n/\log\log n)$, we have
\[
m^{k_n} 2^{k_n}\leq m^{2k_n}\leq n^{o(1)}.
\]
Using the latter two displays in \eqref{Bernstein for lemma 11}, we obtain
\[
\mathbb{P}\{|d_{j_i}|>\beta_i \sqrt{\log n}, i=1,\dots,k_n\}\leq n^{o(1)}\exp\left(-\|\beta\|^2\log n+1+o(1)\right)\leq n^{-\|\beta\|^2+\tau},
\]
for all sufficiently large $n$. This yields the statement of the lemma when $\|\beta\|^2\leq \log n$.

It remains to consider cases where $\|\beta\|^2>\log n$. In such cases, for all sufficiently large $n$, we have
\begin{equation}\label{second case for lamma 11}
\frac{t_n^2}{1+\psi_n t_n}>\frac{1}{2}\|\beta\|\log^{3/2} n.
\end{equation}
To establish this inequality, it is sufficient to show that
\begin{equation}\label{sufficient for lemma 11}
\frac{2}{3}t_n^2>\frac{1}{2}\|\beta\|\log^{3/2}n,\quad\text{and}\quad \frac{1}{3}t_n^2>\psi_n t_n \frac{1}{2}\|\beta\|\log^{3/2}n.
\end{equation}
By definition, 
\[
t_n=\|\beta\|\sqrt{\log n}-\frac{1}{2\|\beta\|\sqrt{\log n}}.
\]
Therefore, for all sufficiently large $n$, we have
\[
\frac{2}{3}t_n^2>\frac{2}{3}\left(\|\beta\|^2\log n-1\right)>\frac{1}{2}\|\beta\|\log^{3/2}n,
\]
where we have used the maintained constraint $\|\beta\|^2>\log n$. This proves the first inequality in \eqref{sufficient for lemma 11}. To establish the second inequality, note that
\[
\frac{1}{3}t_n^2>t_n\frac{1}{6}\|\beta\|\sqrt{\log n}>\psi_n t_n \frac{1}{2}\|\beta\|\log^{3/2}n,
\]
for all sufficiently large $n$, where the last inequality uses the fact that $\psi_n\log n=o(1)$. Hence, both inequalities in \eqref{sufficient for lemma 11} hold, which yields \eqref{second case for lamma 11}.

On the other hand, for all sufficiently large $n$,
\[
m^{k_n}2^{k_n}\leq m^{2k_n}=\exp\left(\log (\lceil 2\pi \|\beta\|\sqrt{\log n}\rceil)2k_n\right)\leq \exp\left(5k_n\log \|\beta\|\right)\leq \exp\left(\|\beta\|\log n\right).
\]
Using this inequality and \eqref{second case for lamma 11} in \eqref{Bernstein for lemma 11}, we obtain
\begin{eqnarray*}
\mathbb{P}\{|d_{j_i}|>\beta_i \sqrt{\log n}, i=1,\dots,k_n\}&\leq&\exp\left(\|\beta\|\log n-\frac{1}{2}\|\beta\|\log^{3/2}n\right)\\
&\leq& \exp\left(-\frac{1}{4}\|\beta\|\log^{3/2}n\right)\leq n^{-\log n/4},
\end{eqnarray*}
where the latter inequality follows from the fact that $\|\beta\|>\sqrt{\log n}$. Hence, the statement of the lemma holds in full generality.
 \end{proof}


\begin{lemma}\label{lem: lemma 12}
    Fix $\eta>0$, let $k_n=o(\log n/\log \log n)$ be positive integers. Then there exists a constant $C>0$ such that for all $0<j_1<j_2<\dots<j_{k_n}<N/2$ and all sufficiently large $n$, we have
\begin{eqnarray*}
    &&\mathbb{P}\left\{|d_{j_1}|>\beta_1\sqrt{\log n},\dots,|d_{j_{k_n}}|>\beta_{k_n}\sqrt{\log n}\right.
    \\
    &&\quad \left.\text{ for some }\beta_1,\dots,\beta_{k_n}>0 \text{ s.t. } \beta_1^2+\dots+\beta_{k_n}^2\geq 1+\eta\right\}\leq Cn^{-1-\eta/3}.
\end{eqnarray*}
\end{lemma}

\begin{proof}
         The proof is nearly identical to  SV13's proof of their Lemma 12. Namely, construct an $\frac{\eta}{4(1+\eta)k_n}$-net $\mathcal{N}$ for the interval $[0,1+\eta]$ by choosing $\left\lceil\frac{4(1+\eta)^2k_n}{\eta}\right\rceil+1$ equally spaced points in $[0,1+\eta]$ including $0$ and $1+\eta$. Given any $\beta_1,\dots,\beta_{k_n}\in[0,1+\eta]$, we can find $\alpha_1,\dots,\alpha_{k_n}\in\mathcal{N}$ such that 
\[
\beta_i-\frac{\eta}{4(1+\eta)k_n}\leq\alpha_i\leq\beta_i
\]
for each $i$. This yields
\[
\alpha_1^2+\dots+\alpha_{k_n}^2\geq \|\beta\|^2-\frac{\eta}{2}.
\]
The event 
\[
\left\{|d_{j_1}|>\beta_1\sqrt{\log n},\dots,|d_{j_{k_n}}|>\beta_{k_n}\sqrt{\log n}\text{ for some }\beta_1,\dots,\beta_{k_n}\in(0,1+\eta]\text{ s.t. } \|\beta\|^2\geq 1+\eta\right\}
\]
is contained in the  union of events 
\[
\{|d_{j_1}|>\alpha_1\sqrt{\log n},\dots,|d_{j_{k_n}}|>\alpha_{k_n}\sqrt{\log n}\}
\]
taken over all tuples $(\alpha_1,\dots,\alpha_{k_n})\in\mathcal{N}^{k_n}$ satisfying $\alpha_1^2+\dots+\alpha_{k_n}^2\geq 1+\eta/2$. By Lemma \ref{lem: refined equiv of lemma 11}, the probability of this union is bounded by 
\[
(Ck_n)^{k_n}O(n^{-1-\eta/2+\tau})=O(n^{-1-\eta/2+\tau/2}),
\]
for arbitrarily small $\tau>0$.  We can choose $\tau$ so that the latter bound becomes $O(n^{-1-\eta/3})$. Finally, again by Lemma \ref{lem: refined equiv of lemma 11}, for arbitrarily small $\tau_1>0$, the probability of event $\{|d_{j_i}|>(1+\eta)\sqrt{\log n}\text{ for some }1\leq i\leq k_n\}$ is no larger than 
\[
k_nO(n^{-(1+\eta)^2+\tau_1})=O(n^{-1-\eta/3}),
\]
where the last equality is achieved by an appropriate choice of $\tau_1$.

\end{proof}

Now we are ready to prove Lemma \ref{cor: corollary 13}. 
 The proof closely follows that of Corollary 13 in SV13. By Lemma \ref{lem: lemma 12}, for $M=M_n$ and a \textit{fixed} admissible block and points $j_i$, the probability that 
 \[
 |d_{j_1}|^2+\cdots+|d_{j_{M}}|^2>(1+\eta)\log n
 \]
 is at most $Cn^{-1-\eta/3}$. Note that, for all sufficiently large $n$,  the number of admissible blocks is smaller than $n$ and the length of any admissible block is at most $4r_nM_n$ with $4r_n=4\lceil \log N\rceil^4$ being the maximal possible length of a brick. Therefore, by union bound, the probability that the statement of Lemma \ref{cor: corollary 13} is violated is at most 
    \[
    n\binom{4r_nM_n}{M_n}Cn^{-1-\eta/3}= O\left((\log n)^{5M_n}n^{-\eta/3}\right)=o(1).
    \]
    This completes the proof.

\subsection{Proof of the identity \eqref{identity from lem link to Ipn} (link between polynomials and positive definite sequences)}
\begin{eqnarray*}
\frac{1}{2\pi}\int_{0}^{2\pi}|\mathcal{P}_{p-1}(e^{\rmi x})\mathcal{P}_{n-1}(e^{\rmi x})|^2\mathrm{d}x&=&\sum_{\nu_1,\nu_2,\nu_3,\nu_4}\frac{1}{2\pi}\int_0^{2\pi}\beta_{\nu_1}e^{\rmi x\nu_1}\overline{\beta_{\nu_2}}e^{-\rmi x\nu_2}\rho_{\nu_3}e^{\rmi x\nu_3}\overline{\rho_{\nu_4}}e^{-\rmi x\nu_4}\mathrm{d}x\\
&=&\sum_{\nu_1-\nu_2+\nu_3-\nu_4=0}\beta_{\nu_1}\overline{\beta_{\nu_2}}\rho_{\nu_3}\overline{\rho_{\nu_4}}\\
&=&\sum_{j,\nu_2,i,\nu_4:i+j=0}\beta_{j+\nu_2}\overline{\beta_{\nu_2}}\rho_{i+\nu_4}\overline{\rho_{\nu_4}}\\
&=&\sum_{i+j=0}\alpha_j\gamma_i=\sum_{i=-p+1}^{p-1}\alpha_{-i}\gamma_{i}=\sum_{\nu=-p+1}^{p-1}\overline{\alpha_{\nu}}\gamma_{\nu}.
\end{eqnarray*}

\subsection{Proof of Lemma \ref{lem: theorem 5.1} (existence of monotone extremal sequences claimed in the proof of Lemma \ref{lem:I is close to K})}\label{sec: extremal seq}

As follows from the Herglotz lemma (see a discussion on p.~6 of \cite{belov13}), $\mathcal{L}_k$ can be described as the set of all sequences $\{\alpha_j\}$ such that $\alpha_0=1$, $\alpha_j=0$ for all $|j|>k$, and the following trigonometric polynomial is non-negative:
\begin{equation}\label{positive poynomial}
T(x):=\sum_{j=-k}^k\alpha_je^{\rmi j x}\geq 0\text{ for all }x.
\end{equation}
By Fej\'{e}r-Riesz theorem (e.g. Theorem 1 in \cite{hussen21}), there exist complex numbers $\beta_0,\dots,\beta_k$ such that
\begin{equation}\label{RF theorem}
T(x)=\left|\sum_{j=0}^k\beta_je^{\rmi j x}\right|^2.
\end{equation}
The following lemma is a reformulation of Lemma 5.2 in \cite{belov13} adapted to our setting.
\begin{lemma}[Lemma 5.2 of \cite{belov13}]\label{lem: belov lemma 5.2}
    Consider equations \eqref{positive poynomial} and \eqref{RF theorem}, and let 
    \[
    g_j=\left(\frac{|\beta_j|^2+|\beta_{k-j}|^2}{2}\right)^{1/2}\text{ for all }j=0,\dots,k.
    \]
    Then the coefficients of the polynomial 
    $
    \sum_{j=-k}^kh_j e^{\rmi jx}:=\left|\sum_{j=0}^k g_je^{\rmi j x}\right|^2
    $
    satisfy the following properties:
    \[
    h_0=\alpha_0\quad\text{and}\quad |\alpha_j|\leq h_j\quad\text{ for all }j=-k,\dots,k. 
    \]
\end{lemma}

Lemma \ref{lem: belov lemma 5.2} immediately implies that for any sequence $\{\alpha_j\} \in \mathcal{L}_{k}$, there exists a non-negative sequence $\{h_j\} \in \mathcal{L}_{k}$ such that $|\alpha_j| \leq h_j$ for all $j \in \mathbb{Z}$. In particular, the domain of the maximization problem \eqref{maximisation Ipn2} defining $I_{p,n}^2$ can be restricted to sequences in $\mathcal{L}_{p-1}^+$ and $\mathcal{L}_{n-1}^+$, where $\mathcal{L}_{k}^+$ denotes the subset of $\mathcal{L}_{k}$ consisting of non-negative sequences. That is,
\begin{equation}\label{recall Ipn}
    I_{p,n}^2 = \max_{\{\alpha_j\} \in \mathcal{L}_{p-1}^+, \{\gamma_j\} \in \mathcal{L}_{n-1}^+} \sum \alpha_j \gamma_j.
\end{equation}

    Moreover, there exist extremal sequences $\{\alpha_j\} \in \mathcal{L}_{p-1}^+$ and $\{\gamma_j\} \in \mathcal{L}_{n-1}^+$ such that the corresponding trigonometric polynomials
\begin{equation}\label{T polynomials}
T_\alpha(x) = \sum_{j=-p+1}^{p-1} \alpha_j e^{\rmi jx}, \quad \text{and} \quad T_\gamma(x) = \sum_{j=-n+1}^{n-1} \gamma_j e^{\rmi jx}
\end{equation}
admit representations of the form
\begin{equation}\label{T representations}
T_\alpha(x) = \left| \sum_{j=0}^{p-1} \beta_j e^{\rmi jx} \right|^2, \quad \text{and} \quad T_\gamma(x) = \left| \sum_{j=0}^{n-1} \rho_j e^{\rmi jx} \right|^2,
\end{equation}
where all $\beta_j$ and $\rho_j$ are non-negative and satisfy the symmetry conditions
\begin{align}
\beta_{p-1-j} &= \beta_j, \quad \text{for all } j = 0, \dots, p-1, \label{beta equalities} \\
\rho_{n-1-j} &= \rho_j, \quad \text{for all } j = 0, \dots, n-1. \label{rho equalities}
\end{align}

Indeed, the existence of such representations with complex coefficients follows from the Fej\'{e}r–Riesz theorem. However, by Lemma~\ref{lem: belov lemma 5.2}, the transformations 
\[
\beta_j \mapsto \left( \frac{|\beta_j|^2 + |\beta_{p-1-j}|^2}{2} \right)^{1/2}, \quad
\rho_j \mapsto \left( \frac{|\rho_j|^2 + |\rho_{n-1-j}|^2}{2} \right)^{1/2}
\]
yield non-negative, symmetric coefficients and can only increase the value of the objective in the extremal problem \eqref{recall Ipn}.


We now introduce a few useful definitions, following \cite{belov13}. Let $S$ denote the set of all finite, non-negative sequences $\{c_j\}_{j=-\infty}^\infty$ such that $c_0 \geq c_j$ for all $j \in \mathbb{Z}$, and the multiset $\{c_j : j \neq 0\}$ can be partitioned into pairs of equal elements. Let $S_0 \subset S$ denote the subset consisting of symmetric sequences, that is, those satisfying $c_{-j} = c_j$ for all $j \in \mathbb{Z}$.

As an example of a sequence in 
$S_0$, that we be useful later, observe (see e.g.~p.23 in \cite{belov13}) that if \begin{equation}\label{eq 5.11}
\alpha_j=\sum_{i\in\mathbb{Z}}\beta_i\beta_{j+i},
\end{equation}
where $\{\beta_j\}$ is a finite sequence of non-negative numbers, then $\{\alpha_j\}\in S_0$.

Next, let $S^\downarrow$ be the subset of $S_0$ consisting of all sequences $\{c_j\}_{j=-\infty}^\infty$ such that $c_j \geq c_{j+1}$ for all $j \in \mathbb{Z}_+$. For any finite sequence $\{c_j\}$ of non-negative numbers, let $\{c_j^\downarrow\}$ denote the sequence consisting of the same elements, reordered so that
\[
c_0^\downarrow \geq c_1^\downarrow \geq c_{-1}^\downarrow \geq c_2^\downarrow \geq c_{-2}^\downarrow \geq \dots.
\]
Note that if $\{c_j\} \in S$, then $\{c_j^\downarrow\} \in S^\downarrow$.

Our final definition introduces the concept of majorization. For finite non-negative sequences $\{a_j\}_{j=0}^\infty$ and $\{b_j\}_{j=0}^\infty$, we say that $\{a_j\}$ is \emph{majorized} by $\{b_j\}$, written $\{a_j\}_{j=0}^\infty \prec \{b_j\}_{j=0}^\infty$, if the inequality
\[
\sum_{j=0}^\infty a_j v_j \leq \sum_{j=0}^\infty b_j v_j
\]
holds for every sequence $\{v_j\} \in S^\downarrow$.

The following two lemmas reformulate Lemmas 5.6 and 5.7 in \cite{belov13}.
\begin{lemma}[Lemma 5.6 of \cite{belov13}]\label{lem: belov 5.6}
    Let $\{\beta_j\}$ be an arbitrary finite sequence of non-negative numbers, $\{\alpha_j\}$ be a sequence defined by \eqref{eq 5.11}, and $\{\alpha_j^\ast\}$ be a sequence defined by
    \[
    \alpha_j^\ast=\sum_{i\in\mathbb{Z}}\beta_i^\downarrow \beta_{j+i}^\downarrow.
    \]
    Then $\{\alpha_j^\ast\}\in S^\downarrow$ and  $\{\alpha_j^\downarrow\}_{j=0}^\infty\prec \{\alpha_j^\ast\}_{j=0}^\infty$.
\end{lemma}
\begin{lemma}[Lemma 5.7 of \cite{belov13}]\label{lem: belov 5.7}
    Let $\{a_j\}_{j=0}^\infty$, $\{b_j\}_{j=0}^\infty$, $\{a_j^\ast\}_{j=0}^\infty$, and $\{b_j^\ast\}_{j=0}^\infty$ be finite non-increasing non-negative sequences such that $\{a_j\}_{j=0}^\infty\prec \{a_j^\ast\}_{j=0}^\infty$ and $\{b_j\}_{j=0}^\infty\prec \{b_j^\ast\}_{j=0}^\infty$. Then $\{a_jb_j\}_{j=0}^\infty\prec \{a_j^\ast b_j^\ast\}_{j=0}^\infty$.
\end{lemma}

The following corollary completes our proof.

\begin{corollary}\label{cor: theorem 5.1}
    In the problem \eqref{recall Ipn}, there always exist extremal sequences $\{\alpha_j\} \in \mathcal{L}_{p-1}^+$ and $\{\gamma_j\} \in \mathcal{L}_{n-1}^+$ such that
    \[
    \alpha_0 \geq \dots \geq \alpha_{p-1} \geq 0 \quad \text{and} \quad \gamma_0 \geq \dots \geq \gamma_{n-1} \geq 0.
    \]
\end{corollary}

\begin{proof}
    Let $T_\alpha(x)$ and $T_\gamma(x)$ be the trigonometric polynomials \eqref{T polynomials} corresponding to some extremal sequences $\{\alpha_j\} \in \mathcal{L}_{p-1}^+$ and $\{\gamma_j\} \in \mathcal{L}_{n-1}^+$, and let $\{\beta_j\}_{j=0}^{p-1}$ and $\{\rho_j\}_{j=0}^{n-1}$ be the real coefficients in their representations \eqref{T representations}, satisfying the symmetry conditions.

    Extend these sequences to all of $\mathbb{Z}$ by defining $\beta_j = 0$ for $j < 0$ and $j \geq p$, and similarly $\rho_j = 0$ for $j < 0$ and $j \geq n$.

    Define the shifted sequences
    \[
    \tilde{\beta}_j := \beta_{j + \lfloor (p-1)/2 \rfloor}, \quad \tilde{\rho}_j := \rho_{j + \lfloor (n-1)/2 \rfloor}.
    \]
    Then,
    \[
    T_\alpha(x) = \left| \sum_{j=-\lfloor (p-1)/2 \rfloor}^{p-1 - \lfloor (p-1)/2 \rfloor} \tilde{\beta}_j e^{\rmi jx} \right|^2,
    \quad
    T_\gamma(x) = \left| \sum_{j=-\lfloor (n-1)/2 \rfloor}^{n-1 - \lfloor (n-1)/2 \rfloor} \tilde{\rho}_j e^{\rmi jx} \right|^2.
    \]

    Now define
    \begin{align*}
        T_\alpha^\ast(x) &= \sum_{j=-p+1}^{p-1} \alpha_j^\ast e^{\rmi jx} := \left| \sum_{j=-\lfloor (p-1)/2 \rfloor}^{p-1 - \lfloor (p-1)/2 \rfloor} \tilde{\beta}_j^\downarrow e^{\rmi jx} \right|^2, \\
        T_\gamma^\ast(x) &= \sum_{j=-n+1}^{n-1} \gamma_j^\ast e^{\rmi jx} := \left| \sum_{j=-\lfloor (n-1)/2 \rfloor}^{n-1 - \lfloor (n-1)/2 \rfloor} \tilde{\rho}_j^\downarrow e^{\rmi jx} \right|^2.
    \end{align*}

    Observe that
    \[
    \alpha_0^\ast = \sum_j (\tilde{\beta}_j^\downarrow)^2 = \sum_j \beta_j^2 = \alpha_0, \quad \text{and similarly} \quad \gamma_0^\ast = \gamma_0.
    \]
    Furthermore,
    \begin{equation}\label{first norm domination}
        \sum_{j=-\infty}^\infty \alpha_j \gamma_j \leq \sum_{j=-\infty}^\infty \alpha_j^\ast \gamma_j^\ast.
    \end{equation}

    Indeed, by construction,
    \[
    \alpha_j^\ast = \sum_{i \in \mathbb{Z}} \tilde{\beta}_i^\downarrow \tilde{\beta}_{i+j}^\downarrow,
    \quad
    \gamma_j^\ast = \sum_{i \in \mathbb{Z}} \tilde{\rho}_i^\downarrow \tilde{\rho}_{i+j}^\downarrow.
    \]
    Hence, by Lemma~\ref{lem: belov 5.6}, we have
    \[
    \{\alpha_j\}_{j=0}^\infty \prec \{\alpha_j^\ast\}_{j=0}^\infty, \quad
    \{\gamma_j\}_{j=0}^\infty \prec \{\gamma_j^\ast\}_{j=0}^\infty,
    \]
    and then by Lemma~\ref{lem: belov 5.7},
    \[
    \{\alpha_j \gamma_j\}_{j=0}^\infty \prec \{\alpha_j^\ast \gamma_j^\ast\}_{j=0}^\infty,
    \]
    which yields \eqref{first norm domination}.

    Since the right-hand side of \eqref{first norm domination} does not decrease the value of the objective, $T_\alpha^\ast(x)$ and $T_\gamma^\ast(x)$ are also extremal polynomials for problem \eqref{recall Ipn}. Redefine $\alpha_j := \alpha_j^\ast$ and $\gamma_j := \gamma_j^\ast$. These redefined sequences are still extremal, and by Lemma~\ref{lem: belov 5.6}, they belong to $S^\downarrow$. This implies the desired monotonicity:
    \[
    \alpha_0 \geq \dots \geq \alpha_{p-1} \geq 0, \quad \gamma_0 \geq \dots \geq \gamma_{n-1} \geq 0.
    \]
\end{proof}

This completes our proof of Lemma \ref{lem: theorem 5.1}.

\subsection{Proof of Lemma \ref{lem: comparison with limiting} (comparison of $\mb{P}_{k,l}$ with $\Pi(k,l)$)}
For $k=l$, the statement of the lemma is trivial. For $k\neq l$, let $x_N:=\pi(k-l)/N$. Then,
\begin{eqnarray*}
(\mb{P}_{r, N})_{kl}-\Pi_{r/N} (k,l)&=&\frac{1}{N}\frac{1-e^{-2\rmi r x_N}}{1-e^{-2\rmi x_N}}-\frac{1}{N}\frac{1-e^{-2\rmi  rx_N}}{2\rmi x_N}\\
&=&\frac{1-e^{-2\rmi r x_N}}{2N}\left[1+\frac{\rmi}{x_N}\left(1-x_N\cot x_N\right)\right].
\end{eqnarray*}
On the other hand, for any $x\neq 0$ from the interval $(-\pi/2,\pi/2)$, 
\begin{equation*}
0< 1-x\cot x< x^2/2.
\end{equation*}
For $x\in(0,\pi/2)$, the first of these inequalities follows from the fact that $\tan x>x$, and the second one  from the estimate $1-x\cot x<1-\cos x<x^2/2$. For $x\in(-\pi/2,0)$ the inequalities hold because $1-x\cot x$ and $x^2/2$ are even functions.

Therefore, for $|x_N|<\pi/2$, we have:
\[
\left|(\mb{P}_{r, N})_{kl}-\Pi_{r/N} (k,l)\right|\leq \frac{1}{N}(1+|x_N|/2)<\frac{2}{N}.
\]
Noting that $|k-l|<N/2$ implies $|x_N|<\pi/2$ completes the proof.

\subsection{Proof of the final inequality in \eqref{tolerance} (the bound on $|K_{p/N,n/N}-K_{\alpha_t,\beta_t}|$)}
Since $T>3/\tau^3$,
\[
1-\sqrt{1-\frac{1}{T\tau}}\leq 1-\sqrt{1-\frac{\tau^2}{3}}\leq \frac{\tau^{3/2}}{4}.
\]
To see the validity of the last inequality, it is sufficient to verify that
\[
(1-\tau^{3/2}/4)^2\leq 1-\tau^2/3.
\]
Taking the square and rearranging, we obtain an equivalent inequality
\[
-\tau^{3/2}/2+\tau^3/16+\tau^2/3\leq 0.
\]
But
\begin{eqnarray*}
-\tau^{3/2}/2+\tau^3/16+\tau^2/3&=&\tau^{3/2}(-1/2+\tau^{3/2}/16+\tau^{1/2}/3)\\
&\leq& \tau^{3/2}(-1/2+1/16+1/3)\\
&=&-(5/48)\tau^{3/2}<0.
\end{eqnarray*}

\subsection{Proof of Lemma \ref{lem: lem16} (approximation of the norm of the Hadamard product of vectors in $\ell^2(\mathbb{C})$ by that of vectors in $\mathbb{C}^k$)}
 We will search for the required $k=k(\tau)$ in the set of odd integers. It will be convenient to think about $\Pi_{r/N}^{[2q+1]}$ as an operator acting in $\ell^2(\mathbb{C})$ with the following representation:
\[
\Pi_{r/N}^{[2q+1]}(i,j)=\begin{cases}
    \Pi_{r/N}(i,j)&\text{for }|i|,|j|\leq q\\
    0&\text{otherwise}
\end{cases}.
\]

    Let $\widetilde{\mb{x}}_1^{(t)},\widetilde{\mb{x}}_2^{(t)}\in\ell^2(\mathbb{C})$ be restrictions of  $\widetilde{\mb{c}}_1^{(t)},\widetilde{\mb{c}}_2^{(t)}\in\ell^2(\mathbb{C})$ to $[-m,m]$, i.e.~$\widetilde{\mb{x}}_i^{(t)}(j)=\widetilde{\mb{c}}_i^{(t)}(j)$ for $|j|\leq m$ and $\widetilde{\mb{x}}_i^{(t)}(j)=0$  for $|j|> m$. Furthermore, let $m=m(\tau)$ be such that
    \[
    \max_{t\in\{1,\dots,T(\tau)\}}\|\widetilde{\mb{x}}_i^{(t)}-\widetilde{\mb{c}}_i^{(t)}\|<\tau^{3/2}/32.
    \]
    Since $\Pi_{r/N}$ is a projection, $\|\Pi_{r/N}\|=1$, and the latter inequality implies that, for each $t\in\{1,\dots,T(\tau)\}$ and all $p/n\in(\tau_{t-1},\tau_t]$,
    \begin{equation}\label{first part of lem16}
    \left|\|\Pi_{p/N}\widetilde{\mathbf{c}}_1^{(t)}\odot \Pi_{n/N}\widetilde{\mathbf{c}}_2^{(t)}\|-\|\Pi_{p/N}\widetilde{\mathbf{x}}_1^{(t)}\odot \Pi_{n/N}\widetilde{\mathbf{x}}_2^{(t)}\|\right|<\tau^{3/2}/32.
    \end{equation}

    Further, for $q\geq m$ we have
    \begin{eqnarray}\label{remainder term}
    \|\Pi_{p/N}\widetilde{\mathbf{x}}_1^{(t)}\odot \Pi_{n/N}\widetilde{\mathbf{x}}_2^{(t)}\|^2-\|\Pi_{p/N}^{[2q+1]}\widetilde{\mathbf{x}}_1^{(t)}\odot\Pi_{n/N}^{[2q+1]}\widetilde{\mathbf{x}}_2^{(t)}\|^2=\sum_{i:|i|>q}|(\Pi_{p/N}\widetilde{\mathbf{x}}_1^{(t)})(i)(\Pi_{n/N}\widetilde{\mathbf{x}}_2^{(t)})(i)|^2.
    \end{eqnarray}
    On the other hand, for $i$ such that $|i|>m$, we have: 
    \begin{equation}\label{Pi v bounds}
    |(\Pi_{p/N}\widetilde{\mathbf{x}}_1^{(t)})(i)|\leq (2m+1)(|i|-m)^{-1},\qquad |(\Pi_{n/N}\widetilde{\mathbf{x}}_2^{(t)})(i)|\leq (2m+1)(|i|-m)^{-1}.
    \end{equation}
    Indeed, these inequalities follow from the bounds $|\widetilde{\mathbf{x}}_1^{(t)}(j)|\leq 1$,  $|\widetilde{\mathbf{x}}_2^{(t)}(j)|\leq 1$, and, for $i\neq j$,  the bounds
    \[
    |\Pi_{p/N}(i,j)|\leq |i-j|^{-1},\quad  |\Pi_{n/N}(i,j)|\leq |i-j|^{-1}.
    \]
    The latter bounds follow directly from the explicit form of the entries of operator $\Pi_{r/N}$ given in \eqref{entries of Pi}.

 Using \eqref{Pi v bounds} in \eqref{remainder term}, we obtain:
\[
\left|\|\Pi_{p/N}\widetilde{\mathbf{x}}_1^{(t)}\odot \Pi_{n/N}\widetilde{\mathbf{x}}_2^{(t)}\|^2-\|\Pi_{p/N}^{[2q+1]}\widetilde{\mathbf{x}}_1^{(t)}\odot\Pi_{n/N}^{[2q+1]}\widetilde{\mathbf{x}}_2^{(t)}\|^2\right| \leq (2m+1)^4 \sum_{i:|i|>q} (|i|-m)^{-4},
\]
which converges to zero as \( q \rightarrow \infty \). Choosing $q$ so that the right hand side of the above inequality is smaller than $\tau^{3/2}/32$, and combining that inequality with \eqref{first part of lem16}, we obtain
\[
\left|\|\Pi_{p/N}\widetilde{\mathbf{c}}_1^{(t)}\odot \Pi_{n/N}\widetilde{\mathbf{c}}_2^{(t)}\|-\|\Pi_{p/N}^{[k]}\mathbf{x}_1^{(t)}\odot \Pi_{n/N}^{[k]}\mathbf{x}_2^{(t)}\|\right|<\tau^{3/2}/16,
\]
where $\Pi_{p/N}^{[k]},\Pi_{n/N}^{[k]}$ are interpreted as $k\times k$ matrices with $k=2q+1$, and $\mathbf{x}_1^{(t)},\mathbf{x}_2^{(t)}\in\mathbb{C}^{k}$ are the finite $[-q,q]$ sections of infinite vectors $\widetilde{\mathbf{x}}_1^{(t)},\widetilde{\mathbf{x}}_2^{(t)}\in\ell^2(\mathbb{C})$. By construction, the above inequality holds for each $t\in\{1,\dots,T(\tau)\}$ and all $p/n\in(\tau_{t-1},\tau_t]$, which finishes the proof of Lemma \ref{lem: lem16}.

\subsection{Proof of the second equality in \eqref{Fourier calculation} (explicit form of a trigonometric sum)}
It suffices to demonstrate that
\begin{equation}\label{Fourier procedure}
\left(\frac{r}{N}\right)^2+2\sum_{j=1}^{\infty}\left(\frac{\sin\frac{\pi j r}{N}}{\pi j}\right)^2=\frac{r}{N}.
\end{equation}
Consider the $2\pi$-periodic even function $g(x)=\pi |x|-x^2$ for $x\in[\pi,\pi]$. Its Fourier series $a_0/2+\sum_{j=1}^\infty a_j\cos(jx)$ has coefficients
\begin{eqnarray*}
\frac{a_0}{2}&=&\frac{1}{2\pi}\int_{-\pi}^{\pi}(\pi|x|-x^2)\mathrm{d}x=\frac{\pi^2}{6},\\
a_j&=&\frac{1}{\pi}\int_{-\pi}^{\pi}(\pi|x|-x^2)\cos(jx)\mathrm{d}x=-2\frac{1+\cos(\pi j)}{j^2}.
\end{eqnarray*}
Since $f(x)$ has bounded right and left derivatives, it satisfies Dini's criterion for point-wise convergence of the Fourier series. In particular, for $c\in [0,1]$,
\begin{eqnarray*}
f(\pi c)&=&\frac{\pi^2}{6}-2\sum_{k=1}^\infty \frac{1+\cos(\pi k)}{k^2}\cos(k\pi c)=\frac{\pi^2}{6}-4\sum_{k\text{ even}} \frac{\cos(\pi c k)}{k^2}\\
&=&\frac{\pi^2}{6}-\sum_{j=1}^\infty \frac{\cos(2\pi c j)}{j^2}=\frac{\pi^2}{6}-\sum_{j=1}^\infty \frac{1-2\sin^2(\pi c j)}{j^2}=2\sum_{j=1}^\infty \frac{\sin^2(\pi c j)}{j^2}.
\end{eqnarray*}
Setting $c=r/N$ yields \eqref{Fourier procedure}. 

\subsection{Proof of Proposition \ref{prop: equivalent of prop 15} (high probability that one of $A_1,\dots,A_m$ occurs)}
This proof closely follows the argument of Proposition 15 in SV13. In places, we reproduce parts of that proof \emph{verbatim} for completeness and clarity.

Let $\psi:\mathbb{R}\rightarrow [0,1]$ be a smooth function such that $\psi(x)=0$ for $x\leq 0$ and $\psi(x)=1$ for $x\geq 1$. For any real number $R>0$ and any complex number $a\in\mathbb{C}$, we have
\[
\mathbf{1}_{\sqrt{\log n}\times B(a,R)}(z)\geq \zeta_{B(a,R)}(z):=\psi(R^2\log n-|z-a\sqrt{\log n}|^2),
\]
where $\mathbf{1}_\Omega(z)$ is the indicator function of set $\Omega\subset\mathbb{C}$. For each $1\leq i\leq m$, define
\begin{equation}\label{Wi definition}
W_i:=\prod_{j=-b+1}^{b+k}\zeta_{B_j}(d_{iq+j}(\mb{a})).
\end{equation}
We have $W_i\leq \mathbf{1}_{A_i}$. By the Paley-Zygmund inequality,
\begin{equation}\label{Paley_Zygmund}
\mathbb{P}\left\{\sum_{i=1}^m\mathbf{1}_{A_i}\geq 1\right\}\geq\mathbb{P}\left\{\sum_{i=1}^m W_i>0\right\}\geq \frac{(\mathbb{E}\sum_{i=1}^m W_i)^2}{\mathbb{E}[(\sum_{i=1}^m W_i)^2]}.
\end{equation}
Incidentally, the numerator and the denominator in the similar inequality were inadvertently swapped in SV13. 

Inequality \eqref{Paley_Zygmund} reduces the proof to verifying that
\begin{equation}\label{need to prove for prop 3a}
\mathbb{E}\left[\left(\sum_{i=1}^m W_i\right)^2\right] = (1 + o(1))\left(\mathbb{E}\sum_{i=1}^m W_i\right)^2.
\end{equation}
Following SV13, we first establish \eqref{need to prove for prop 3a} in the special case where all $a_j$, $j \in \mathbb{Z}$, are replaced by Gaussian random variables. We then extend the result to the general case using a version of the invariance principle from \cite{chatterjee06}, as formulated in Lemma 17 of SV13.

\paragraph{Gaussian case.}
Let $\{G_j\}_{j\in\mathbb{Z}}$ be a sequence of i.i.d.~standard Gaussian random variables, independent of $\{a_j\}_{j\in\mathbb{Z}}$. Define $d_j(G)$ analogously to $d_j(\mb{a})$ (see \eqref{def of d again}), but with $\mb{a} = (a_0, \dots, a_{N-1})$ replaced by $G = (G_0, \dots, G_{N-1})$. 

Let $W_i^G$ denote the random variable defined in the same way as $W_i$, except that all instances of $d_{iq+j}(\mb{a})$ are replaced by $d_{iq+j}(G)$. Then we have (cf. \eqref{Wi definition})
\begin{equation}\label{EWG}
\mathbb{E} W_i^G=\prod_{j=-b+1}^{k+b}\mathbb{E}\zeta_{B_j}(d_{iq+j}(G)).
\end{equation}
This interchange of product and expectation is valid because all indices $iq+j$ in the product lie strictly between $1$ and $N/2$, ensuring that the corresponding random variables $d_{iq+j}$ are independent.

To evaluate $\mathbb{E}\zeta_{B_j}(d_{iq+j}(G))$, recall from Section \ref{sec: intro} that the variables $\{d_t(G)\}_{0<t<N/2}$ are i.i.d. $\mathcal{N}_{\mathbb{C}}(0,1)$. As a result, the squared magnitudes $\{|d_t(G)|^2\}_{0<t<N/2}$ are i.i.d.~$\operatorname{Exp}(1)$ random variables, so that for any $0<t<N/2$,
\begin{equation}\label{explicit probability}
\mathbb{P}(|d_t(G)|^2>x)=e^{-x},\quad x \geq 0.
\end{equation}

Using \eqref{explicit probability}, we obtain
\begin{equation}\label{d in epsilon a}
    \mathbb{E}\zeta_{B(0,\epsilon_n)}(d_t(G))=\mathbb{E}\psi(\epsilon_n^2\log n-|d_t(G)|^2)\geq \mathbb{P}( |d_t(G)|^2\leq \epsilon_n^2 \log n-1)=1-O(n^{-\epsilon_n^2 }),
\end{equation}
and
\begin{eqnarray}\notag
\mathbb{E}\zeta_{B(u_j,\eta)}(d_{iq+j}(G))&\geq& \mathbb{P}(|d_{iq+j}(G)-u_j\sqrt{\log n}|^2\leq \eta^2\log n-1)\\\label{k balls bound}
&\geq&\mathbb{P}(|d_{iq+j}(G)-u_j\sqrt{\log n}|^2\leq (\eta/2)^2\log n)
\end{eqnarray}
for all sufficiently large $n$.

To bound the latter probability from below, note that the distribution of $d_{iq+j}(G)$ is rotationally symmetrical. Therefore, for any $\theta\in\mathbb{R}$, we have
\begin{eqnarray}\notag
\mathbb{P}(|d_{iq+j}(G)-u_j\sqrt{\log n}|^2\leq (\eta/2)^2\log n)&=&\mathbb{P}(|d_{iq+j}(G)-e^{\rmi\theta}u_j\sqrt{\log n}|^2\leq (\eta/2)^2\log n)\\\label{spherical}
&=&\mathbb{P}\left(d_{iq+j}(G)\in B(e^{\rmi\theta}u_j\sqrt{\log n}, (\eta/2)\sqrt{\log n})\right).
\end{eqnarray}
Now define $\theta_s=\eta s/2$, for $s=1,\dots,\lceil 4\pi/\eta\rceil$, and consider the union of balls 
\[
\mb{B}=\cup_{s=1,\dots,\lceil 4\pi/\eta\rceil}B(e^{\rmi\theta_s}u_j\sqrt{\log n},(\eta/2)\sqrt{\log n}).
\]
This set $\mathbf{B}$ contains the annular region:
\[
\mb{A}:=\{z:(|u_j|-\eta/4)\sqrt{\log n}\leq |z|\leq (|u_j|+\eta/4)\sqrt{\log n}\}.
\]
Indeed, any $z\in\mb{A}$ lies within distance $(\eta/4)\sqrt{\log n}$ from the circle with radius $|u_j|\sqrt{\log n}$. Moreover, since the angular spacing between the centers $e^{\rmi\theta_s}u_j\sqrt{\log n}$ is at most $\eta/2$, the arc length between them is at most $(\eta/2)|u_j|\sqrt{\log n}$. Hence, any such $z$ satisfies 
\[
\min_{s=1,\dots,\lceil 4\pi/\eta\rceil}|z-e^{\rmi\theta_s}u_j\sqrt{\log n}|\leq (\eta/4) \sqrt{\log n}+(\eta/4)|u_j|\sqrt{\log n}\leq (\eta/2)\sqrt{\log n},
\]
and therefore $z\in\mb{B}$, as claimed.

Since $\mb{A}\subseteq \mb{B}$, we have
\begin{eqnarray*}
\mathbb{P}\left(d_{iq+j}(G)\in \mb{A}\right)&\leq&\mathbb{P}\left(d_{iq+j}(G)\in \mb{B}\right)\\
&\leq&\sum_{s=1}^{\lceil 4\pi/\eta\rceil}\mathbb{P}\left(d_{iq+j}(G)\in B(e^{\rmi\theta_s}u_j\sqrt{\log n},(\eta/2)\sqrt{\log n})\right)\\
&=&\lceil 4\pi/\eta\rceil \mathbb{P}(|d_{iq+j}(G)-u_j\sqrt{\log n}|^2\leq (\eta/2)^2\log n),
\end{eqnarray*}
where for the last equality we used \eqref{spherical}. 

Combining this with \eqref{k balls bound}, we obtain
\begin{eqnarray}\notag
\mathbb{E}\zeta_{B(u_j,\eta)}(d_{iq+j}(G))&\geq& \frac{1}{\lceil 4\pi/\eta\rceil}\mathbb{P}\left(d_{iq+j}(G)\in \mb{A}\right)\\\notag
&=&\frac{1}{\lceil 4\pi/\eta\rceil}\mathbb{P}\left((|u_j|-\eta/4)\sqrt{\log n}\leq|d_{iq+j}(G)|\leq (|u_j|+\eta/4)\sqrt{\log n}\right)\\ \label{k balls bound 1}
&=&\frac{n^{-(|u_j|-\eta/4)^2}-n^{-(|u_j|+\eta/4)^2}}{\lceil 4\pi/\eta\rceil}=O(n^{-(u_j')^2}),
\end{eqnarray}
where $u_j':=|u_j|-\eta/4$.

Using the estimates \eqref{k balls bound 1} and \eqref{d in epsilon a} in \eqref{EWG}, we obtain 
\begin{eqnarray}\label{EWG a} 
\mathbb{E} W_i^G=\prod_{j=-b+1}^{k+b}\zeta_{B_j}(d_{ir+j}(G))\geq
    (1-O(n^{-\epsilon_n^2 }))^{2b}O(n^{-(u_1')^2-\dots-(u_k')^2}).
\end{eqnarray}
Note that  $(u_1')^2+\dots+(u_k')^2<1$, and that
\[
(1-O(n^{-\epsilon_n^2}))^{2b}=(1-O(e^{\log^{1-\alpha}n}))^{24\lceil \log N\rceil^4}=1+o(1).
\]
Therefore,
\[
\sum_{i=1}^m\mathbb{E} W_i^G=m\mathbb{E} W_i^G=\frac{N}{400\lceil\log N\rceil^4}O(n^{-(u_1')^2-\dots-(u_k')^2})\rightarrow \infty
\]
as $n\rightarrow\infty$.
Further, as in (22) of SV13,
\begin{equation}\label{22 equiv a}
\sum_{i=1}^m\operatorname{Var} W_i^G\leq \sum_{i=1}^m\mathbb{E}[(W_i^G)^2]\leq \sum_{i=1}^m \mathbb{E}W_i^G=o\left(\left(\sum_{i=1}^m\mathbb{E} W_i^G\right)^2\right).
\end{equation}

Now observe that for any $1\leq i<i'\leq m$, 
\begin{align*}
(i'q - b + 1) - (iq + k + b) &= (i' - i)q + 1 - 2b - k \\
&\geq 100\lceil\log N\rceil^4 - 24\lceil\log N\rceil^4 - k > 0
\end{align*}
for all sufficiently large $n$. Therefore, the index sets $\{iq-b+1,\dots,iq+k+b\}$ and $\{i'q-b+1,\dots,i'q+k+b\}$ do are disjoint. It follows that $W_i^G$ and $W_{i'}^G$ are independent, and hence 
\[
\mathbb{E}[W_i^GW_{i'}^G]=\mathbb{E}[W_i^G]\mathbb{E}[W_{i'}^G].
\]
Combining this identity with the bound \eqref{22 equiv a}, we obtain
\begin{eqnarray*}
\mathbb{E}\left[\left(\sum_{i=1}^m W_i^G\right)^2\right]&=&\sum_{i=1}^m\mathbb{E}[(W_i^G)^2]+\sum_{i\neq i'}\mathbb{E}[W_i^G]\mathbb{E}[W_{i'}^G]\\
&=&\sum_{i=1}^m\operatorname{Var}W_i^G+\left(\sum_{i=1}^m\mathbb{E}W_i^G\right)^2=(1+o(1))\left(\sum_{i=1}^m\mathbb{E}W_i^G\right)^2.
\end{eqnarray*}
This establishes \eqref{need to prove for prop 3a} in the Gaussian case, and thus completes the proof of Proposition~\ref{prop: equivalent of prop 15} under the assumption that the $a_k$ are Gaussian.

\paragraph{General case.} We now use a generalization of \cite{chatterjee06}'s invariance principle (Lemma 17 in SV13) to establish  Proposition \ref{prop: equivalent of prop 15} in full generality. For the sake of completeness, we restate SV13’s Lemma 17 here.

Let $\mb{X} = (X_1, \dots, X_R )$ and $\mb{Y} = (Y_1, \dots, Y_R )$ be independent random vectors, with each $X_i$ and $Y_i$ taking values in a common open interval $I\subseteq\mathbb{R}$. Suppose further that
\[
\mathbb{E}X_i=\mathbb{E}Y_i,\quad \mathbb{E}X_i^2=\mathbb{E}Y_i^2<\infty\quad\text{for all }i=1,\dots,R.
\]

\begin{lemma}[Lemma 17 of \cite{sen13}]\label{lem: SV13 lem17 a}
Let $f=(f_1,\dots,f_M): I^R\rightarrow \mathbb{R}^M$ be thrice continuously differentiable. If we set $\mathbf{U}=f(\mb{X})$ and $\mb{V}=f(\mb{Y})$, then for any thrice continuously differentiable $g:\mathbb{R}^M\rightarrow \mathbb{R}$,
\[
\left|\mathbb{E}[g(\mb{U})]-\mathbb{E}[g(\mb{V})]\right|\leq\sum_{\tau=1}^R\mathbb{E}[S_\tau]+\sum_{\tau=1}^R\mathbb{E}[T_\tau],
\]
where
\begin{eqnarray*}
    S_\tau&:=&\frac{1}{6}|X_\tau|^3\times \sup_{x\in[0\wedge X_\tau,0\vee X_\tau]}|h_\tau(X_1,\dots,X_{\tau-1},x,Y_{\tau+1},\dots,Y_R)|,\\
    T_\tau&:=&\frac{1}{6}|Y_\tau|^3\times \sup_{y\in[0\wedge Y_\tau,0\vee Y_\tau]}|h_\tau(X_1,\dots,X_{\tau-1},y,Y_{\tau+1},\dots,Y_R)|,\\    h_\tau(\mb{x})&:=&\sum_{\ell,s,q=1}^M\partial_\ell\partial_s\partial_q g(f(\mb{x}))\partial_\tau f_\ell(\mathbf{x})\partial_\tau f_s(\mathbf{x})\partial_\tau  f_q(\mb{x})\\
    &&+3\sum_{\ell,s=1}^M\partial_\ell\partial_s g(f(\mb{x}))\partial_\tau^2 f_\ell(\mathbf{x})\partial_\tau f_s(\mb{x})+\sum_{\ell=1}^M\partial_\ell g(f(\mathbf{x}))\partial_\tau^3f_\ell(\mathbf{x}).
\end{eqnarray*}
\end{lemma}

Adapting the lemma to our setting, we make the following identifications:
\begin{itemize}
    \item Let $R=N$, $M=2k+4b$;
    \item Let $\mb{X}=\mb{a}$, $\mb{Y}=G$;
    \item Define the functions $f(\mb{X})$ and $f(\mb{Y})$ as
    \begin{align*}
f(\mb{X}) &= \big( \Re(d_{iq - b + 1}(\mb{a})), \Im(d_{iq - b + 1}(\mb{a})), \dots, \Re(d_{iq + k + b}(\mb{a})),, \Im(d_{iq + k + b}(\mb{a})) \big), \\
f(\mb{Y}) &= \big( \Re(d_{iq - b + 1}(G)), \Im(d_{iq - b + 1}(G)), \dots, \Re(d_{iq + k + b}(G)), \Im(d_{iq + k + b}(G)) \big);
\end{align*}
\item Finally, define the function $g:\mathbb{R}^M\rightarrow\mathbb{R}$ by
\[
g(\mb{z})=\prod_{j=1}^{k+2b}\zeta_{B_{j-b}}(z_{2j-1}+\rmi z_{2j}),\quad\text{where }\mb{z}=(z_1,\dots,z_{2k+4b}).
\]
\end{itemize}
With these identifications, we have
\[
\left|\mathbb{E}[g(\mb{U})]-\mathbb{E}[g(\mb{V})]\right|=\left|\mathbb{E}W_i-\mathbb{E}W_i^G\right|,
\]
and Lemma \ref{lem: SV13 lem17 a} implies that
\begin{equation}\label{implication of Lemma lem17}
|\mathbb{E}[W_i]-\mathbb{E}[W_i^G]|\leq \sum_{\tau=1}^N\mathbb{E}[S_\tau]+\sum_{\tau=1}^{N}\mathbb{E}[T_\tau].
\end{equation}

First, we establish an upper bound on $\sum_{\tau=1}^N\mathbb{E}[S_\tau]$.  Recall that for any $\mb{x}=(x_0,\dots,x_{N-1})$, the real and imaginary parts of $d_j(\mb{x})$ are given by:
\[
\Re( d_j(\mb{x}))=\frac{1}{\sqrt{N}}\sum_{k=0}^{N-1}\cos\left(\frac{2\pi\rmi}{N}kj\right)x_k,\qquad \Im( d_j(\mb{x}))=\frac{1}{\sqrt{N}}\sum_{k=0}^{N-1}\sin\left(\frac{2\pi\rmi}{N}kj\right)x_k.
\]
By definition of $f$, for each $\tau=1,\dots,N$ and $t=1,\dots,M$, we have
\begin{equation}\label{bound on partial der 1 a}
    |\partial_\tau f_t|\leq 1/\sqrt{N},\qquad \partial_\tau^2 f_t=0,\qquad \partial_\tau^3 f_t=0.
\end{equation}
Next, observe that for all $\ell,s,q=1,\dots,M$, the third derivatives of $g$ are uniformly bounded: 
\[
\|\partial_\ell\partial_s\partial_q g\|_\infty=O(1).
\]
Moreover, $\partial_\ell\partial_s\partial_q g(\mathbf{z})\neq 0$ only if 
\[
z_{2j-1}+\rmi z_{2j}\in \sqrt{\log n}B_{j-b}
\]
for all $j=1,\dots,k+2b$.

For any $x\in\mathbb{R}$, define the random vector
\[
\mb{Z}^{(j)}(x):=(a_0,\dots,a_{j-2},x,G_{j},\dots,G_{N-1}).
\]
Applying Lemma~\ref{lem: SV13 lem17 a}, we obtain the bound, for some constant $C>0$,
\begin{eqnarray}\label{ES array a}
\mathbb{E}S_\tau&\leq& \frac{C M^3}{N^{3/2}}\mathbb{E}\left[|a_{\tau-1}|^3 \sup_{x:|x|\leq |a_{\tau-1}|}\mathbf{1}\{d_{iq+s}(\mb{Z}^{(\tau)}(x))\in\sqrt{\log n} B_s,1\leq s\leq k\}\right]
.
\end{eqnarray}
This corresponds to the first line in the display defining $h_\tau(\mb{x})$ in Lemma \ref{lem: SV13 lem17 a}. The term $N^{3/2}$ in the denominator  comes from the bound 
\[
|\partial_\ell\partial_s\partial_q g(f(\mb{x}))\partial_\tau f_\ell(\mathbf{x})\partial_\tau f_s(\mathbf{x})\partial_\tau  f_q(\mb{x})|\leq  \partial_\ell\partial_s\partial_q g(f(\mb{x}))N^{-3/2}\leq O(1) N^{-3/2},
\]
where the first of the inequalities follows from \eqref{bound on partial der 1 a}.

By definition,
\[
\left|d_{iq+s}(\mb{Z}^{(\tau)}(x))-d_{iq+s}(\mb{Z}^{(\tau)}(0))\right|=\left|\frac{1}{\sqrt{N}}e^{\frac{2\pi\rmi}{N}(\tau-1)(iq+s)}x\right|=\frac{|x|}{\sqrt{N}}.
\]
Therefore,
\begin{eqnarray}\notag
&&\sup_{x:|x|\leq |a_{\tau-1}|}\mathbf{1}\{d_{iq+s}(\mb{Z}^{(\tau)}(x))\in\sqrt{\log n} B_s,1\leq s\leq k\}\\\notag
&\leq&\sup_{x:|x|\leq |a_{\tau-1}|}\mathbf{1}\{|d_{iq+s}(\mb{Z}^{(\tau)}(x))|\geq \sqrt{\log n} (|u_s|-\eta), 1\leq s\leq k\}\\\notag
&\leq& \mathbf{1}\{|d_{iq+s}(\mb{Z}^{(\tau)}(0))|>\sqrt{\log n}(|u_s|-\eta)-|a_{\tau-1}|/\sqrt{N}, 1\leq s\leq k\}\\\label{first part of Stau}
&\leq&\mathbf{1}\{|d_{iq+s}(\mb{Z}^{(\tau)}(0))|>\sqrt{\log n}(|u_s|-\eta)-\sqrt{n/N}n^{1/\gamma-1/2}, 1\leq s\leq k\},
\end{eqnarray}
where the latter inequality follows from the bound $|a_{\tau-1}|\leq n^{1/\gamma}$. 

Note that the random variable on the right hand side of the  inequality \eqref{first part of Stau} does not depend on $a_{\tau-1}$. Therefore, using \eqref{first part of Stau} in \eqref{ES array a}, we obtain: 
\[
\mathbb{E}S_\tau\leq \frac{C M^3}{N^{3/2}}\mathbb{E}|a_{\tau-1}|^3 \mathbb{P}\left\{|d_{iq+s}(\mb{Z}^{(\tau)}(0))|>\sqrt{\log n}(|u_s|-\eta)-\sqrt{n/N}n^{1/\gamma-1/2}, 1\leq s\leq k\right\}
.
\]
Further, since $\mathbb{E}a_{\tau-1}^2=1$ and $|a_{\tau-1}|\leq n^{1/\gamma}$, we have
\[
\mathbb{E}|a_{\tau-1}|^3\leq n^{1/\gamma}.
\]
Therefore, recalling that $N=O(n)$ and changing the value of the constant $C$ accordingly, we obtain:
\begin{equation}\label{bound on ES a}
\mathbb{E}S_\tau\leq \frac{CM^3}{n^{3/2-1/\gamma}}\mathbb{P}\left\{|d_{iq+s}(\mb{Z}^{(\tau)}(0))|>\sqrt{\log n}(|u_s|-\eta)-\sqrt{n/N}n^{1/\gamma-1/2}, 1\leq s\leq k\right\}.
\end{equation}

Let us derive an upper bound for the latter probability.  Define
\[
\hat{\mb{Z}}^{(j)}(x):=(a_0,\dots,a_{j-2},x,G_{j}\mb{1}_{\{|G_{j}|\leq n^{1/\gamma}\}},\dots,G_{N-1}\mb{1}_{\{|G_{N-1}|\leq n^{1/\gamma}\}}).
\]
Then we have:
\begin{eqnarray*}
&&\mathbb{P}\left\{|d_{iq+s}(\mb{Z}^{(\tau)}(0))|>\sqrt{\log n}(|u_s|-\eta)-\sqrt{n/N}n^{1/\gamma-1/2}, 1\leq s\leq k\right\}\\
&\leq&\mathbb{P}\left\{|d_{iq+s}(\hat{\mb{Z}}^{(\tau)}(0))|>\sqrt{\log n}(|u_s|-\eta)-\sqrt{n/N}n^{1/\gamma-1/2}, 1\leq s\leq k\right\}\\
&&+\mathbb{P}\left\{|G_\ell|>n^{1/\gamma}\text{ for some } \ell=0,\dots,N-1\right\}.
\end{eqnarray*}
Observe that
\begin{equation}\label{Gaussian tails}
\mathbb{P}\left\{|G_\ell|>n^{1/\gamma}\text{ for some } \ell=0,\dots,N-1\right\}\leq O\left(n^{1-1/\gamma}\exp\left(-\frac{1}{2}n^{2/\gamma}\right)\right).
\end{equation}
Further, by Lemma \ref{lem: refined equiv of lemma 11}, for any small $\delta>0$, we have:
\[
\mathbb{P}\left\{|d_{iq+s}(\hat{\mb{Z}}^{(\tau)}(0))|>\sqrt{\log n}(|u_s|-\eta)-\sqrt{n/N}n^{1/\gamma-1/2}, 1\leq s\leq k\right\}\leq O\left(n^{-(|u_1|-\eta)^2-\dots -(|u_k|-\eta)^2+\delta}\right).
\]
Recalling the definition $u_j':=|u_j|-\eta/4$, we obtain:
\begin{eqnarray*}
(|u_1|-\eta)^2+\dots +(|u_k|^2-\eta)^2&=&(u_1'-3\eta/4)^2+\dots +(u_k'-3\eta/4)^2\\
&=&\sum_{s=1}^k(u_s')^2-\frac{3}{2}\sum_{s=1}^k u_s'\eta+\frac{9k}{16}\eta^2.
\end{eqnarray*}
Note that $u'_j<1$ for all $j=1,\dots,k$. Choosing $\delta>0$ so small that
\[
\frac{3}{2}\sum_{s=1}^k u_s'\eta+\delta<2k\eta,
\]
we get the following upper bound.
\[
\mathbb{P}\left\{|d_{iq+s}(\hat{\mb{Z}}^{(\tau)}(0))|>\sqrt{\log n}(|u_s|-\eta)-\sqrt{n/N}n^{1/\gamma-1/2}, 1\leq s\leq k\right\}\leq O\left(n^{-((u_1')^2+\dots+(u_k')^2)+2k\eta}\right).
\]
Combining this upper bound with \eqref{Gaussian tails}, we obtain:
\begin{eqnarray*}
&&\mathbb{P}\left\{|d_{iq+s}(\mb{Z}^{(\tau)}(0))|>\sqrt{\log n}(|u_s|-\eta)-\sqrt{n/N}n^{1/\gamma-1/2}, 1\leq s\leq k\right\}\\&\leq& O\left(n^{-((u_1')^2+\dots+(u_k')^2)+2k\eta}\right)+  O\left(n^{1-1/\gamma}\exp\left(-\frac{1}{2}n^{2/\gamma}\right)\right).
\end{eqnarray*}

Using the latter inequality in \eqref{bound on ES a}, we obtain the following upper bound on $\sum_{\tau=1}^N\mathbb{E}[S_\tau]$:
\begin{eqnarray}\notag
\sum_{\tau=1}^N\mathbb{E}[S_\tau]&\leq& \frac{M^3}{n^{1/2-1/\gamma}}\left(O\left(n^{-((u_1')^2+\dots+(u_k')^2)+2k\eta}\right)+O\left(n^{1-1/\gamma}\exp\left(-\frac{1}{2}n^{2/\gamma}\right)\right)\right)\\ \label{bound on sum ES a}
&\leq&\frac{\log^{12} n}{n^{1/2-1/\gamma}}O\left(n^{-((u_1')^2+\dots+(u_k')^2)+2k\eta}\right),
\end{eqnarray}
uniformly in $i$ (recall that $\sum_{\tau=1}^N\mathbb{E}[S_\tau]$ depends on $i$ because it involves random variables $d_{iq+j}$). 

Our next goal is to derive a similar bound on $\sum_{\tau=1}^N\mathbb{E}[T_\tau]$. Similarly to \eqref{ES array a}, we get
\begin{eqnarray}\label{ET array a}
\mathbb{E}T_\tau\leq \frac{C M^3 }{N^{3/2}}\mathbb{E}\left[|G_{\tau-1}|^3 \sup_{x:|x|\leq |G_{\tau-1}|}\mathbf{1}\{d_{iq+s}(\mb{Z}^{(\tau)}(x))\in\sqrt{\log n} B_s, 1\leq s\leq k\}\right].
\end{eqnarray}
Note that the expected value on the right hand side of \eqref{ET array a} is no larger than 
\begin{eqnarray*}
\mathbb{E}\left[|G_{\tau-1}|^3 \sup_{x:|x|\leq n^{1/\gamma}}\mathbf{1}\{d_{iq+s}(\mb{Z}^{(\tau)}(x))\in\sqrt{\log n} B_s, 1\leq s\leq k\}\right]+\mathbb{E}[|G_{\tau-1}|^3\mathbf{1}_{|G_{\tau-1}|>n^{1/\gamma}}].
\end{eqnarray*}
The first term in the latter sum can be analyzed very similarly to the above analysis of\linebreak  $\mathbb{E}\left[|a_{\tau-1}|^3 \sup_{x:|x|\leq |a_{\tau-1}|}\mathbf{1}\{d_{iq+s}(\mb{Z}^{(\tau)}(x))\in\sqrt{\log n} B_s, 1\leq s\leq k\}\right]$. For the second term, we have
\[
\mathbb{E}[|G_{\tau-1}|^3\mathbf{1}_{|G_{\tau-1}|>n^{1/\gamma}}]=\frac{1}{\sqrt{2\pi}}\int_{|x|>n^{1/\gamma}}|x|^3e^{-x^2/2}\mathrm{d}x=\sqrt{\frac{8}{\pi}}\int_{y>\frac{1}{2}n^{2/\gamma}}ye^{-y}\mathrm{d}y=\sqrt{\frac{8}{\pi}}\Gamma\left(0,\frac{1}{2}n^{2/\gamma}\right),
\]
where $\Gamma(s,x)$ is the incomplete Gamma function. As is well known (see e.g.~\cite{olver97}, p.~66), for $s=0$, $\Gamma(0,x)=(1+o(1))x^{-1}e^{-x}$ as $x\rightarrow\infty$. Hence,
\[
\mathbb{E}[|G_{\tau-1}|^3\mathbf{1}_{|G_{\tau-1}|>n^{1/\gamma}}]=O(1)n^{-2/\gamma}e^{-\frac{1}{2}n^{2/\gamma}}.
\]
Now a similar argument to the one used above to bound $\sum_{\tau=1}^N \mathbb{E}[S_\tau]$ yields the same bound for $\sum_{\tau=1}^N \mathbb{E}[T_\tau]$:
\[
\sum_{\tau=1}^N\mathbb{E}[T_\tau]\leq \frac{\log^{12} n}{n^{1/2-1/\gamma}}O\left(n^{-((u_1')^2+\dots+(u_k')^2)+2k\eta}\right).
\]

Using the obtained bounds $\sum_{\tau=1}^N \mathbb{E}[S_\tau]$ $\sum_{\tau=1}^N \mathbb{E}[T_\tau]$ in \eqref{implication of Lemma lem17}, we obtain:
\begin{eqnarray}\notag
\mathbb{E}[W_i]&=&\mathbb{E}[W_i^G]+\log^{12}n O\left(n^{1/\gamma-1/2-((u_1')^2+\dots +(u_k')^2)+2k\eta}\right)\\\label{EW and EWG are close a}
&=&\mathbb{E}[W_i^G]+\log^{12}n O\left(n^{\frac{1}{3}(1/\gamma-1/2)-((u_1')^2+\dots +(u_k')^2)}\right)=(1+o(1))\mathbb{E}[W_i^G],
\end{eqnarray}
uniformly in $1\leq i\leq m$. Here, to obtain the second to the last equality, we used \eqref{new equation}. For the last equality, we used  \eqref{EWG}.

Similar arguments lead to the identity
\begin{equation}\label{EWW and EWWG are close}
\mathbb{E}[W_iW_{i'}]=(1+o(1))\mathbb{E}[W_i^GW_{i'}^G]=(1+o(1))\mathbb{E}[W_i^G]\mathbb{E}[W_{i'}^G],
\end{equation}
uniformly in $1\leq i\neq i'\leq m$. 

The asymptotic equivalence \eqref{EW and EWG are close a}, together with the divergence of $\sum_{i=1}^m\mathbb{E}[W_i^G]$, implies that $\sum_{i=1}^m\mathbb{E}[W_i]$ diverges to infinity too. Hence, we have the following analogue of \eqref{22 equiv a} for $W_i$: 
\begin{equation}\label{32 equiv a}
    \sum_{i=1}^m \operatorname{Var}W_i\leq o\left(\left(\sum_{i=1}^m\mathbb{E}W_i\right)^2\right).
\end{equation}
Further, the later three displays yield
\begin{eqnarray*}
\mathbb{E}\left[\left(\sum_{i=1}^mW_i\right)^2\right]&=&\sum_{i=1}^m\mathbb{E}[W_i^2]+\sum_{1\leq i\neq i'\leq m}\mathbb{E}[W_iW_{i'}]\\
&=&\sum_{i=1}^m \operatorname{Var}W_i+\sum_{i=1}^m(\mathbb{E}[W_i])^2+\sum_{1\leq i\neq i'\leq m}\mathbb{E}[W_iW_{i'}]\\
&=&\sum_{i=1}^m \operatorname{Var}W_i+(1+o(1))\left(\sum_{i=1}^m(\mathbb{E}[W_i^G])^2+\sum_{1\leq i\neq i'\leq m}\mathbb{E}[W_i^G]\mathbb{E}[W_{i'}^G]\right)\\
&=&\sum_{i=1}^m \operatorname{Var}W_i+(1+o(1))\left(\sum_{i=1}^m\mathbb{E}[W_i^G]\right)^2\\
&=&\sum_{i=1}^m \operatorname{Var}W_i+(1+o(1))\left(\sum_{i=1}^m\mathbb{E}[W_i]\right)^2\\
&=&(1+o(1))\left(\sum_{i=1}^m\mathbb{E}[W_i]\right)^2.
\end{eqnarray*}
Hence \eqref{need to prove for prop 3a} holds, which implies that Proposition \ref{prop: equivalent of prop 15} holds in full generality. \qed

\subsection{Proof of Lemma \ref{lem: grid lemma for circulant}}
First, observe that
\[
\sup_{\substack{\mathbf{v},\,\mathbf{w} \in \ell^2 \\ \|\mathbf{v}\| = \|\mathbf{w}\| = 1}}
\|(\Pi_{p/n} \mathbf{v}) \odot (\Pi_{1} \mathbf{w})\|
= \sup_{\substack{\mathbf{v},\,\mathbf{w} \in \ell^2 \\ \|\mathbf{v}\| = \|\mathbf{w}\| = 1}}
\|(\Pi_{p/n} \mathbf{v}) \odot \mathbf{w}\|
= \|\Pi_{p/n}\|_{2 \to \infty}.
\]
From identity~\eqref{Fourier procedure}, we have
\[
\|\Pi_{p/n}\|_{2 \to \infty} = \sqrt{p/n}.
\]
Hence, there exist vectors $\widetilde{\mathbf{c}}_1^{(t)}, \widetilde{\mathbf{c}}_2^{(t)}$ in the unit ball of $\ell^2(\mathbb{C})$ such that
\[
\left| \sqrt{p/n} - \| \Pi_{p/n} \widetilde{\mathbf{c}}_1^{(t)} \odot \Pi_{1} \widetilde{\mathbf{c}}_2^{(t)} \| \right| < \tau^{3/2} / 8.
\]
To establish~\eqref{C grid bound}, it therefore suffices to show that there exist $k=k(\tau)$ and $\mb{u}^{(1)},\dots,\mb{u}^{(T)}$ from the interior of the unit ball in $\mathbb{C}^k$ such that
\[
\|\Pi_{p/n}^{[k]} \operatorname{diag}\{\mathbf{u}^{(t)}\}\| 
> \| \Pi_{p/n} \widetilde{\mathbf{c}}_1^{(t)} \odot \Pi_{1} \widetilde{\mathbf{c}}_2^{(t)} \| - \tau^{3/2} / 2.
\]
for each $t\in\{1,\dots,T\}$ and all $p/n\in(\tau_{t-1},\tau_t]$. This is the circulant analogue of~\eqref{almost final} and can be proven by an almost identical argument, which we omit to avoid repetition.

Inequality \eqref{C first grid interval} follows by an argument nearly identical to that used to establish \eqref{first grid interval}, which finishes the proof.

\subsection{Proof of Lemma \ref{lem: Xj}}
    Using the properties \eqref{SG multiplication} and \eqref{SG independent} of sub-gamma random variables, we obtain $Z_j\in SG(v_j,u_j)$, where
    \begin{eqnarray}\label{X vj}
v_j&=&2(\omega_0^{(j)})^2+2(\omega_{n/2}^{(j)})^2+\sum_{i\in\mathbb{Z}}(\check{\omega}^{(j)}_i)^2,\\\label{X uj}
u_j&=&\max\left\{2\omega_0^{(j)},2\omega_{n/2}^{(j)},\max_{i\in\mathbb{Z}}|\check{\omega}^{(j)}_i|\right\}.
\end{eqnarray}
By definition,
\begin{equation}\label{omega0}
\omega_0^{(j)}=\frac{|(\mb{P}_{p,n})_{0j}|^2}{c^2}\leq\frac{1}{4\log^2 n }\frac{1}{c^2}\quad\text{for all }j\in J_n, 
\end{equation}
where the latter inequality uses \eqref{bound on the entries of P} and the fact that $j^2\geq\log^2 n$ for all $j\in J_n$. Similarly,
\begin{equation}\label{omegai}
\omega_{n/2}^{(j)}=\frac{|(\mb{P}_{p,n})_{n/2,j}|^2}{c^2}\leq\frac{1}{4\log^2 n }\frac{1}{c^2}\quad\text{for all }j\in J_n. 
\end{equation}
Further, for $i=1,\dots,n/2-1$, we have
\[
|\check{\omega}_i^{(j)}|=\left|\frac{|(\mb{P}_{p,n})_{i,j}|^2+|(\mb{P}_{p,n})_{n-i,j}|^2-|\Pi_c(i,j)|^2}{c^2}\right|.
\]
On the other hand, 
\begin{eqnarray*}
\left||(\mb{P}_{p,n})_{i,j}|^2-|\Pi_c(i,j)|^2\right|\leq \left|(\mb{P}_{p,n})_{i,j}-\Pi_c(i,j)\right|\left(|(\mb{P}_{p,n})_{i,j}|+|\Pi_c(i,j)|\right)\leq \frac{4c}{n},
\end{eqnarray*}
where the latter inequality holds for all for $i=1,\dots,n/2-1$ and all $j\in J_n$ by Lemma \ref{lem: comparison with limiting}. Therefore, 
\begin{equation*}
|\check{\omega}_i^{(j)}|\leq \frac{4}{cn}+\frac{|(\mb{P}_{p,n})_{n-i,j}|^2}{c^2}\leq \frac{4}{cn}+\max\left\{\frac{1}{4c^2(n-i-j)^2},\frac{1}{4c^2(i+j)^2}\right\},
\end{equation*}
where the last inequality follows from \eqref{bound on the entries of P}.
For $i\in\{i:i\leq 0\text{ or }i\geq n/2\}$ and all $j\in J_n$, by definition,
\begin{equation*}
    |\check{\omega}_i^{(j)}|\leq \frac{1}{c^2\pi^2(i-j)^2}.
\end{equation*}

The latter two displays imply that, for any $j\in J_n$ and all sufficiently large $n$,
\begin{equation}\label{max omegacheck}
    \max_{i\in\mathbb{Z}}|\check{\omega}^{(j)}_i|\leq \frac{4}{cn}+\frac{1}{4 c^2\log^2 n}+\frac{1}{c^2\pi^2 \log^2 n}\leq \frac{1}{c^2\log^2 n}.
\end{equation}
Moreover, since the minimum possible value of $\min\{n-i-j,i+j\}$ over $i=1,\dots,n/2-1$ and $j\in J_n$ is no smaller than $\lceil \log n\rceil $, we have
\begin{equation}\label{2 omegacheck}
    \sum_{i=1}^{n/2-1}(\check{\omega}^{(j)}_i)^2\leq  \sum_{i=1}^{n/2-1}2\left(\frac{4}{cn}\right)^2+2\sum_{k=\lceil \log n\rceil}^\infty\frac{1}{16 c^4 k^4}.
\end{equation}
Similarly, since the minimum possible value of $|i-j|$ over $i\in\{i:i\leq 0\text{ or }i\geq n/2\}$ and $j\in J_n$ is no smaller than $\lceil \log n\rceil $, we have
\begin{equation}\label{2 omegacheck out}
    \sum_{i\leq 0}(\check{\omega}^{(j)}_i)^2\leq  \sum_{k=\lceil \log n\rceil}^\infty\frac{1}{\pi^4  c^4 k^4}\quad\text{and}\quad \sum_{i\geq n/2}(\check{\omega}^{(j)}_i)^2\leq  \sum_{k=\lceil \log n\rceil}^\infty\frac{1}{\pi^4  c^4 k^4}
\end{equation}
Combining the last two display and using an integral upper bound on the sums of $k^{-4}$, we obtain
\begin{equation}\label{2 omegacheck final}
    \sum_{i\in\mathbb{Z}}(\check{\omega}^{(j)}_i)^2\leq \frac{16}{c^2 n}+\left(\frac{1}{16}+\frac{1}{\pi^4}\right)\frac{2}{3c^4(\log n-1)^3}\leq \frac{1}{8c^4\log^3 n},
\end{equation}
for all sufficiently large $n$.

Using \eqref{omega0}-\eqref{omegai} and \eqref{2 omegacheck final} in \eqref{X vj}, we obtain
\[
v_j\leq \frac{1}{4 c^4 \log^4 n}+\frac{1}{8c^4\log^3 n}\leq \frac{1}{c^4\log^3 n},
\]
for all sufficiently large $n$.
Using \eqref{omega0}-\eqref{omegai} and \eqref{max omegacheck} in \eqref{X uj}, we obtain
\[
u_j\leq \frac{1}{c^2\log^2 n},
\]
which completes the proof.

\bibliography{ddbib}

\end{document}